\documentclass[natbib, final]{svjour3}       

\usepackage{color}
\usepackage{xcolor}
\usepackage{graphicx}
\usepackage{mathrsfs}

\usepackage{bm}
\usepackage{float}
\usepackage{framed}
\makeatletter
\usepackage[utf8]{inputenc}
\usepackage{amsfonts}
\usepackage{amssymb}
\usepackage{amsmath}
\usepackage{listings}
\usepackage{mathtools}
\usepackage{multirow}
\usepackage[flushleft]{threeparttable}
\usepackage{pdfpages}
\usepackage[normalem]{ulem}
\usepackage{comment}
\usepackage{soul}

\usepackage[breaklinks,colorlinks,citecolor=blue,linkcolor=blue]{hyperref}

\smartqed

\newcommand{\ind}{1\hspace{-2.1mm}{1}} 
\newcommand{\defeq}{\vcentcolon=} 
\newcommand{\prob}{\mathbb{P}}
\newcommand{\partder}[2]{\frac{\partial #1}{\partial #2}} 

\usepackage{pifont}

\numberwithin{equation}{section}
\numberwithin{figure}{section}
 \usepackage[nodayofweek]{datetime}

\definecolor{darkgreen}{rgb}{0,0.35,0}

\newcommand{\dif}{\ensuremath{\mathrm{d}}}

\newcommand{\T}{\ensuremath{\mathrm{\scriptscriptstyle T}}}
\usepackage{bm}

\newcommand{\alphab}{\ensuremath{\bm\alpha}}
\newcommand{\betab}{\ensuremath{\bm\beta}}

\usepackage{amsmath}

\newcommand{\dto}{\ensuremath{\overset{\mathrm{d}}{\rightarrow}}}

\usepackage{mathrsfs}

  \usepackage{color}

 \let\oldthebibliography=\thebibliography
 \let\oldendthebibliography=\endthebibliography
 \renewenvironment{thebibliography}[1]{%
   \footnotesize
   \oldthebibliography{#1}%
   \setlength{\parskip}{0mm}%
   \setlength{\itemsep}{0mm}%
 }{\oldendthebibliography}

\journalname{}

\authorrunning{M.~de Carvalho et al.}

\begin{document}
\title{Regression-type analysis for block maxima on block maxima}

\author{\mbox{Miguel de Carvalho \and Gon\c{c}alo dos Reis \and Alina~Kumukova}
}

\institute{M. de Carvalho  
           \at
           School of Mathematics, University of Edinburgh, The King's Buildings, Edinburgh, EH9 3FD, UK \\
           Tel.: +44-131-650-5054 \\
           \email{Miguel.deCarvalho@ed.ac.uk}           
           \and     
           G. dos Reis  \at
            School of Mathematics, University of Edinburgh, The King's Buildings, Edinburgh, EH9 3FD, UK, and \\
            {Centro de Matem\'atica e Aplica\c c$\tilde{\text{o}}$es (CMA), FCT, UNL, Quinta da Torre, 2829--516 Caparica, Portugal}
                      \and 
           A. Kumukova  \at
           Maxwell Institute for Mathematical Sciences School of Mathematics, University of Edinburgh, Edinburgh UK
}

\date{Received: date / Accepted: date \\ \copyright~The Author(s) 2021}

\maketitle

\begin{abstract}
  This paper devises a regression-type model for the situation where both the response and covariates are extreme. 
  The proposed approach is designed for the setting where both the response and covariates are themselves block maxima, and thus contrarily to standard regression methods it takes into account the key fact that the limiting distribution of suitably standardized componentwise maxima is an extreme value copula. An important target in the proposed framework is the regression manifold, which consists of a family of regression lines obeying the latter asymptotic result. To learn about the proposed model from data, we employ a Bernstein polynomial prior on the space of angular densities which leads to an induced prior on the space of regression manifolds. Numerical studies suggest a good performance of the proposed methods, and a finance real-data illustration reveals interesting aspects on the comovements of extreme losses between two leading stock markets. 
\end{abstract}

\footnotesize

\normalsize

\section{Introduction}
Block maxima data---such as annual maxima---are a mainstay of
statistics of extremes. Whereas classical statistical modeling is
mostly concerned with inferences surrounding the bulk of a
distribution, the field of statistics of extremes deals with the
rather challenging situation of conducting inferences about the tail
of a distribution. The behavior of extreme values in large samples is
often mathematically tractable, and this tractability is often used to
build sound statistical methods for modeling risk and extreme values. As an
example of this asymptotic tractability, it is well known that if
$Y_1, \dots, Y_n$ is a random sample with sample maximum
$M_n = \max(Y_1, \dots, Y_n)$ and if there exist sequences
$\{a_n > 0\}$ and $\{b_n\}$ such that $(M_n - b_n) / a_n \dto Z$, then
$Z$ follows a GEV (Generalized Extreme Value) distribution with
location, scale, and shape parameters $\mu \in \mathbb{R}$,
$\sigma > 0$, and $\xi \in \mathbb{R}$ respectively; see, for
instance, \cite[Theorem~3.2.3]{embrechts1997}.  Details on the paradigm
of statistics of extremes can be found in monographs 
\citep[e.g.][]{coles2001, beirlant2004, dehaan2006, resnick2007} as well as review papers \citep[e.g.][]{davison2015}.

\vspace{0.2cm}
In this paper, we devise a regression-type method for the situation
where both the response and the covariates are themselves block
maxima. Here and below, the expression ``regression-type'' is used to
refer to the class of statistical models that relate the conditional
quantiles of a response with covariates via a joint
distribution---rather than by specifying a functional relation between
response and covariate as, for example, in quantile regression
\citep{koenker1978}. An important result in the field of statistics of
extremes---that will be fundamental for our developments---is that the
properly standardized vector of block maxima converges in distribution
to a so-called extreme value copula \citep{gudendorf2010}. Thus, a key
target in the proposed framework is what we will refer below as the
regression manifold, that is, a family of regressions lines that obeys
the latter large sample result.  Our methods thus take on board information
on the dependence structure between the extreme values so to assess
what effects block maxima covariates can have on a block maxima
response. To learn about the proposed model from data, we develop a
prior in the space of regression manifolds by resorting to a flexible
Bernstein polynomial prior on the space of angular densities as
recently proposed by \cite{hanson2017}. \vspace{0.2cm}

Our approach contributes to the literature on conditional modeling given large observed values \citep[e.g.][]{wang2011, cooley2012}, nonetheless, our focus differs from the latter papers in a number of important ways as we describe next. The main difference is that, as anticipated above, here the focus is on devising a regression framework for a block maxima response on block maxima covariate, whereas the latter papers focus mainly on using the conditional density as a way to make probabilistic statements about the likelihood of an extreme given the occurrence of another extreme. Since our main target of analysis is regression, our method has some links with
statistical approaches for nonstationary extremes \citep[e.g.][Section~6]{katz2013, eastoe2009, wang2009, yee2007, coles2001}; the most elementary version of approaches for nonstationary extremes aims to learn about how the limiting law of a suitably standardized block maxima response ($S_\mathbf{x}$) changes according to a covariate $\mathbf{x} = (x_1, \dots, x_p)^{\T}$, via the specification 
\begin{align}\label{gevx}
  (S \mid \mathbf{X} = \mathbf{x}) \sim \text{GEV}(\mu_{\mathbf{x}}, \sigma_{\mathbf{x}}, \xi_{\mathbf{x}}). 
\end{align}
Since the approach in \eqref{gevx} is built from the univariate theory of extremes it is not tailored for conditioning on another variable being extreme as it fails to take on board information from the dependence structure between the extremes. \vspace{0.2cm}

Additionally, the method proposed in this work is loosely related to 
quantile regression \citep{koenker1978}, whose original version consists in modeling the conditional quantile of a response $Y$ given a covariate $\mathbf{X} = (X_1, \dots, X_p)^{\T}$ in a linear fashion, that is
\begin{align}\label{qr}
  F^{-1}(q \mid \mathbf{x}) = \mathbf{x}^{\T} \betab_{q}, \qquad 0 < q < 1,
\end{align}
where $F^{-1}(q \mid \mathbf{x}) = \inf\{y: F(y \mid \mathbf{x}) \geq q\}$ and $F(y \mid \mathbf{x})$ is the distribution function of $Y \mid \mathbf{X} = \mathbf{x}$. Versions of quantile regression that aim to equip \eqref{qr} with the ability to extrapolate into the tail of $Y$ are often known as extremal quantile regression methods \citep[e.g.][]{chernozhukov2005}. While flexible and sturdy, such quantile regression-based approaches do not take into account information on the fact that the limiting joint distribution of suitably standardized componentwise maxima is an extreme value copula, and thus fail to be equipped with the ability to extrapolate into the joint tail. {The approach proposed in this paper will take such knowledge on the limiting joint distribution into consideration and will assume a conditional law that stems from such knowledge---rather than imposing a linear specification as in \eqref{qr}; yet, the proposed approach is not to be seen as a competitor to quantile regression but rather as a method based on some loosely related principles and specific to the context where we have a block maxima response and a block maxima covariate.}

\vspace{0.2cm}
The remainder of the paper unfolds as follows. In Section~\ref{model} we introduce the proposed model and Section~\ref{learn} devises an approach for learning about it from data. Section~\ref{simulation} reports the main findings of a Monte Carlo simulation study. We showcase the proposed methodology in a real data application to stock market data in Section~\ref{application}. Finally, in Section~\ref{discussion} we present closings remarks. Proofs and derivations can be found in the appendix, and further numerical experiments and other technical details are presented in the supplementary material. 

\newpage 
\section{Modelling limiting block maxima conditioned on block maxima}\label{model}

\subsection{Background on multivariate extremes}
Prior to introducing a regression of block maxima on block maxima we need to lay groundwork on multivariate extremes. Let $\{(Y_{i},\mathbf{X}_{i})\}_{i=1}^n$ be a sequence of independent random vectors with unit Fr\'echet marginal distributions, i.e. $\exp(-1 / z)$, for $z>0$. In our setup, $Y_i$ should be understood as a response, whereas $\mathbf{X}_i = (X_{1, i}, \dots, X_{p, i})$ should be understood as a $p$-dimensional covariate. Let the componentwise block maxima be $\mathbf{M}_n= (M_{n, y}, M_{n, x_1}, \dots, M_{n, x_p})$ with $M_{n, y} = \max \{ Y_1,\dots,Y_n \}$ and $M_{n, x_j} = \max(X_{j,1}, \dots, X_{j,n})$, for $j = 1, \dots, p$. Under this setup, it is well-known that 
the vector of normalized componentwise maxima $\mathbf{M}_n/n$ converges in distribution to a random vector $(Y, \mathbf{X})$ which follows a multivariate extreme value distribution with the joint distribution function \begin{align}
\label{GEV}
    G(y,\mathbf{x}) 
	= 
	\exp \{-V(y,\mathbf{x})\},
	\quad 
	y, x_1, \dots, x_{p} >0.
\end{align}
Here, $$V(y,\mathbf{x}) = d \int_{\Delta_d} \max \left( \dfrac{w_1}{y}, \dfrac{w_2}{x_1}, \dots, \dfrac{w_d}{x_{p}} \right) \, H(  \dif \mathbf{w}),$$ is the exponent measure; see, for instance, \cite{dehaan1977}, \cite{pickands1981}, and \citet[][Theorem~8.1]{coles2001}.
In addition, $H$ is a parameter of the multivariate extreme value distribution $G$ known as angular measure, which controls the dependence between the extreme values; specifically, $H$ is a probability measure on the unit simplex $\Delta_d = \{(w_1, \dots, w_d) \in [0,1]^d, \sum_{i=1}^d w_i = 1\} \subset \mathbb{R}^d$, with $d = p + 1$, and obeying the mean constraint 
\begin{align}
\label{momconstr}
    \int_{\Delta_d} \mathbf{w} \, H(\dif \mathbf{w}) = \dfrac{1}{d} \ind_d, 
\end{align}
where $\ind_d$ is a vector of ones in $\mathbb{R}^d$. If $H$ is absolutely continuous with respect to the Lebesgue measure then its density is given by the Radon--Nikodym derivative $h = \dif H / \dif \mathbf{w}$, for $\mathbf{w} \in \Delta_d$. 
\subsection{Regression manifold for block maxima on block maxima}
We are now ready to introduce our regression method for block maxima on block maxima. We define the regression manifold as the family of regression lines, 
\begin{equation}
  \label{lines}
  \mathscr{L} = \{L_q: 0 < q < 1\} \quad \text{with} \quad L_q = \{y_{q\mid \mathbf{x}}: \mathbf{x} \in (0, \infty)^p\}, 
\end{equation}
where 
\begin{align}
\label{genInv}
    y_{q\mid \mathbf{x}} 
    =
    \inf \left\{ y>0: G_{Y\mid \mathbf{X}}(y\mid \mathbf{x}) \geq q \right\},
\end{align}
 is a conditional quantile of a multivariate extreme
value distribution, with $q \in (0,1)$ and $\mathbf{x} \in  (0, \infty)^p$, and
$G_{Y|\mathbf{X}}(y \mid \mathbf{x})
=
\mathbb{P} \left\{ Y \leq y \mid \mathbf{X}=\mathbf{x} \right\}$ is a conditional multivariate extreme value distribution function. \par

In higher dimensions $G_{Y \mid \mathbf{X}}$ can be expressed with the help of a joint multivariate extreme value density $g_{Y , \mathbf{X}}$ and its expression has been derived by \cite{stephenson2005}. By applying Bayes' theorem, we deduce $G_{Y \mid \mathbf{X}}(y \mid \mathbf{x}) = \int_0^y g_{Y \mid \mathbf{X}}(z \mid \mathbf{x}) \, \dif z$ from $g_{Y , \mathbf{X}}$ with $g_{Y|\mathbf{X}}$ given as follows: 
\begin{align} 
\label{heavy}
	g_{Y|\mathbf{X}}(y \mid \mathbf{x})
    =
    \dfrac{ \exp\{ - V(y,\mathbf{x})\} \sum\limits_{i=1}^d \sum\limits_{j=1}^{n_{i}} (-1)^i \prod\limits_{\Lambda \in r_{ij}} V_{\Lambda}(y,\mathbf{x}) }{ \sum\limits_{i=1}^d \sum\limits_{j=1}^{n_{i}} (-1)^i \int\limits_0^{\infty} \exp\{ - V(y,\mathbf{x})\} \prod\limits_{\Lambda \in r_{ij}} V_{\Lambda}(y,\mathbf{x})  \, \dif y },
    \quad y, \mathbf{x} >0,
\end{align}
where $V_{\Lambda}(y,\mathbf{x})$ corresponds to mixed partial derivative of the exponent measure $V(y,\mathbf{x})$ with respect to the $l$th components of $(y,\mathbf{x})$ such that $l \in \Lambda$, $n_{i}$ is the number of partitions of $\{1,\dots,d\}$ of size $i=1,\dots,d$, and $r_{ij}$ is the $j$th partition of $\{1,\dots,d\}$ of size $i$, with $1 \leq j \leq n_{i}$.

In the particular case where we have a single covariate $(p=1)$, the regression manifold $\mathscr{L}$ in \eqref{lines} can be derived using properties of bivariate copulas; see Appendix~A. Accordingly, for an absolutely continuous angular measure $H$ (with density $h)$, it follows that
\begin{align} 
\label{condD}
    G_{Y \mid X}(y \mid x) = 
    2
    \exp \left\{ - 2 \int_0^1 \max \left\{ \dfrac{w}{x}, \dfrac{1-w}{y}  \right\} h(w)   \, \dif w + \frac1x \right\} \int_{\omega(x, y)}^1 w h(w)  \, \dif w,
    \quad x,y>0,
\end{align}
where $\omega(x, y) = x / (x + y)$, and $y_{q\mid x}$ is then calculated via \eqref{genInv}.

\begin{figure}
  \centering
  \includegraphics[width=0.5\textwidth]{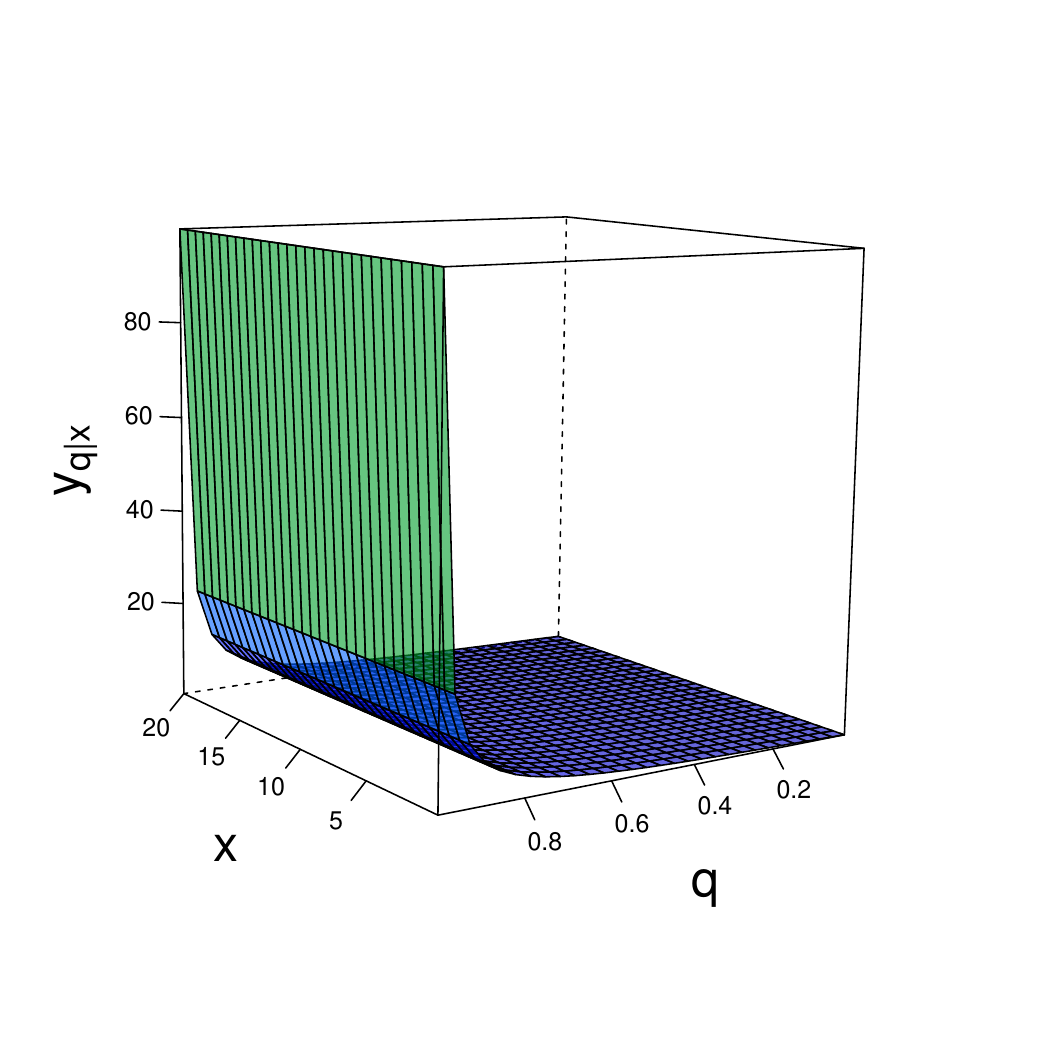} \hspace{-0.8cm}
  \includegraphics[width=0.5\textwidth]{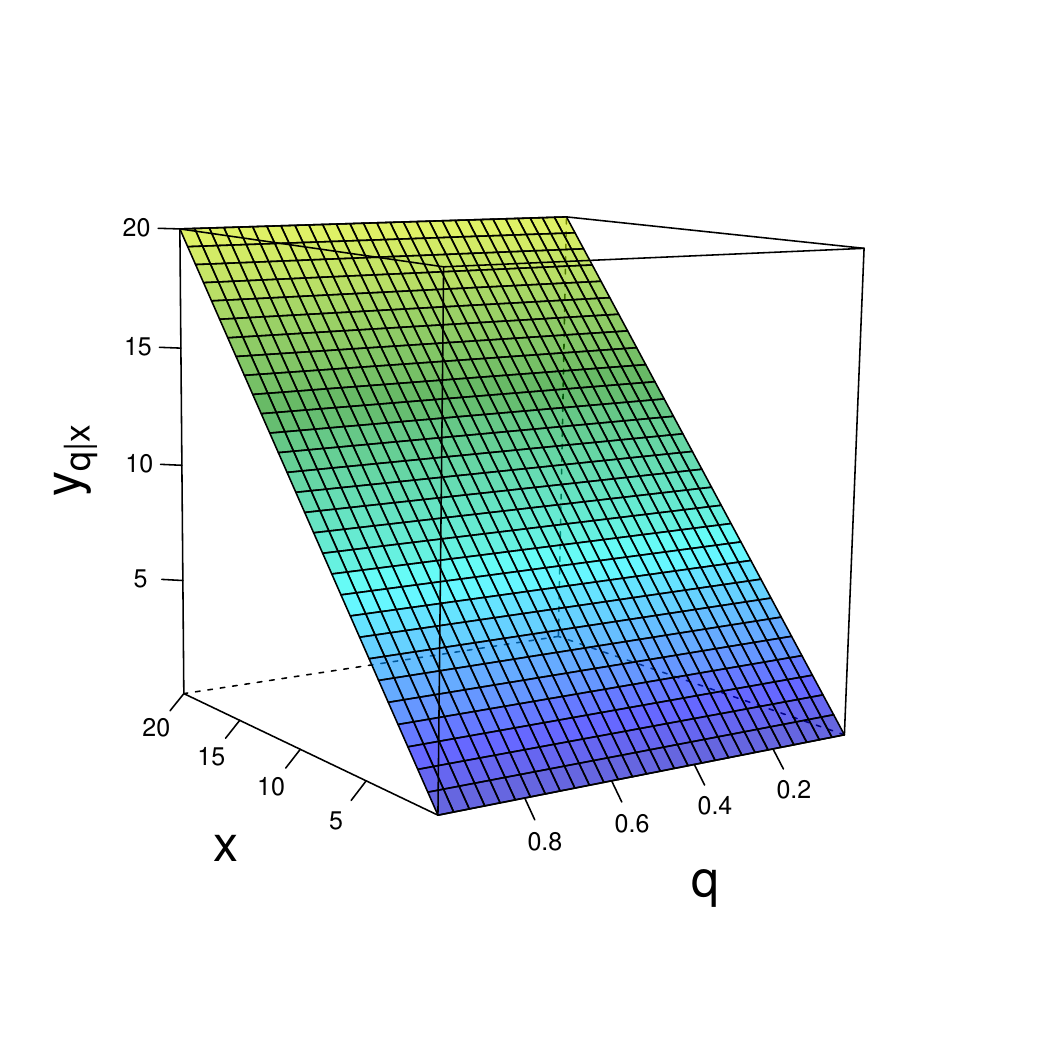}  \vspace{-0.2cm}
  \caption{\label{inddep} Regression mainfolds for cases of complete independence (left) and perfect dependence (right).}
\end{figure}
We now derive regression manifolds $\mathscr{L}$ in \eqref{lines} for the cases of independent and perfectly dependent extremes, which are depicted in Figure~\ref{inddep}. When extremes are independent, $H$ assigns equal mass to the boundaries of the simplex, which also corresponds to asymptotic independence of $\mathbf{X}$ and $Y$ \citep{husler2009}, resulting in
\begin{align}\label{regind}
  L_q = \{ -1/\log q: \mathbf{x} \in (0, \infty)^p \},
\end{align}
with
\begin{align*}
    G_{Y\mid \mathbf{X}}(y\mid \mathbf{x}) 
    =
    \int_0^y
    \dfrac{\exp(- z^{-1} - x^{-1}_1 - \dots - x_{p}^{-1}) (-z^{-2}) \prod\limits_{j=1}^{p} (-x_j^{-2})}{ \exp(- x^{-1}_1 - \dots - x_{p}^{-1}) \prod\limits_{j=1}^{p} (-x_j^{-2})} \, \dif z 
    =
    \exp({-1 / y}),
    \quad 
    y > 0.
\end{align*}

\noindent When extremes are perfectly dependent, the angular measure $H$ assigns all its mass to the barycenter of the simplex, $d^{-1}\ind_d$, leading to $G(y,\mathbf{x}) = \exp \{ -  \max (y^{-1}, x_1^{-1}, \dots, x_p^{-1}) \}$. Taking derivatives of $G(y,\mathbf{x})$ in this case is non-trivial, and we replace the maximum function with a soft maximum \citep{cook2011} so to obtain an approximation for the shape of the regression lines for perfectly dependent extremes. Thus, the soft maximum approximation for the regression lines for perfectly dependent extremes is 
\begin{align}\label{softy}
  \tilde L_q = \{\min(x_1,\dots,x_p): \mathbf{x} \in (0, \infty)^p\}.
\end{align}
Thus, regression lines for the case of perfectly dependent extremes do not depend on $q$. See~Appendix~B for the derivation, and Figure~\ref{inddep} for a chart of its regression manifold.

We end this section with comments on properties of regression manifolds. Trivially, regression lines obey the standard properties of quantile functions \citep[][Chap.~21]{vaart1998}. Less trivial is however the fact that, monotone regression dependence of bivariate extremes \cite[][Theorem~1]{guillem2000} implies that regression lines $y_{q \mid x}$ in \eqref{genInv} are non-decreasing in $x$, for $p=1$, under some mild  assumptions.
\begin{proposition}[Monotonicity of regression manifold]
\label{prop7}
    Let $G_{Y|X}(y \mid x) =
\mathbb{P} \left\{ Y \leq y \mid X=x \right\}$ be a conditional bivariate extreme value distribution function, which we assume to be jointly continuously differentiable and strictly increasing in $y$ for any fixed $x \in (0,\infty)$.
    Then, the regression lines for bivariate extremes
	$(0,\infty)\ni x\mapsto y_{q\mid x}$
	are non-decreasing for all $q\in(0,1)$.
      \end{proposition}
      \begin{proof}
        See Appendix~C.
      \end{proof}
An example of a bivariate extreme value distribution satisfying
the assumptions of Proposition~\ref{prop7} is the Logistic
model, whose regression manifold is discussed in Example~\ref{logistic} below.
      
\subsection{Parametric instances of regression manifolds}
\noindent We now consider some parametric instances of regression manifolds as defined in \eqref{lines}. Charts of regression manifolds for these parametric examples are depicted in Figure~\ref{F2}. In Appendix~D, we show that for sufficiently large $x$, the following linear approximation holds for the regression manifold, $L_q = \{y_{q\mid x}: x \in (0, \infty)\}$, of the Logistic model from Example~\ref{logistic} with  
        \begin{equation}
          \label{RLLogisticA}
          y_{q\mid x} 
	= \gamma_q + \beta_q x + o{(x)}. 
        \end{equation}
Here, $\gamma_q$ and $\beta_q$ are functions of both $\alpha$ and $q$ (see \eqref{gammaq} and \eqref{betaq}), and $o{(x)}$ is little-$o$ of $x$ in Bachmann--Landau notation; the numerical accuracy of this approximation is illustrated in the supplementary material.

\begin{example}[Logistic]\normalfont
  \label{logistic}
  An instance of the Logistic regression manifold can be found in Figure~\ref{F2} (top). It stems from the Logistic bivariate extreme value  distribution function given by 
	\begin{align*}
	    G(x,y) 
	    = 
	    \exp \{ - ( x^{-1/\alpha} + y^{-1/\alpha} )^{\alpha} \},
	    \quad
	    x,y>0,
	\end{align*}
    where $\alpha \in (0,1]$ characterizes the dependence between extremes: The closer $\alpha$ is to $0$, the stronger the dependence, with the limit $\alpha \to 0$ corresponding to the case of perfect dependence. The conditional distribution of $Y$ given $X$ is  
	\begin{align*}
	    G_{Y\mid X}(y\mid x)
	    = 
	    G(x,y) ( x^{-1/\alpha} + y^{-1/\alpha})^{\alpha - 1} x^{1-1/\alpha}\exp(1/x),
	    \quad 
	    x,y > 0,
	\end{align*}
	thus leading to the following family of regression lines $L_q$ in \eqref{lines} where
	\begin{align}
          \label{RLLogistic}
			y_{q\mid x}
			=
			\left[ \left\{ \dfrac{1-\alpha}{\alpha} x W \left( \dfrac{\alpha}{1-\alpha} x^{-1} e^{ {\alpha}/(1-\alpha) x^{-1}} q^{{\alpha}/(\alpha-1)} \right) \right\}^{1/\alpha} - 1 \right]^{-\alpha} x,
	\end{align}
	and $x > 0$. Here, $W$ is the so-called Lambert $W$ function, that is, the multivalued analytic inverse of $f(z) = z \exp(z)$ with $z$ denoting a real or complex number \citep{BorweinLindstrom2016LambertFunc}; see the supplementary material for further details. As it can be seen from Figure~\ref{F2} (top), the regression lines obey what is claimed in Proposition~\ref{prop7} in the sense that $(0,\infty)\ni x\mapsto y_{q\mid x}$ are non-decreasing for all $q \in (0, 1)$.
        
\begin{figure}
  \begin{center} 
  \begin{footnotesize}
    \textbf{Logistic regression manifold} \\
    \end{footnotesize}   \vspace{-0.65cm} \hspace{-1cm}
    \begin{minipage}[c]{4cm}
    \centering \includegraphics[width=1.4\textwidth]{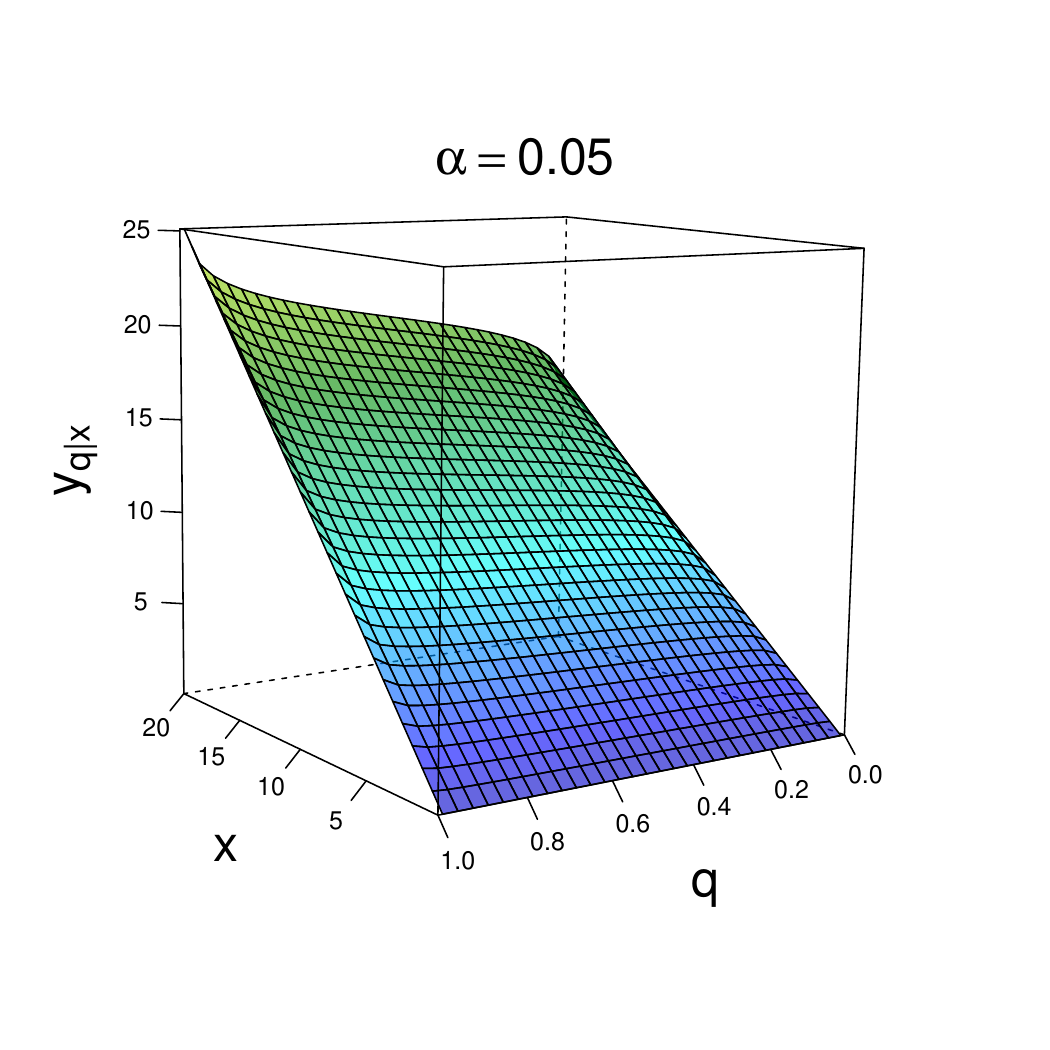}
    \text{\hspace{0cm}}
  \end{minipage} 
   \hspace{0.65cm}
  \begin{minipage}[c]{4cm}
    \includegraphics[width=1.4\textwidth]{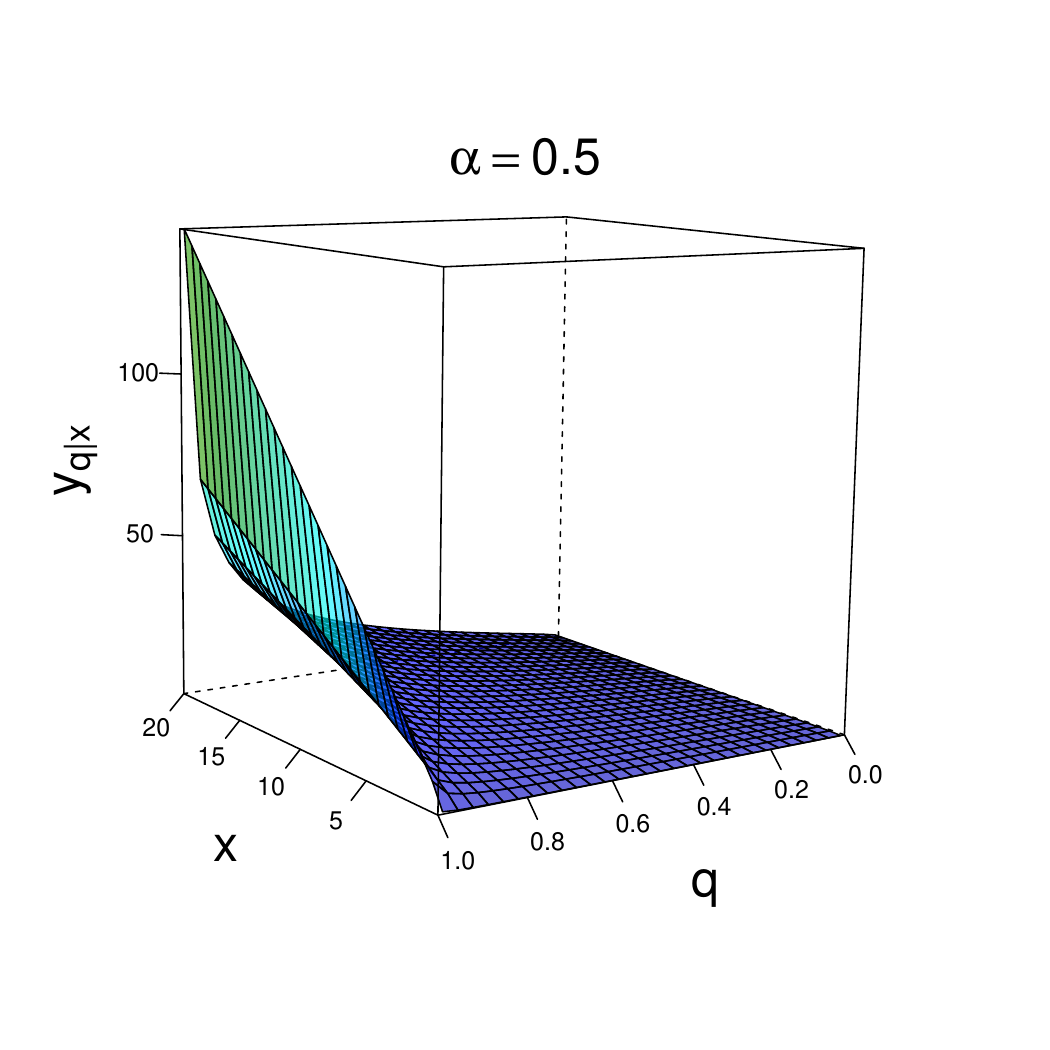}
    \end{minipage}
  \hspace{0.65cm}
    \begin{minipage}[c]{4cm}
    \includegraphics[width=1.4\textwidth]{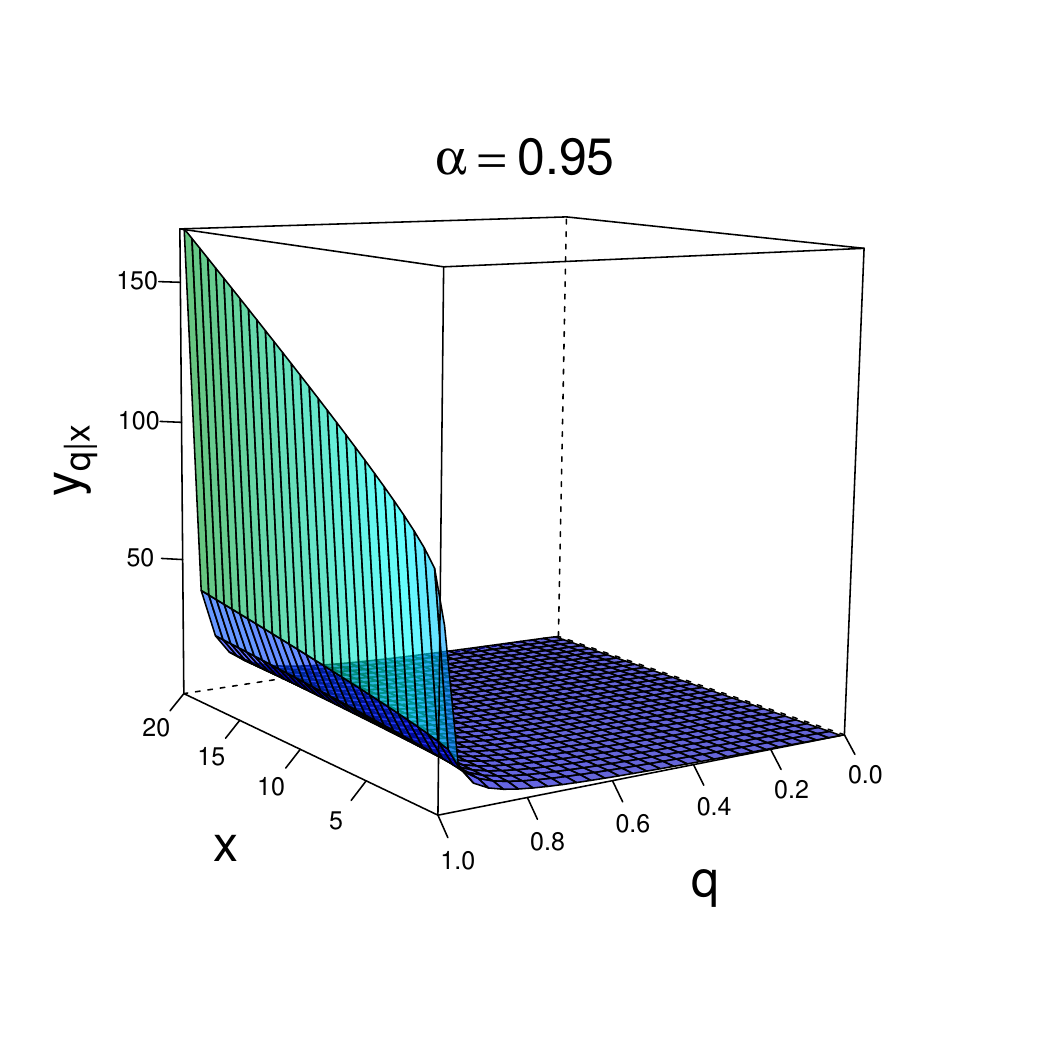}
  \end{minipage} 
\end{center}\vspace{-1cm}
\begin{center}
\begin{footnotesize}
    \textbf{Husler--Reiss regression manifold} \\ \vspace{-0.65cm} \hspace{-1cm}
\end{footnotesize}
    \begin{minipage}[c]{4cm}
    \centering \includegraphics[width=1.4\textwidth]{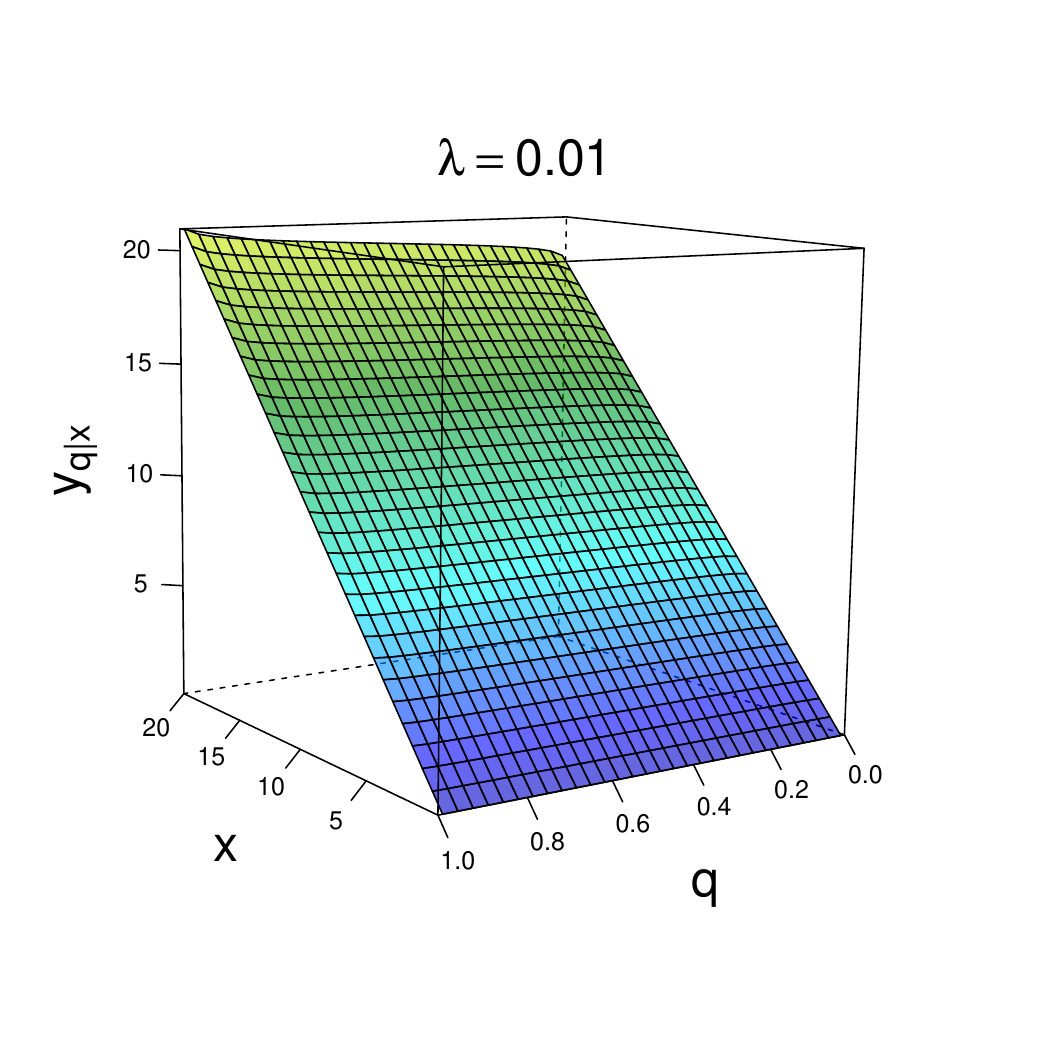}
    \text{\hspace{0cm}}
  \end{minipage} 
   \hspace{0.65cm}
  \begin{minipage}[c]{4cm}
    \includegraphics[width=1.4\textwidth]{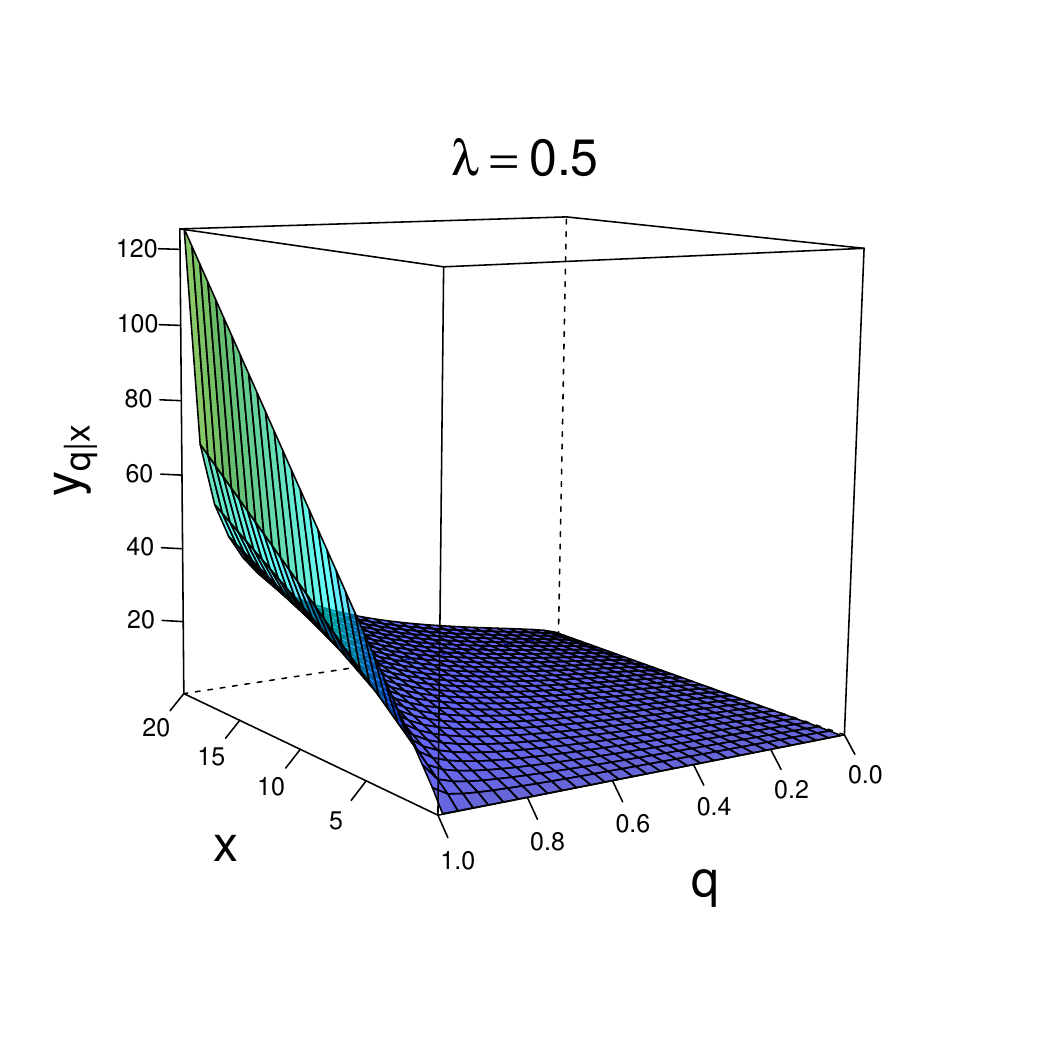}
    \end{minipage}
  \hspace{0.65cm}
    \begin{minipage}[c]{4cm}
    \includegraphics[width=1.4\textwidth]{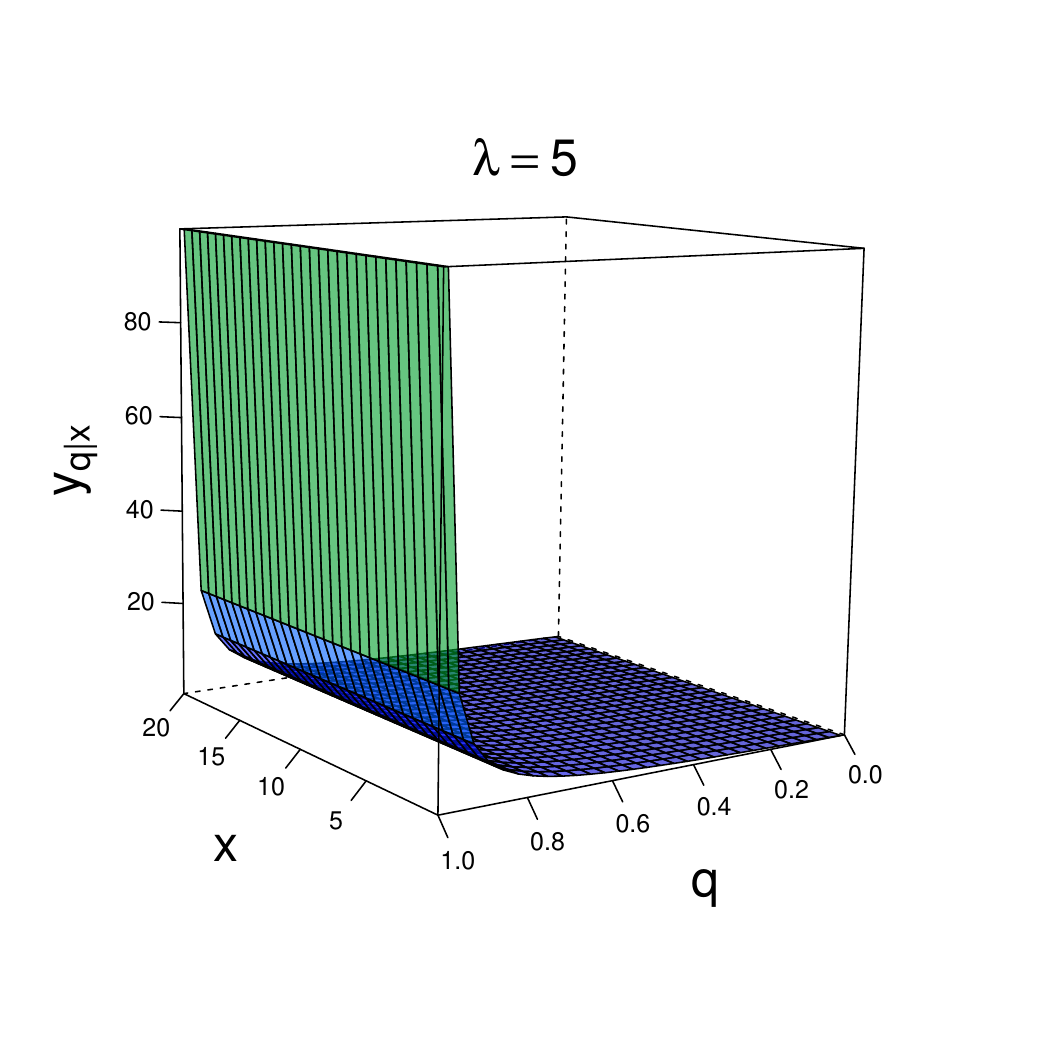}
  \end{minipage} 
  \end{center} \vspace{-1cm}
  \begin{center} 
  \begin{footnotesize}
  \textbf{Coles--Tawn regression manifold} \\ \vspace{-0.65cm} \hspace{-1cm}
  \end{footnotesize}
    \begin{minipage}[c]{4cm}
    \centering \includegraphics[width=1.4\textwidth]{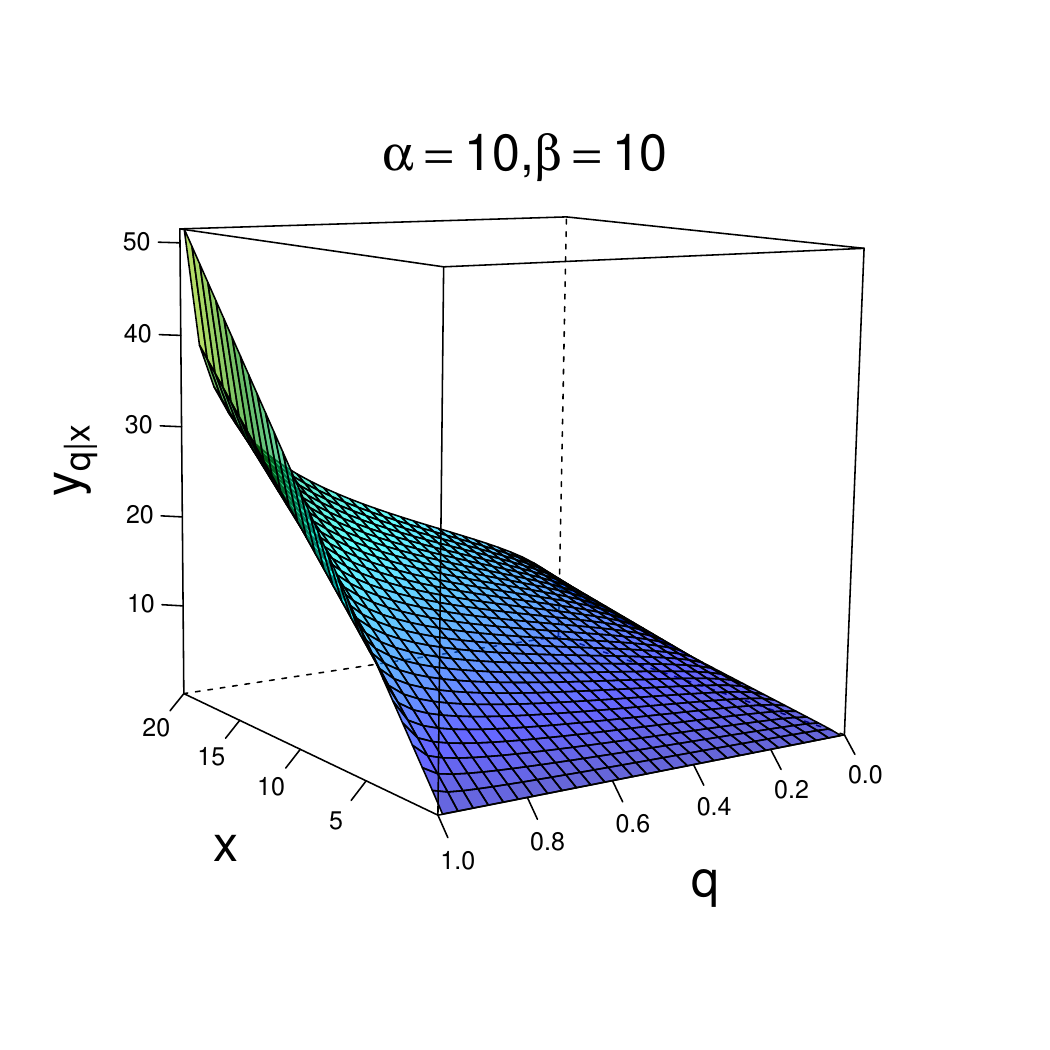}
    \text{\hspace{0cm}}
  \end{minipage} 
   \hspace{0.65cm}
  \begin{minipage}[c]{4cm}
    \includegraphics[width=1.4\textwidth]{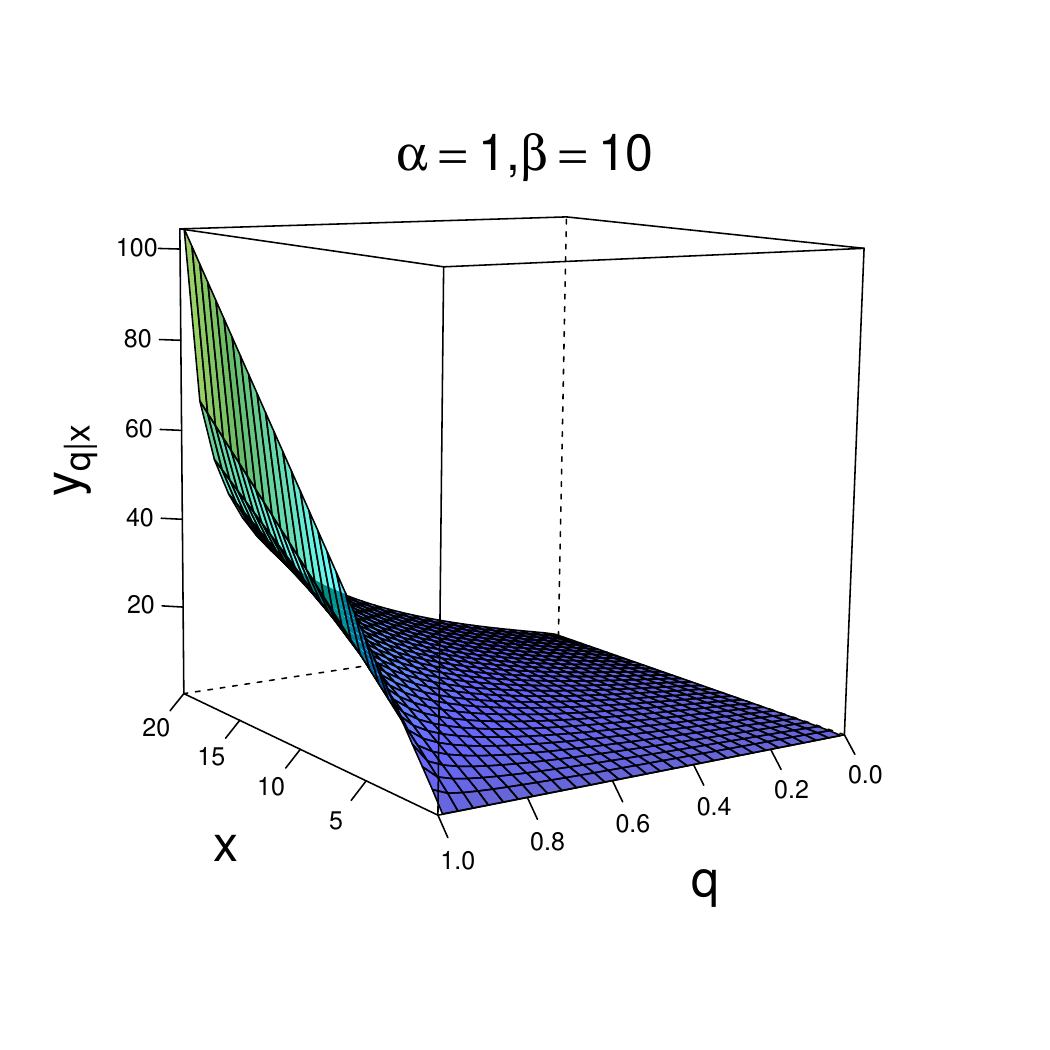}
    \end{minipage}
  \hspace{0.65cm}
    \begin{minipage}[c]{4cm}
    \includegraphics[width=1.4\textwidth]{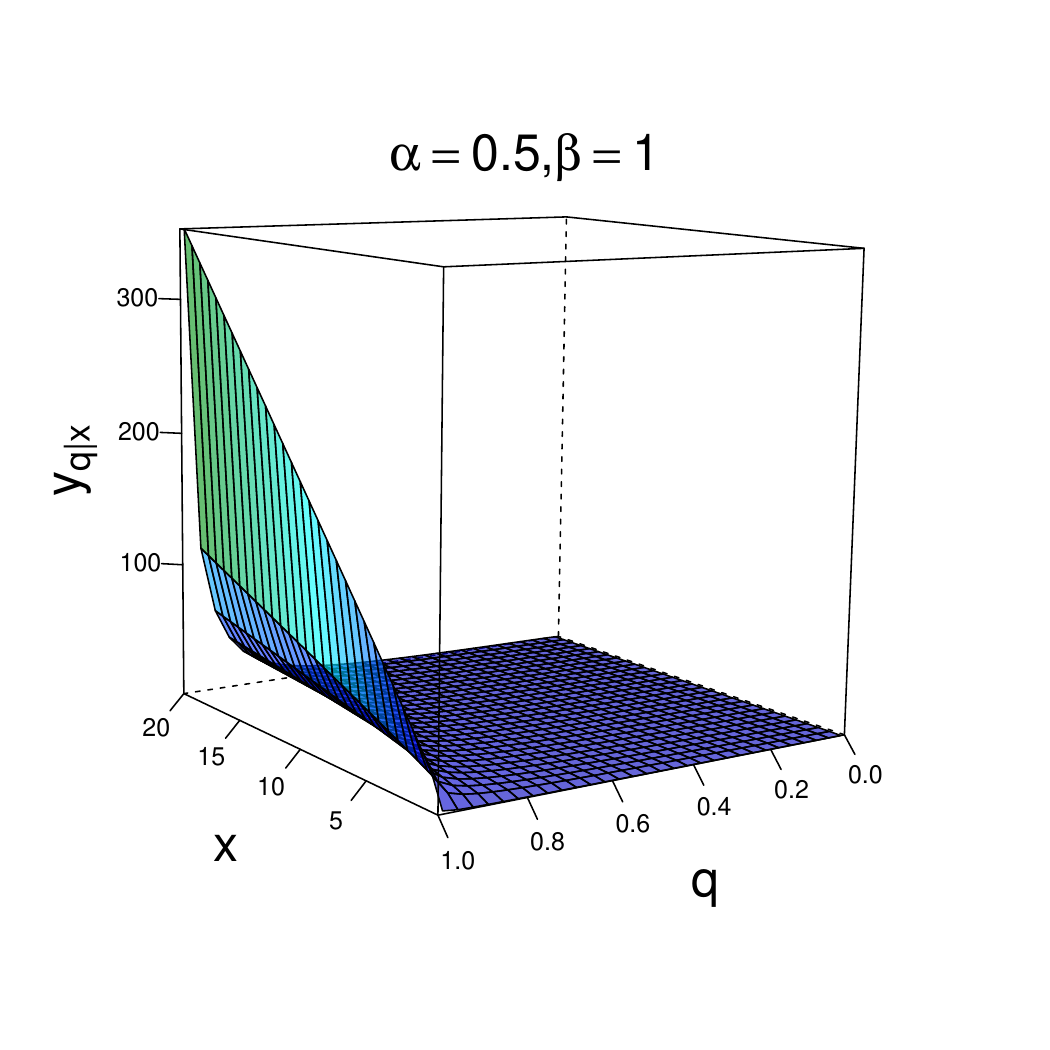}
  \end{minipage} 
  \end{center}
  \vspace{-1cm}
\caption{Regression manifold $\mathscr{L}$, as defined in \eqref{lines}, for bivariate Logistic, Husler--Reiss, and Coles--Tawn models (top to bottom) with strong dependence, intermediate and weak extremal dependence (left to right).}
\label{F2}
\end{figure}
\end{example}

\begin{example}[Husler--Reiss]\normalfont
  \label{husler_reiss}
  An instance of the Husler--Reiss regression manifold is depicted in Figure~\ref{F2} (middle). It follows from the Husler--Reiss bivariate extreme value distribution function which has the following form:
	\begin{align*}
	    G(x,y) 
	    = 
	    \exp \left\{ - x^{-1} \Phi\left( \lambda + \dfrac{1}{2 \lambda} \log \dfrac{y}{x} \right) - y^{-1} \Phi\left( \lambda + \dfrac{1}{2 \lambda} \log \dfrac{x}{y} \right) \right\},
	    \quad
	    x,y>0,
	\end{align*}
        where $\Phi$ is the standard Normal distribution function and $\lambda \in (0,\infty]$ is the parameter regulating the dependence between extremes: $\lambda \to 0$ corresponds to perfect dependence and the limit case $\lambda \to \infty$ corresponds to complete independence. The family of regression lines $L_q$ in \eqref{lines} for this model does not have explicit representations and is obtained using \eqref{genInv} with
	\begin{align*}
	    G_{Y\mid X}(y\mid x)
	    &= 
	    \left[ \Phi\left( \lambda + \dfrac{1}{2 \lambda} \log \dfrac{y}{x} \right) + { \dfrac{1}{2 \lambda} \phi \left( \lambda +  \dfrac{1}{2\lambda} \log \dfrac{y}{x} \right)} -
	     \dfrac{
	     x y^{-1}}{2\lambda} \phi \left( \lambda + \dfrac{1}{2 \lambda} \log \dfrac{x}{y} \right)  \right]\\
	    &\quad 
	    \times
	    G(x,y) \exp(1 / x),
	    \quad 
	    x,y > 0,
	\end{align*}
    where $\phi$ is the standard Normal density function.  
\end{example}

\begin{example}[Coles--Tawn]\normalfont
  \label{coles_tawn}
    An instance of the Coles--Tawn regression manifold is depicted in Figure~\ref{F2} (bottom). It follows from the Coles--Tawn bivariate extreme value  distribution function which has the following form:
	\begin{align*}
	    G(x,y) 
	    = 
	    \exp[ - x^{-1} \{ 1 - \text{Be} ( q; \alpha + 1, \beta)\}- y^{-1} \text{Be}( q; \alpha, \beta + 1 )
	    ],
	    \quad
	    x,y>0,
	\end{align*}
    where $\text{Be}(q;a,b)$ is the distribution function of a Beta distribution function with parameters $a,b > 0$, $q = \alpha y^{-1}/(\alpha y^{-1} + \beta x^{-1})$ and $\alpha,\beta > 0$ are the parameters regulating dependence between extremes; the case $\alpha = \beta=0$ corresponds to complete independence, whereas $\alpha = \beta \to \infty$ corresponds to perfect dependence. For fixed $\alpha$ ($\beta$) the strength of dependence increases with $\beta$ ($\alpha$). The family of regression lines $L_q$ in \eqref{lines} for this model does not have an explicit representation and is calculated using \eqref{genInv}, for $x,y > 0$, with
	\begin{align*}
	    G_{Y\mid X}(y\mid x)
	    = 
	    \Big[ 1 - \text{Be}\left( q; \alpha + 1, \beta \right) 
	    & + \dfrac{(\alpha + 1)\beta}{ \gamma} \text{be}\left( q; \alpha + 2, \beta + 1 \right) 
	    \\ 
	    &-
	    \dfrac{x}{y} 
	    \dfrac{\alpha(\beta + 1)}{\gamma} \text{be}\left( q; \alpha + 1, \beta + 2 \right)
	    \Big]
	    x^{-2} G(x,y) \exp\{ x^{-1} \} ,
	\end{align*}
    where $\text{be}(q;a,b)$ is the density function of the Beta distribution with parameters $a,b>0$ and $\gamma =(\alpha + \beta + 2)(\alpha + \beta + \textcolor{blue}{1})$.  
\end{example}


\noindent Section~\ref{model} introduced our key parameter of interest---regression manifolds for block maxima on block maxima, i.e.~$\mathscr{L}$ as in \eqref{lines}---, it commented on some of its properties, and gave examples of parametric instances. 
Next, we discuss Bayesian inference for $\mathscr{L}$. 

\section{Learning about regression manifolds via Bernstein polynomials}
\label{learn}
\subsection{Induced prior on the space of regression manifolds for $p=1$}\label{bern}
In this section we discuss how to learn about regression manifolds from data. To achieve this, we resort to the Bayesian paradigm and will define an induced prior on the space of regression manifolds by resorting to a flexible prior on the space of all angular measures that was recently proposed by \cite{hanson2017}. To lay the groundwork, we start by defining the setup of interest. Let $\{(Y_i,\mathbf{X}_i)\}_{i=1}^n$ be a sequence of independent random vectors with unit Fr\'echet marginal distributions; define  $R_i = Y_i + \sum_{j=1}^p X_{j,i}$ and $\mathbf{W}_i = (Y_i , \mathbf{X}_i)/ R_i$, known as the pseudo-angular decomposition of the observations. In \cite{dehaan1977} it is shown the equivalence of the convergence of normalized componentwise maxima to $G$ to the following weak convergence of measures 
\begin{align*}
    \mathbb{P} \left\{ \mathbf{W} \in \cdot \mid R > u \right\} \dto 
    H(\cdot), \quad \text{as} \; u \rightarrow \infty.
\end{align*}
This means that when the radius $R$ is sufficiently large, the pseudo-angles $\mathbf{W}$ are nearly independent of $R$ and follow approximately a distribution associated with the angular measure $H$. Thus, to learn about $L_q$ in \eqref{lines}, we first learn about $H$ based on $k = |\{\mathbf{W}_i: R_i > u, i = 1, \dots, n\}|$ exceedances above a large threshold $u$, with a methodology we describe next.

Following \cite{hanson2017}, we model the angular density $h$ via a Bernstein polynomial defined on the unit simplex $\Delta_d$, and hence basis polynomials are Dirichlet densities. More precisely, our specification for the angular density is 
\begin{align}
\label{berndens}
    h(\mathbf{w})
    =
    \sum\limits_{|\alphab|=J} \pi_{\alphab} \, \text{dir}(\mathbf{w}; {\alphab}), \quad 
\end{align}
with $\mathbf{w} \in \Delta_d$. Here, $\text{dir}_d$ is the density of a Dirichlet distribution supported on $\Delta_d$, that is, 
\begin{equation*}
\text{dir}(\mathbf{w}; \alphab) = 
    \dfrac{\Gamma(|\alphab|)}{\prod\limits_{i=1}^d \Gamma(\alpha_i)} \prod_{i=1}^d w_i^{\alpha_i-1},
\end{equation*}
where $\alphab \in \mathbb{N}^d$ (with $\mathbb{N}:=\{1,2,3,\dots\}$), $|\alphab|=\sum_{j=1}^d \alpha_j$, and $\Gamma(z) = \int_0^\infty x^{z - 1} \exp(-x) \, \dif x$ is the gamma function; finally in \eqref{berndens} the $\pi_{\alphab} > 0$ are weights and $J \in \mathbb{N}$ controls the order of the resulting polynomial.

To ensure that the resulting $h(\mathbf{w})$ is a valid angular density (i.e.~an actual density satisfying the moment constraint \eqref{momconstr}), the weights must obey  
\begin{align}
  \label{c1}
  \sum\limits_{|\alphab|=J} \pi_{\alphab} = 1, \quad 
  \sum\limits_{i=1}^{J-d+1} i \sum\limits_{|\alphab|=J, \alpha_j=i} \pi_{\alphab} = \dfrac{J}{d}, 
\end{align}
for $j=1,\dots,d$. The normalization and mean constraints in \eqref{c1} imply that there are $m - d$ parameters, where $m = {J-1 \choose d-1}$ is the number of basis functions in \eqref{berndens}; denote such free weights as $\{\pi_{\alphab}: \alphab \in \mathscr{F}\},$ where $\mathscr{F} = \{\alphab \in \mathbb{N}^d, |\alphab| = J, \text{ and } \alphab \not\in \{\mathbf{a}_1, \dots, \mathbf{a}_d\}\}$ with $\mathbf{a}_j$ being a $J$-vector of ones except element $i$ is $J - d + 1$. Similarly to \cite{hanson2017}, 
we parametrize the free weights via a generalized logit transformation that implicitly defines the auxiliary parameters  $\pi_{\alphab}'$, that is,
\begin{equation}\label{auxi}
  \pi_{\alphab} = \frac{\exp(\pi_{\alphab}')}{d + \sum_{\tilde \alphab \in \mathscr{F}}\exp(\pi_{\tilde \alphab}')}.
\end{equation}
Now, to induce a prior in the space of regression manifolds we plug-in the angular density in \eqref{berndens} into \eqref{heavy}; subsequent integration with respect to $y$ and inversion of $G_{Y|\mathbf{X}}(y|\mathbf{x})$ leads to an induced prior on the space of  regression lines $L_q$. In detail, to define a prior on the space of  regression manifolds we proceed as follows. The Bernstein polynomial prior in \eqref{berndens} induces a prior on the space of  regression lines $L_q = \{{y}_{q \mid x}: x \in (0,\infty)\}$, where ${y}_{q \mid x}$ is a solution to equation, ${G}_{Y\mid X} (y\mid x)
    =
  q$, for $q \in (0,1)$, where
\begin{align}
    &{G}_{Y\mid X} (y\mid x) \nonumber
    \\ 
    &= \label{indGYX}
    \dfrac{2}{J} \exp\Bigg\{ 
    -\dfrac{2}{J} \sum\limits_{|\alpha|=J} \pi_{\alphab} 
    [ \alpha_1 x^{-1} 
        \{ 1 -  \text{Be}(\omega(x, y);\alpha_1 + 1,\alpha_2)
        \} 
    + \alpha_2 y^{-1} \text{Be}(\omega(x, y);\alpha_1,\alpha_2+1)] \Bigg\} \nonumber
    \\
    &\quad
    \times
    \sum\limits_{|\alphab|=J} \pi_{\alphab} \alpha_1 \{1 -  \text{Be}(\omega(x, y);\alpha_1+1,\alpha_2) \} \exp(1/x),
\end{align}
where $\omega(x, y) = x / (x + y)$, for $x, y > 0$. Finally, to complete the model specification we set the following Dirichlet prior on the free parameters
\begin{equation*}
  p(\pi_{\alphab}) \propto \text{dir}(\mathbf{w} \mid c \, \ind_m) \prod_{j = 1}^d I\left\{\sum\limits_{i=1}^{J-d+1} i \sum\limits_{|\alphab|=J, \alpha_j=i} \pi_{\alphab} = \dfrac{J}{d}\right\}, 
\end{equation*}
where $I$ is the indicator function, which accordingly induces a prior on the auxiliary parameters $\pi_{\alphab}'$ in \eqref{auxi}.

\subsection{Induced prior on the space of regression manifolds for $p>1$}\label{higher_dim}
When $p \geq 2$ we proceed as in Section~\ref{bern}, that is our induced prior in the space of regression lines is again induced by the Bernstein polynomial prior for the angular density in \eqref{berndens}, and it follows by solving ${G}_{Y\mid \mathbf{X}} (y\mid \mathbf{x}) = q$, with $h(\mathbf{w})$ as in \eqref{berndens}. The expression for the conditional multivariate extreme value distribution ${G}_{Y\mid \mathbf{X}} (y\mid \mathbf{x})$ for $p \geq 2$ is however not as manageable as the one in \eqref{indGYX}.
We thus propose an approach for learning about the regression manifold $\mathscr{L}$, as defined in \eqref{lines}, via an approximation to the conditional multivariate GEV density.
Let $\mathbf{u}=(y,\mathbf{x}) \in (0,\infty)^d$ and $\mathbf{u}(t)=(t,\mathbf{x}) \in (0,\infty)^d$ with $\lVert \mathbf{u} \rVert = y + \sum_{i=1}^p x_i > u$ for a large threshold $u$.

Then, following \citet[Proposition~1]{cooley2012} the conditional density of a multivariate extreme value distribution can be approximated, via a point process representation for extremes, as follows
\begin{equation}
\label{approxi}
    g_{Y \mid \mathbf{X}} (y \mid \mathbf{x})
    \approx  
    \dfrac{\lVert \mathbf{u} \rVert^{-d-1} h(\mathbf{u}/\lVert \mathbf{u} \rVert)}{\int_0^{\infty} \lVert \mathbf{u}(t) \rVert^{-d-1} h(\mathbf{u}(t)/\lVert \mathbf{u}(t) \rVert) \, \dif t}.
  \end{equation}
 
\renewcommand{\thefigure}{4.1}
\begin{figure}
\begin{center}
    \begin{center}
      \begin{footnotesize}
        \textbf{Scenario 1}---strongly dependent extremes: Husler--Reiss model
      \end{footnotesize}
  \end{center}
  \vspace{-1cm}
    \begin{minipage}[c]{7cm}\hspace{1cm}
    \includegraphics[width=1\textwidth]{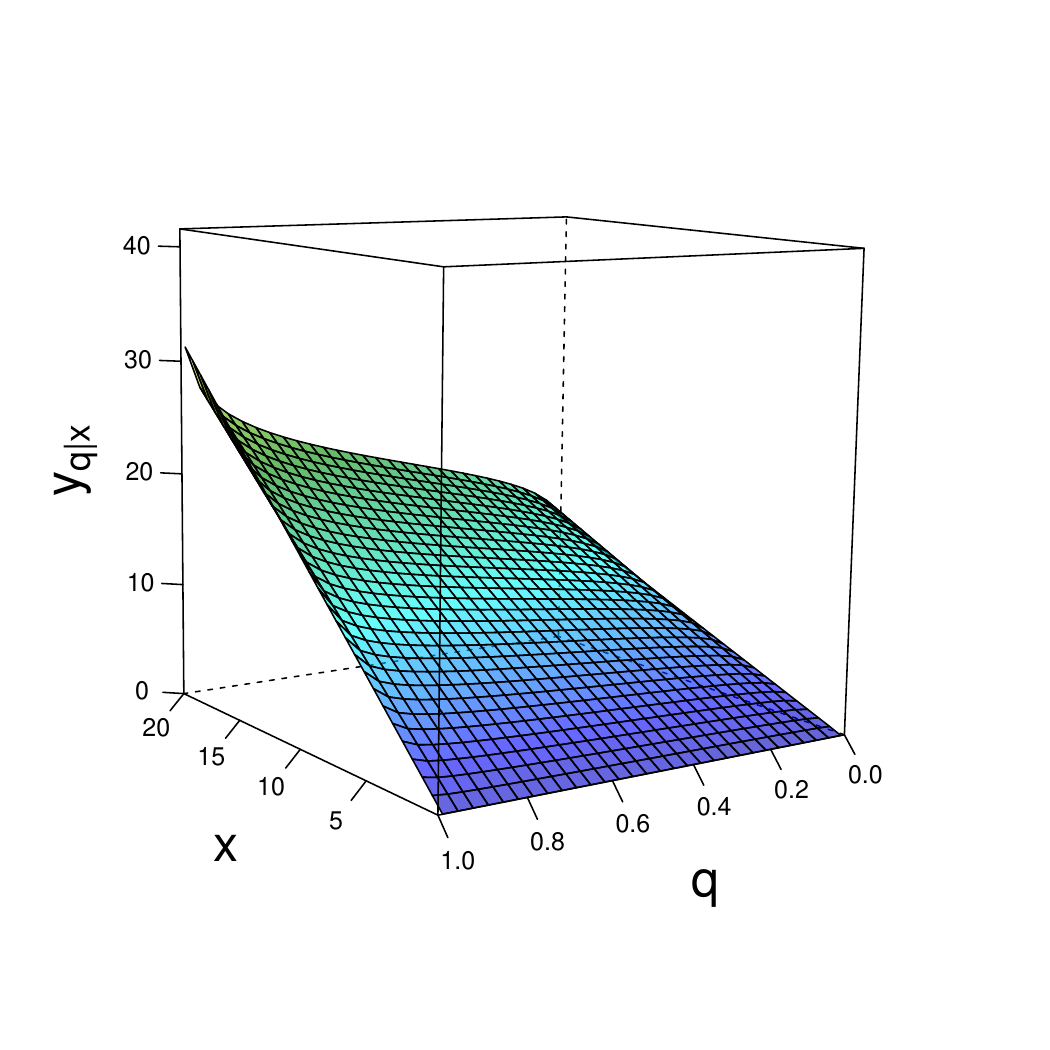}
    \text{\hspace{0cm}}
  \end{minipage} 
  \begin{minipage}[c]{7cm}
    \includegraphics[width=1\textwidth]{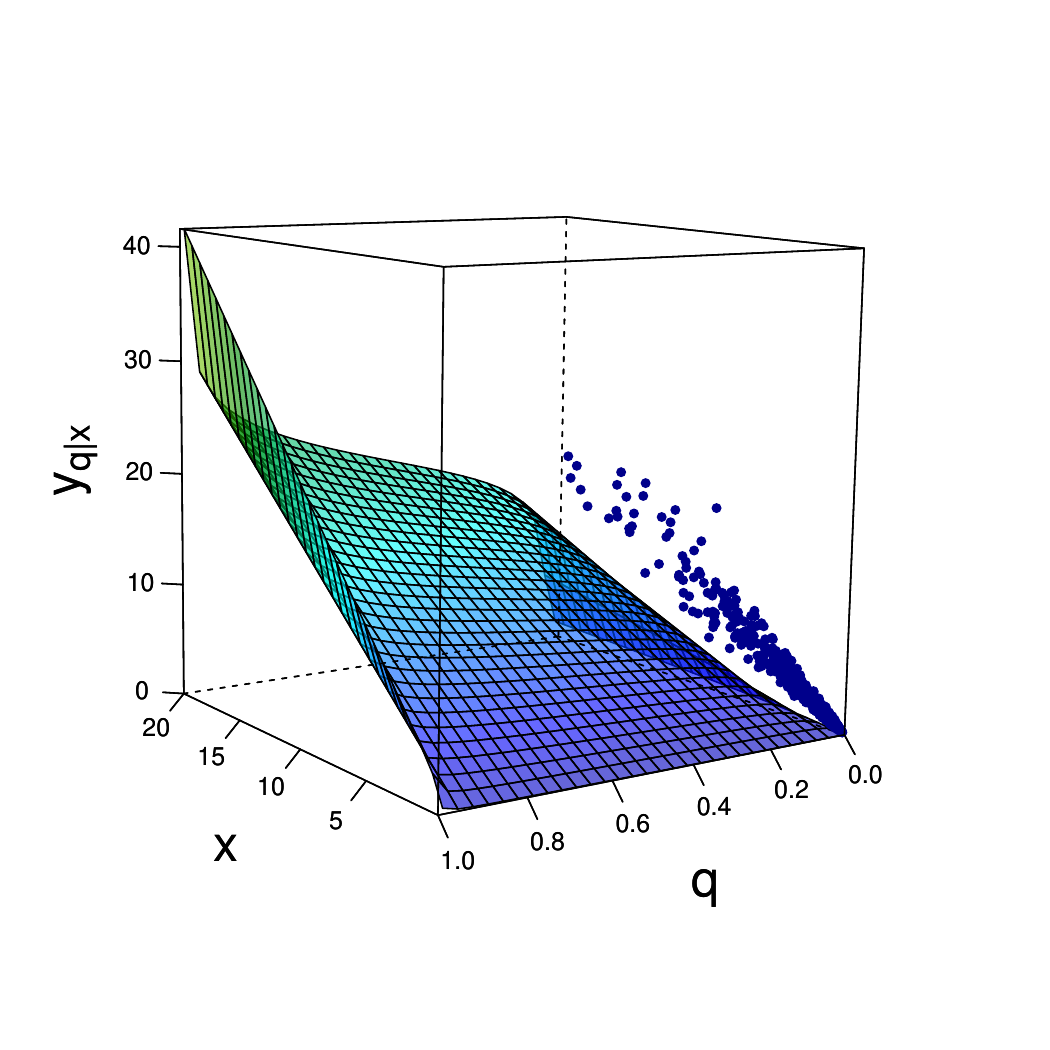}
  \end{minipage}
  \vspace{-1cm}
    \begin{center}
  \begin{footnotesize}
    \textbf{Scenario 2}---weakly dependent extremes: Logistic model
  \end{footnotesize}
  \end{center}
  \vspace{-1cm}
    \begin{minipage}[c]{7cm}\hspace{1cm}
    \includegraphics[width=1\textwidth]{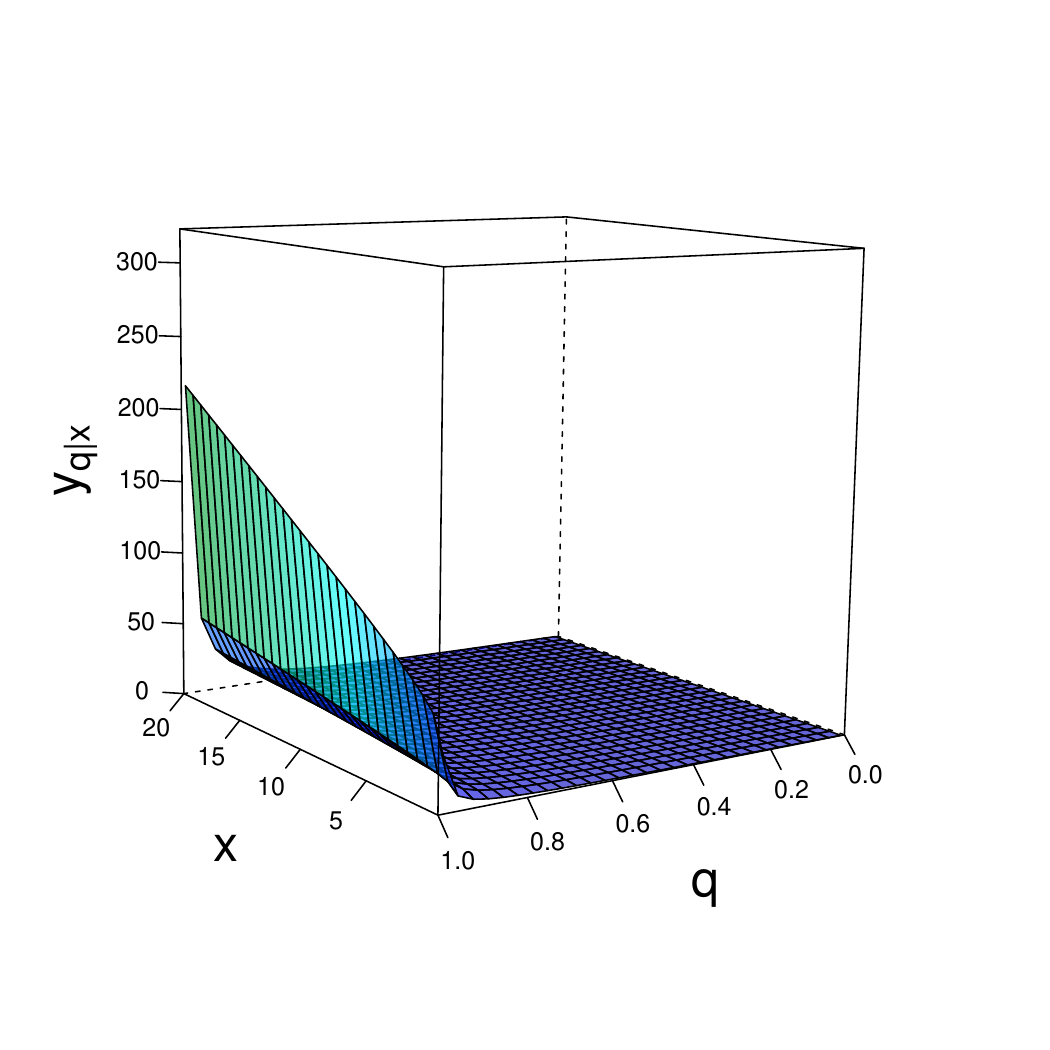}
    \text{\hspace{0cm}}
  \end{minipage} 
  \begin{minipage}[c]{7cm}
    \includegraphics[width=1\textwidth]{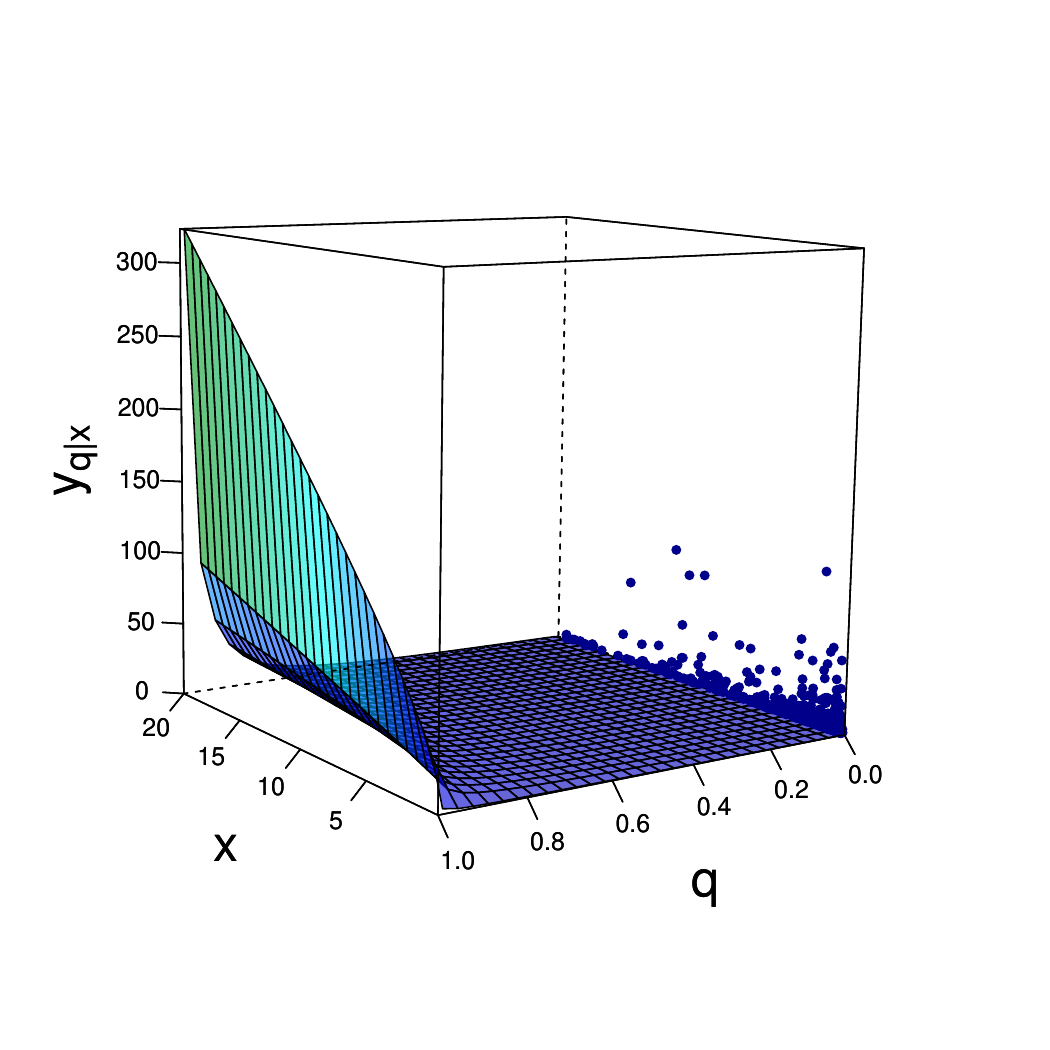}
  \end{minipage}
      \vspace{-1cm}
    \begin{center}
  \begin{footnotesize}
    \textbf{Scenario 3}---asymmetric intermediate dependence: Coles--Tawn model
  \end{footnotesize}
  \end{center}
  \vspace{-1cm}
    \begin{minipage}[c]{7cm}\hspace{1cm}
    \includegraphics[width=1\textwidth]{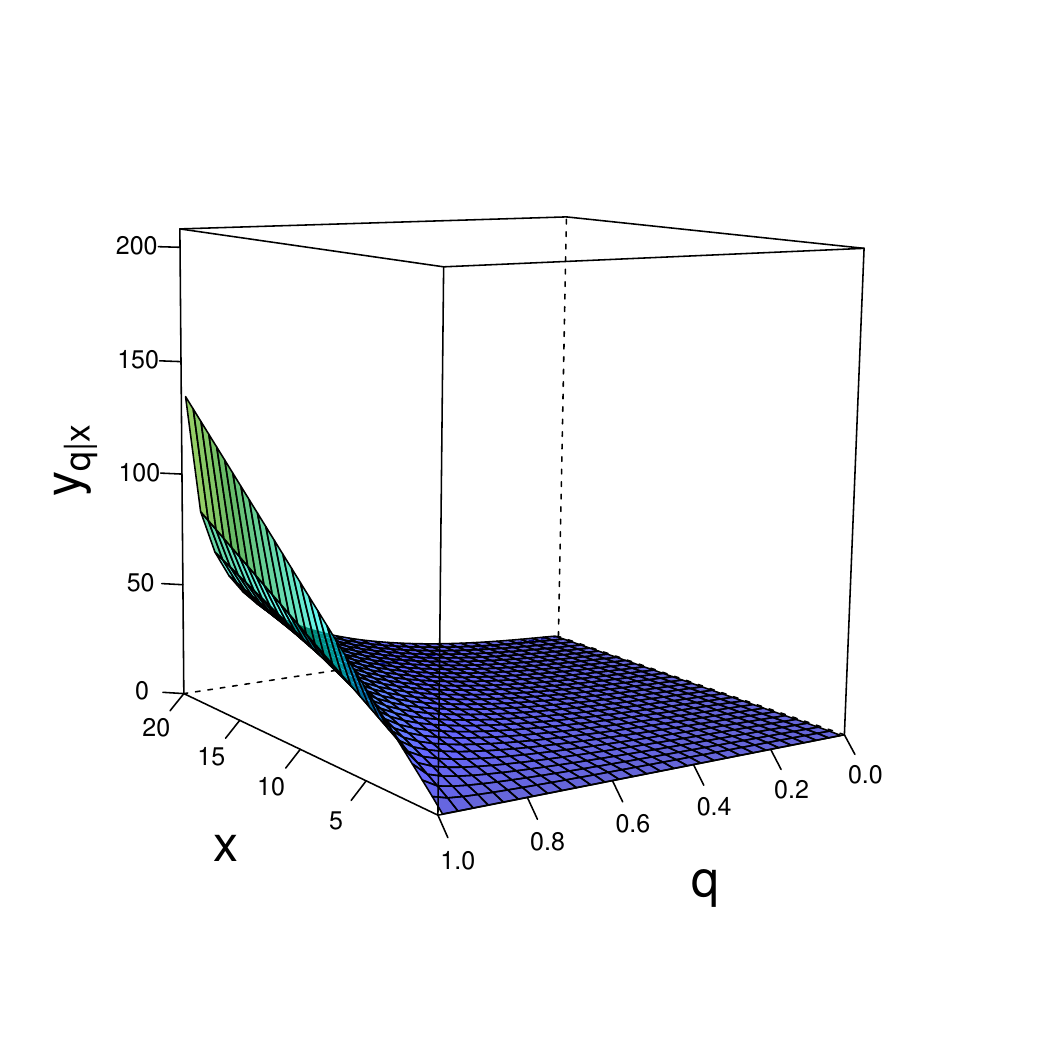}
    \text{\hspace{0cm}}
  \end{minipage} 
  \begin{minipage}[c]{7cm}
    \includegraphics[width=1\textwidth]{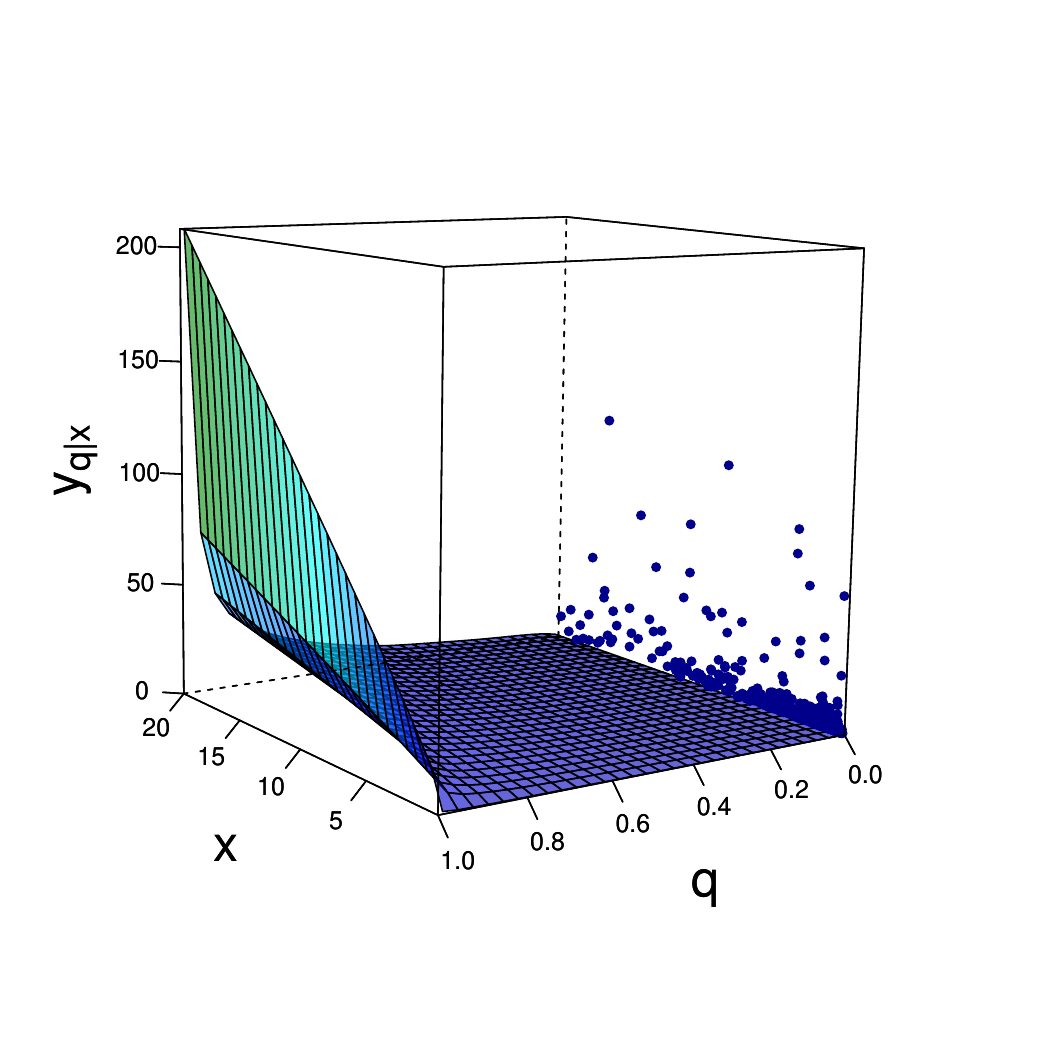}
  \end{minipage}
      \vspace{-1cm}
\end{center}
\caption{True regression manifold $\mathscr{L}$, as defined in \eqref{lines}, along with its posterior mean estimate obtained using the methods from Section~\ref{learn} for Husler--Reiss, Logistic, and Coles--Tawn bivariate extreme value models (top to bottom) on a single-run experiment. Simulated data are overlaid on one of the faces of the box.
}
\label{RLS}
\end{figure}

An induced prior for $g$ can be devised by plugging the approximation in \eqref{approxi} with the specification from Section~\ref{bern}, which leads to the following prior for the conditional multivariate extreme value distribution function, 
\begin{align}
\label{cooley-bernstein}
    G_{Y \mid \mathbf{X}} (y \mid \mathbf{x})
  \approx 
    \dfrac{\sum\limits_{|\alpha|=J} \pi_{\alphab} / B_{\alpha} \prod\limits_{i=1}^p x_i^{\alpha_i-1}  \int_0^y t^{\alpha_d-1} \lVert \mathbf{u}(t) \rVert^{-J-1} \, \dif t}{\sum\limits_{|\alpha|=J} \pi_{\alphab} / B_{\alpha} \prod\limits_{i=1}^p x_i^{\alpha_i-1} \int_0^{\infty} t^{\alpha_d-1} \lVert \mathbf{u}(t) \rVert^{-J-1} \, \dif t},
\end{align}
where $B_{\alpha} =  \prod_{i=1}^d \Gamma(\alpha_i) / \Gamma(|\alphab|)$ is the multivariate beta function. Hence, we can learn about the regression manifold $\mathscr{L}$ by estimating $\pi_{\alphab}$ as described in Section \ref{bern}, that is, by plugging in the Bernstein polynomial estimates \eqref{berndens} into the approximation \eqref{approxi} and numerically inverting  \eqref{cooley-bernstein} with respect to $y$. This strategy 
is illustrated numerically in the supplementary material.


\section{Simulation study}\label{simulation}

\subsection{Preliminary experiments}\label{prelim}
\label{one-shot}
\noindent We study the finite sample performance of the proposed methods under three data generating scenarios that were introduced in Section~\ref{model}; see Examples~\ref{logistic}--\ref{coles_tawn}. Specifically, we simulate data as follows:  
\begin{itemize}
     \item \textbf{Scenario 1}---strongly dependent extremes: Husler--Reiss model with $\lambda=0.1$.
    \item \textbf{Scenario 2}---weakly dependent extremes: Logistic model with $\alpha=0.9$.
    \item \textbf{Scenario 3}---asymmetric intermediate dependence: Coles--Tawn model with $\alpha = 0.5$, $\beta = 100$.
\end{itemize}
For now we focus on illustrating the methods in a single-run experiment; a Monte Carlo simulation study will be reported in Section~\ref{monte}. To illustrate how the resulting estimates compare with the true regression lines on a one-shot experiment, in each scenario we generate $n=5000$ samples $\{(X_i,Y_i)\}_{i=1}^n$. For the analysis we use observations for which $X_i+Y_i > u$, where $u$ is the $95\%$ quantile of the pseudo-radious, providing $k=250$ exceedances to fit the model. To learn about regression lines from data, we exploit the single component adaptive Markov Chain Monte Carlo (MCMC) with a wide Dirichlet prior, Dirichlet$(0.1 \times \ind_k)$, defined on a generalized logit transformation of weights $\pi_{\alphab}$. {The length of each MCMC chain is $10000$ 
with a burn-in period of $4000$. The multivariate effective sample sizes are $903$, $1121$, $1808$  for Scenarios 1,2,3 respectively.}

\par
In Figure~\ref{RLS} we plot true and estimated regression manifolds under the three scenarios above over the range $(x ,y) \in (0,20]\times(0,20]$, where $20$ corresponds to the $95\%$ quantile of the unit Fr\'echet marginal distributions.
Figure~\ref{RLS} shows that, for these one shot experiments, the proposed estimator recovers well the shape of $\mathscr{L}$ for all three cases, although as expected for $q$ closer to 0 and 1 there is some bias. Figure~\ref{RLS} also anticipates a feature that we will revisit in Section~\ref{monte}, i.e. that the case of weakly dependent extremes is more challenging---which is a consequence of the fact that it is more to challenging to learn about U-shaped angular densities from data. To have a closer look into the outputs from these  numerical experiments, we depict in Figure~\ref{CS} cross sections of the angular manifold, over $q$ and over $x$, thus leading to regression lines and conditional quantiles for the Husler--Reiss (top), Logistic (central), and Coles--Tawn (bottom) models. Once more, we see that the fits are fairly reasonable overall although a bit more of bias is visible for $q$ closer to 0 and 1. Interestingly, it can also be seen from Figure~\ref{CS} that regression lines are approximtely linear for the Logistic model, and we prove that this indeed the case for large $x$; see Appendix~D.


\renewcommand{\thefigure}{4.2}
\begin{figure}
    \begin{center}
  \begin{footnotesize}
    \textbf{Scenario 1}---strongly dependent extremes: Husler--Reiss model
  \end{footnotesize}
  \end{center}
  \vspace{0cm}
    \begin{minipage}[c]{7cm}
    \includegraphics[width=1\textwidth]{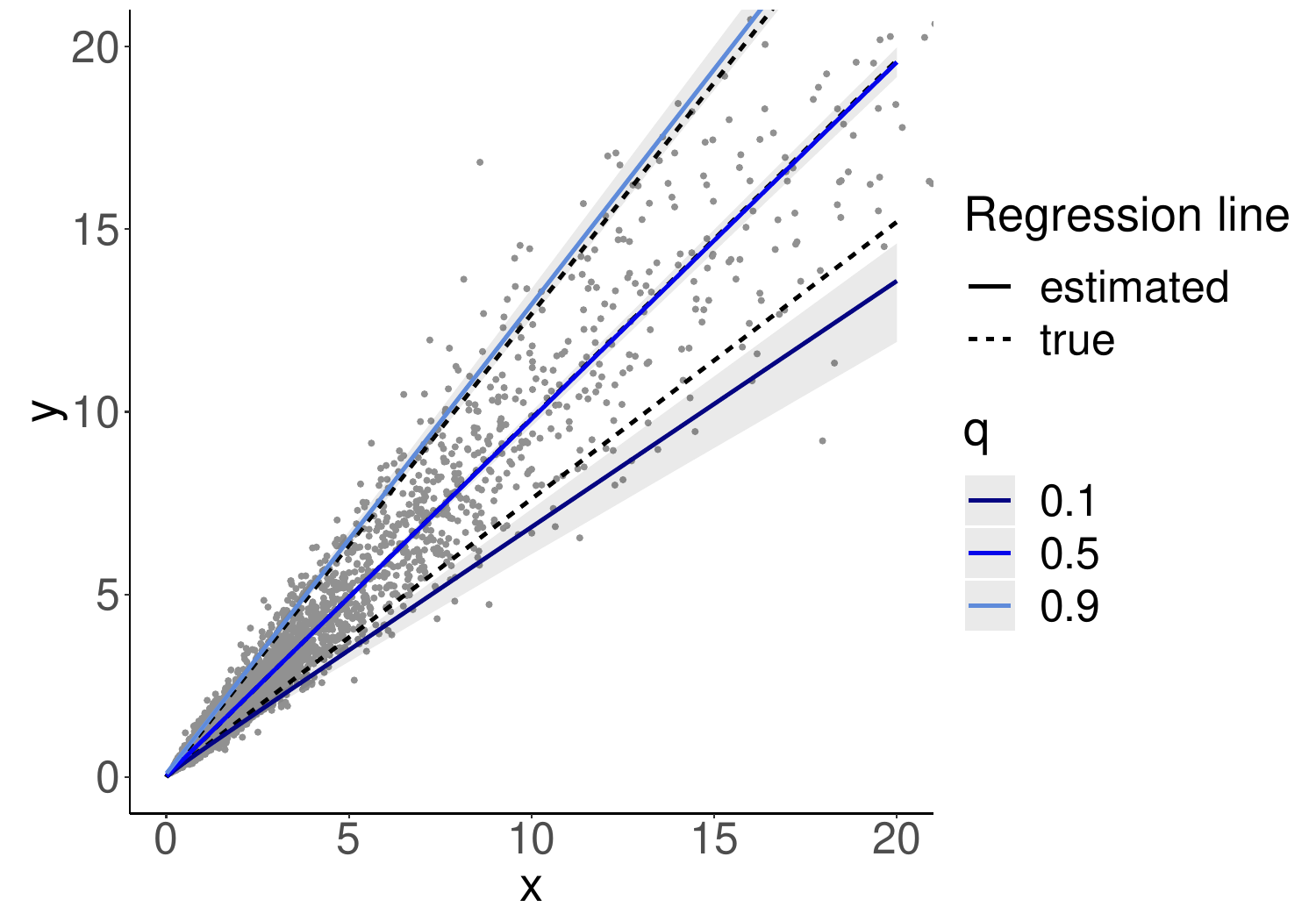}
    \text{\hspace{0cm}}
  \end{minipage} 
  \begin{minipage}[c]{7cm}
    \includegraphics[width=1\textwidth]{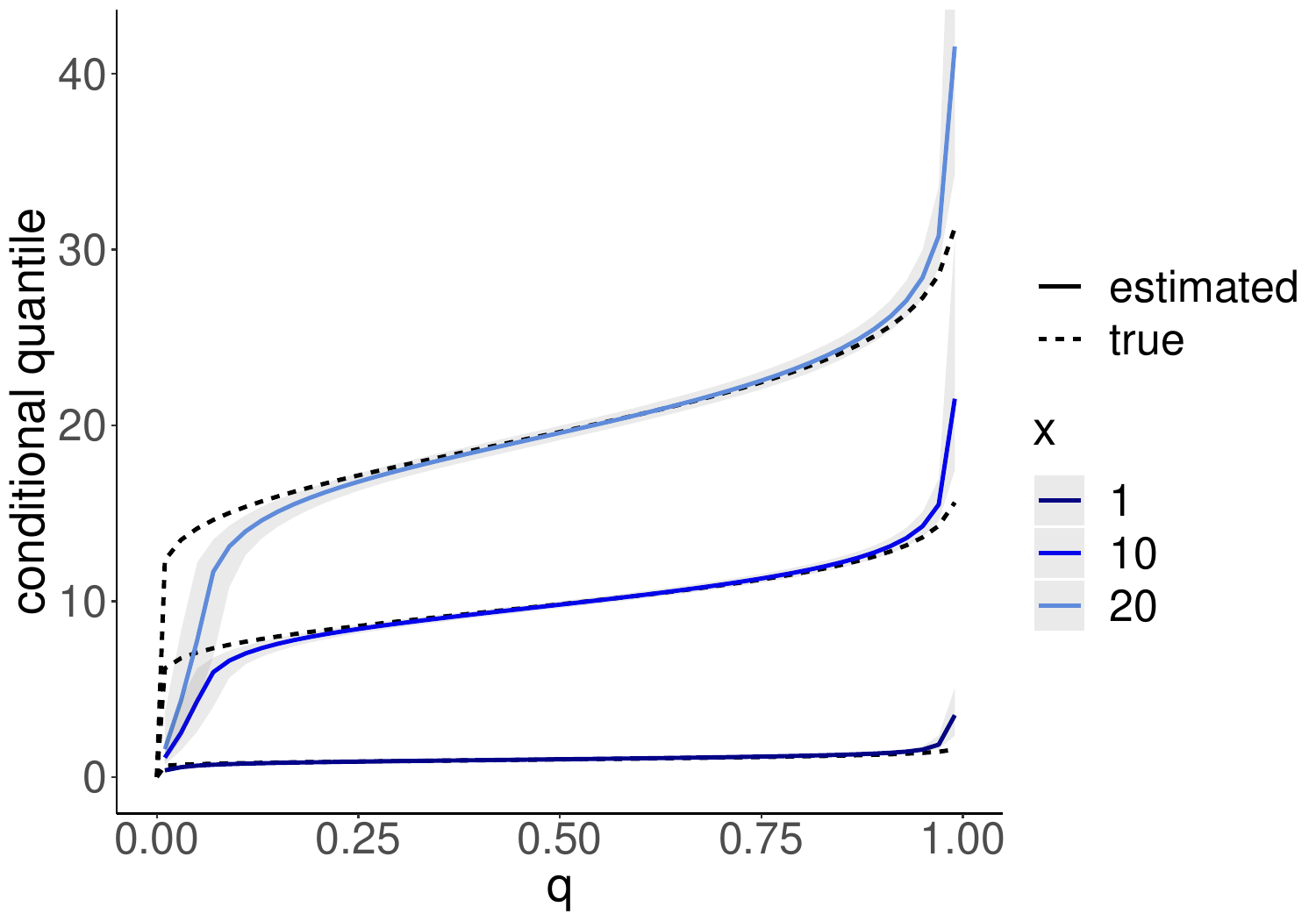}
  \end{minipage}
  \vspace{0cm}
    \begin{center}
  \begin{footnotesize}
    \textbf{Scenario 2}---weakly dependent extremes: Logistic model
  \end{footnotesize}
  \end{center}
  \vspace{0cm}
    \begin{minipage}[c]{7cm}
    \includegraphics[width=1\textwidth]{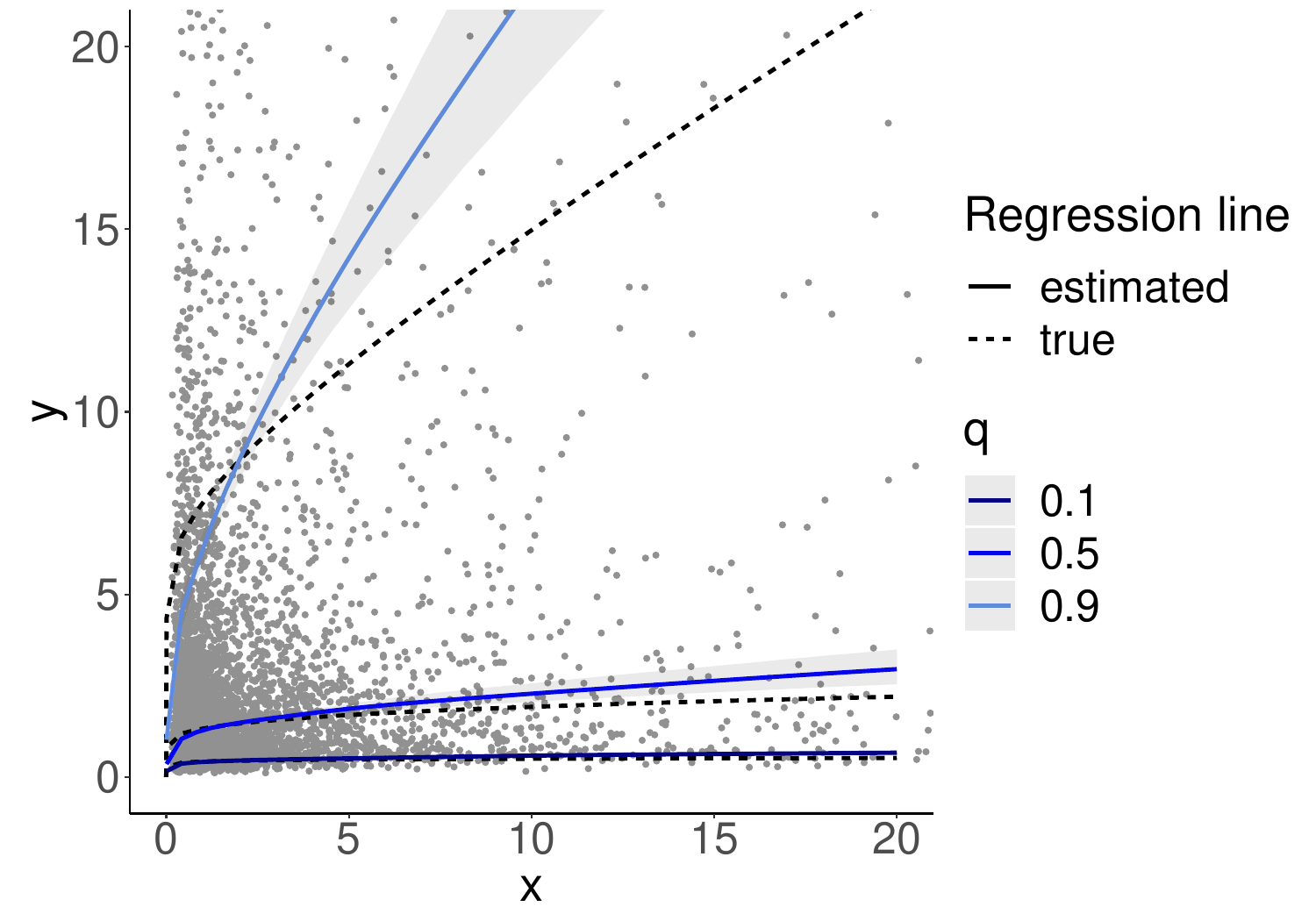}
    \text{\hspace{0cm}}
  \end{minipage} 
  \begin{minipage}[c]{7cm}
    \includegraphics[width=1\textwidth]{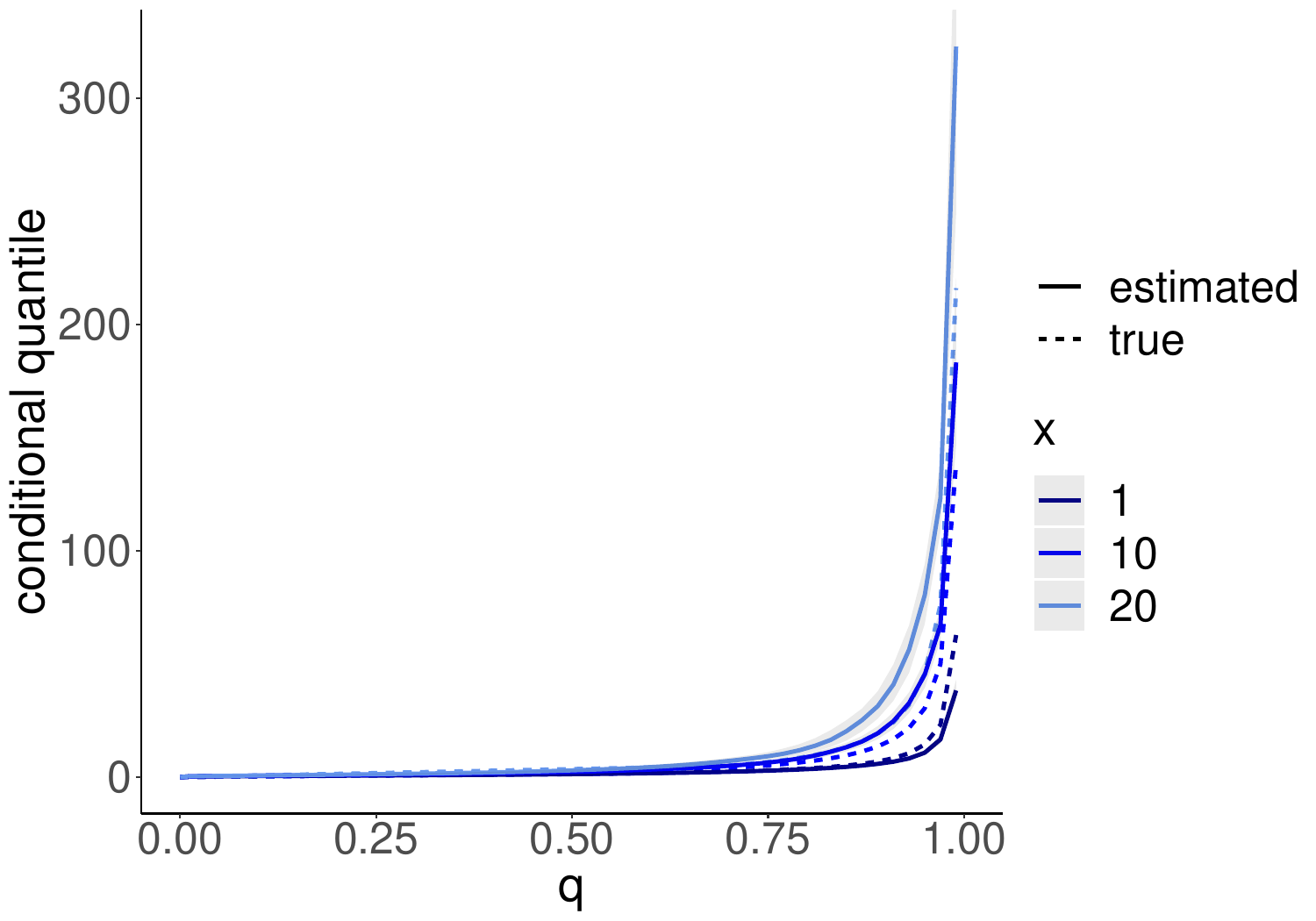}
  \end{minipage}
  \vspace{0cm}
   \begin{center}
  \begin{footnotesize}
    \textbf{Scenario 3}---asymmetric intermediate dependence: Coles--Tawn model
  \end{footnotesize}
  \end{center}
      \vspace{0cm}
    \begin{minipage}[c]{7cm}
    \includegraphics[width=1\textwidth]{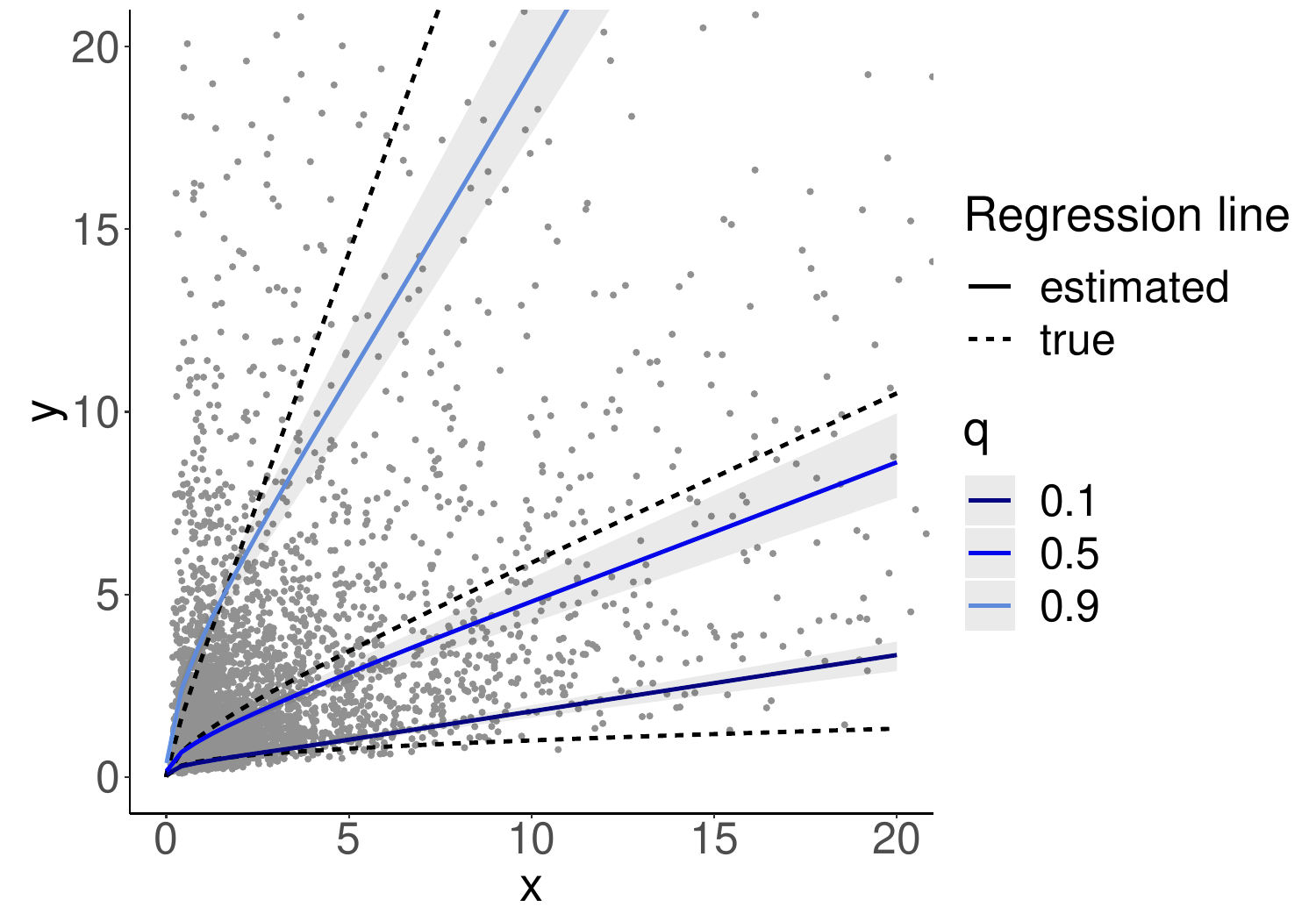}
    \text{\hspace{0cm}}
  \end{minipage} 
  \begin{minipage}[c]{7cm}
    \includegraphics[width=1\textwidth]{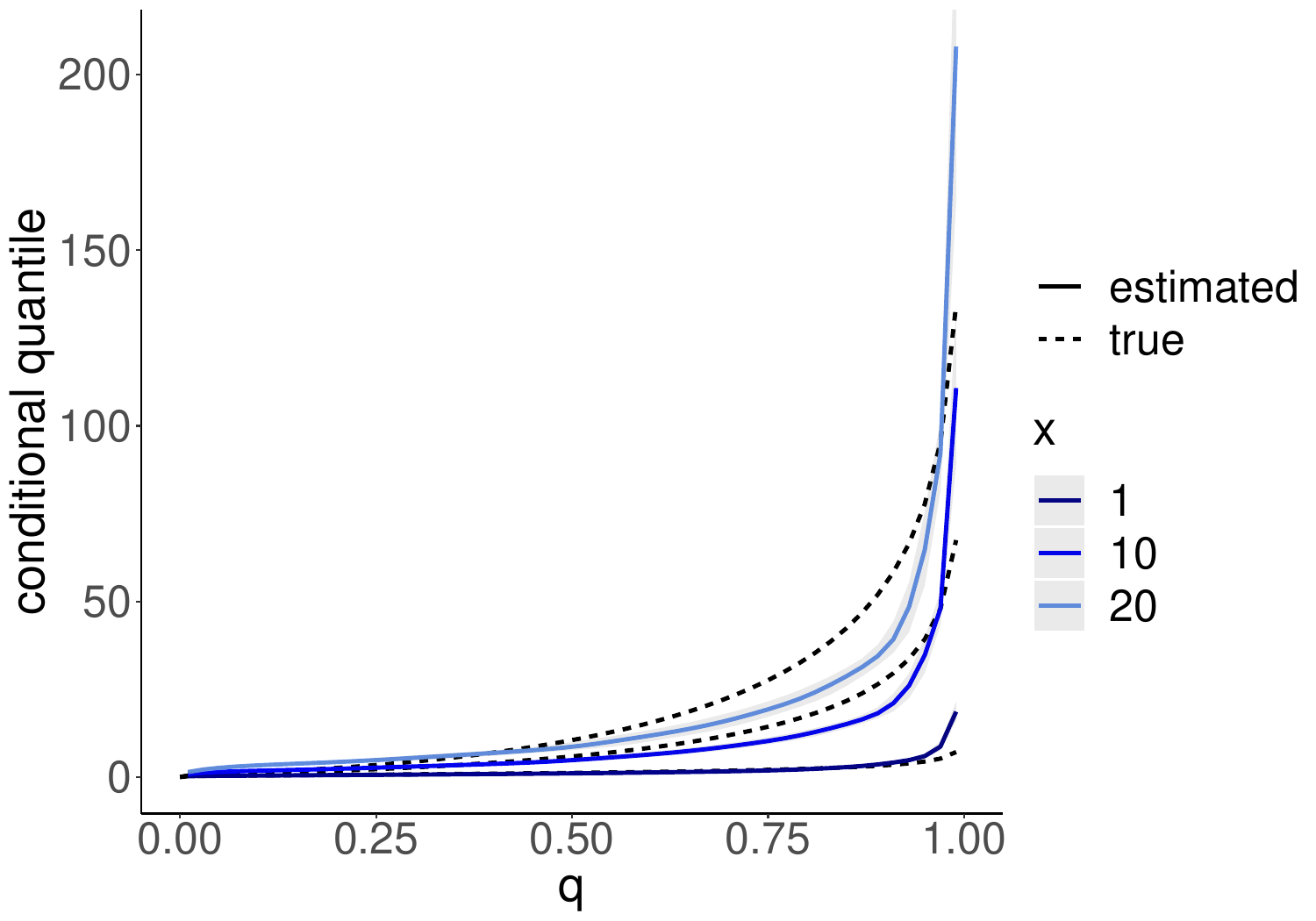}
  \end{minipage}
\caption{Posterior mean regression lines $L_q$ for $q=\{0.1,0.5,0.9\}$ and $x \in (0,20]$ (left) and conditional quantile curves $\{y_{q \mid x}:q\in(0,1)\}$ along with credible bands, for $x = \{1,10,20\}$ (right) for Husler--Reiss, Logistic, and Coles--Tawn bivariate extreme value models (top to bottom) on a single-run experiment.
} 
\label{CS}
\end{figure}

\renewcommand{\thefigure}{4.3}
\begin{figure}
  \centering
    \begin{center}
  \begin{footnotesize}
    \textbf{Scenario 1}---strongly dependent extremes: Husler--Reiss model
  \end{footnotesize}
  \end{center}
  \vspace{0cm}
  \hspace{-0.5cm}
    \includegraphics[width=1\textwidth]{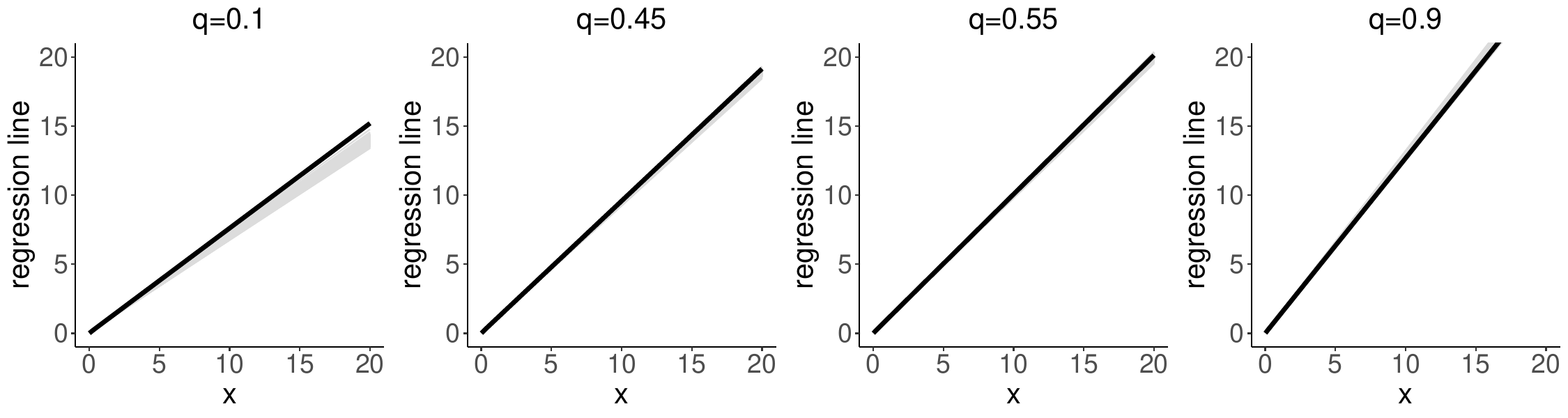}
  \vspace{-0.1cm}
    \begin{center}
  \begin{footnotesize}
    \textbf{Scenario 2}---weakly dependent extremes: Logistic model 
  \end{footnotesize}
  \end{center}
  \vspace{0cm}
  \hspace{-0.5cm}
    \includegraphics[width=1\textwidth]{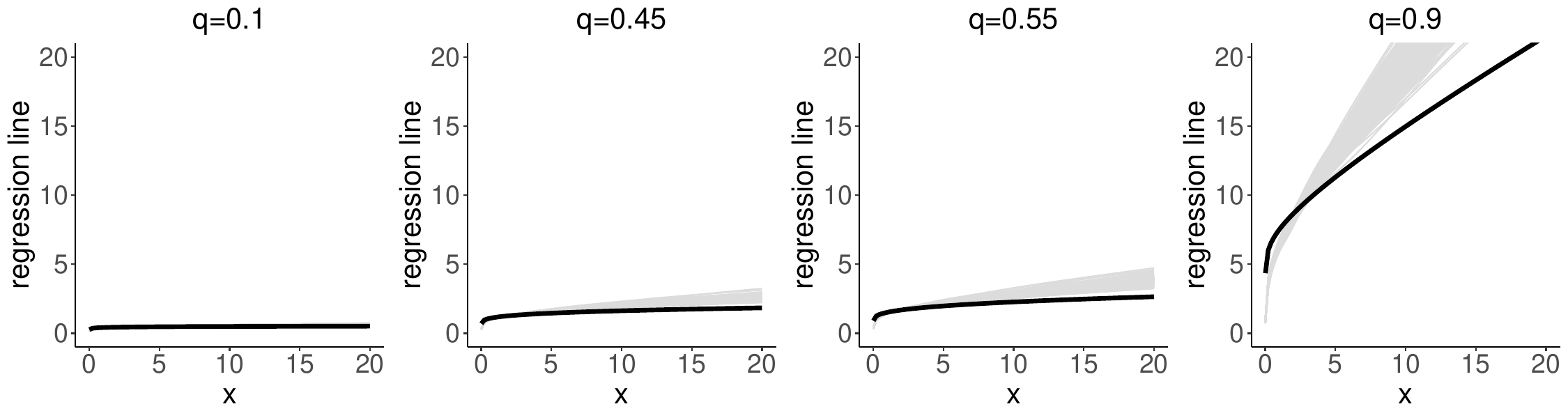}
  \vspace{-0.1cm}
   \begin{center}
  \begin{footnotesize}
    \textbf{Scenario 3}---asymmetric intermediate dependence: Coles--Tawn model
  \end{footnotesize}
  \end{center}
    \vspace{0cm}
  \hspace{-0.5cm}
    \includegraphics[width=1\textwidth]{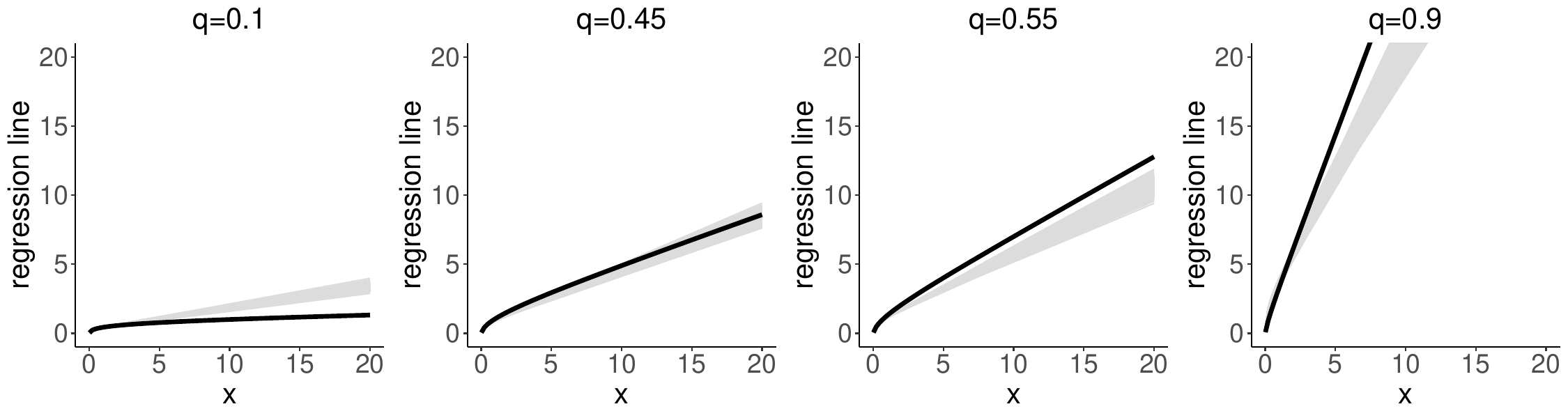}
\caption{Posterior mean regression lines $L_q$ for $q=\{0.1,0.45,0.55,0.9\}$ and $x \in (0,20]$ for each of the $500$ Monte Carlo samples ($k = 500$, gray lines) plotted against the true conditional quantiles (black line) for Husler--Reiss, Logistic, and Coles--Tawn bivariate extreme value models (top to bottom).}
\label{mc500}
\end{figure}

\subsection{Monte Carlo simulations}\label{monte}
To conduct a simulation study we generate $500$ Monte Carlo samples of sizes $n=5000$ and $n=10000$ resulting in $k=250$ and $k=500$ for the three scenarios described in Section~\ref{prelim}. We use the MCMC algorithm as described in Section \ref{one-shot} with the same prior on $\pi_{\alphab}$'s. The performance of our methods will be visualized via a comparison of posterior mean estimates of the regression lines with the true regression lines  $L_q$ for a few fixed $q \in (0,1)$. We focus on the region $x \in (0,20]$ as the bivariate extreme value  concentrates most of its mass (at least $90$\%) in the set $(0,20]\times(0,20]$. \par
The regression lines corresponding to the described scenarios for $k=500$ are shown in Figure~\ref{mc500}; a similar chart for $k = 250$ is available from the supplementary material (Figure~\ref{mc250}). Figure~\ref{mc500} outlines that the model fits the data from Scenario~1 reasonably well and, for weakly dependent extremes (Scenario~2) and asymmetrically dependent extremes (Scenario 3), it provides relatively precise estimates for middle values of $q$, but as expected it presents some bias for $q$ close to $0$ and $1$. Comparing different sample sizes (i.e.~comparing Figure~\ref{mc500} and Figure~\ref{mc250} in the supplementary material) we can observe that increasing sample size reduces the variation of estimates for all $q \in \{0.1,0.45,0.55,0.9\}$ and for all scenarios. \par 
{Overall, the difference in the performance for the considered scenarios is mainly related to the degree of association between extremes. For moderate-strongly dependent extremes with bell-shaped angular densities, as in Scenarios 1 and 3, we observe a reasonably good fit, whereas for weakly dependent extremes with $U$-shaped angular densities, as in Scenario 2, the estimates tend to be less accurate---as it is more challenging learning about the latter from data.}

\section{Application to stock markets}\label{application}
\subsection{Data, preprocessing, and applied rational for the analysis}
We now apply the proposed method to two of the world's biggest stock markets---the NASDAQ (National Association of Securities Dealers Automated Quotations) and NYSE (New York Stock Exchange). According to the Statistics Portal of the World Federation of Exchanges (\url{https://statistics.world-exchanges.org}), the total equity market capitalization of NASDAQ and NYSE are respectively 20.99 and 24.67 trillion US\$, as of 2021 / Apr, thus illustrating well the scale of these players in the worldwide stock-exchange industry. The data were gathered from Yahoo Finance (\url{https://finance.yahoo.com}), and consist of daily closing prices of the NASDAQ and NYSE composite indices over the period from February 5, 1971 to June 9, 2021.

A key goal of the analysis will be to learn about 
spillover between extreme losses in these markets 
through the lenses of our model, and thus we focus on modeling negative log returns, which can be regarded as a proxy for losses, and which consist of first differences of prices on a log-scale; the resulting sequence of $m$ componentwise weekly maxima losses  for NASDAQ and NYSE is denoted below as $\{(\mathcal{X}_i, \mathcal{Y}_i)\}_{i = 1}^m$. The sample period under analysis is sufficiently broad to cover a variety of major downturns and selloffs including, for example, those related with the 2007--2010 subprime mortgage crisis, the ongoing China--US trade war, and with the 2020 COVID-19 pandemic. We take weekly maxima of negative log returns 
and convert them to unit Fr\'echet margins via the transformation $(\widehat{X}_i, \widehat{Y}_i) = (- 1 / \log\{\widehat{F}_\mathcal{X}(\mathcal{X}_i)\}, - 1 / \log\{\widehat{F}_\mathcal{Y}(\mathcal{Y}_i)\}),$ where $\widehat{F}_\mathcal{X}$ and $\widehat{F}_\mathcal{Y}$ respectively denote the empirical distribution functions (normalized by $m + 1$ rather than by $m$ to avoid division by zero) of negative log returns for NASDAQ ($\mathcal{X}$) and NYSE ($\mathcal{Y}$); the supplementary material include the reverse analysis that swaps the roles of NASDAQ and NYSE (i.e.~NASDAQ becomes $\mathcal{Y}$ and NYSE becomes $\mathcal{X}$). The raw data and resulting preprocessed data are depicted in Figure~\ref{data_ang}. As can be seen from the latter figure the composite indices exhibit a similar dynamics reacting to different economic shocks (9/11 attacks, 2001; 2008 financial crisis; China-US trade war started in 2018) alike. Also, as can be seen from Figure~\ref{data_ang}, the shape of the scatterplot of the negative log returns brought to unit Fr\'echet margins in log-log scale above the boundary threshold evidences intermediate level of extremal dependence between negative log returns. 


\renewcommand{\thefigure}{5.1}
\begin{figure}[htbp]
  \centering \footnotesize 
  \begin{minipage}{1.0\linewidth}\centering
  (a)    
  \end{minipage}
  \includegraphics[scale = 0.38]{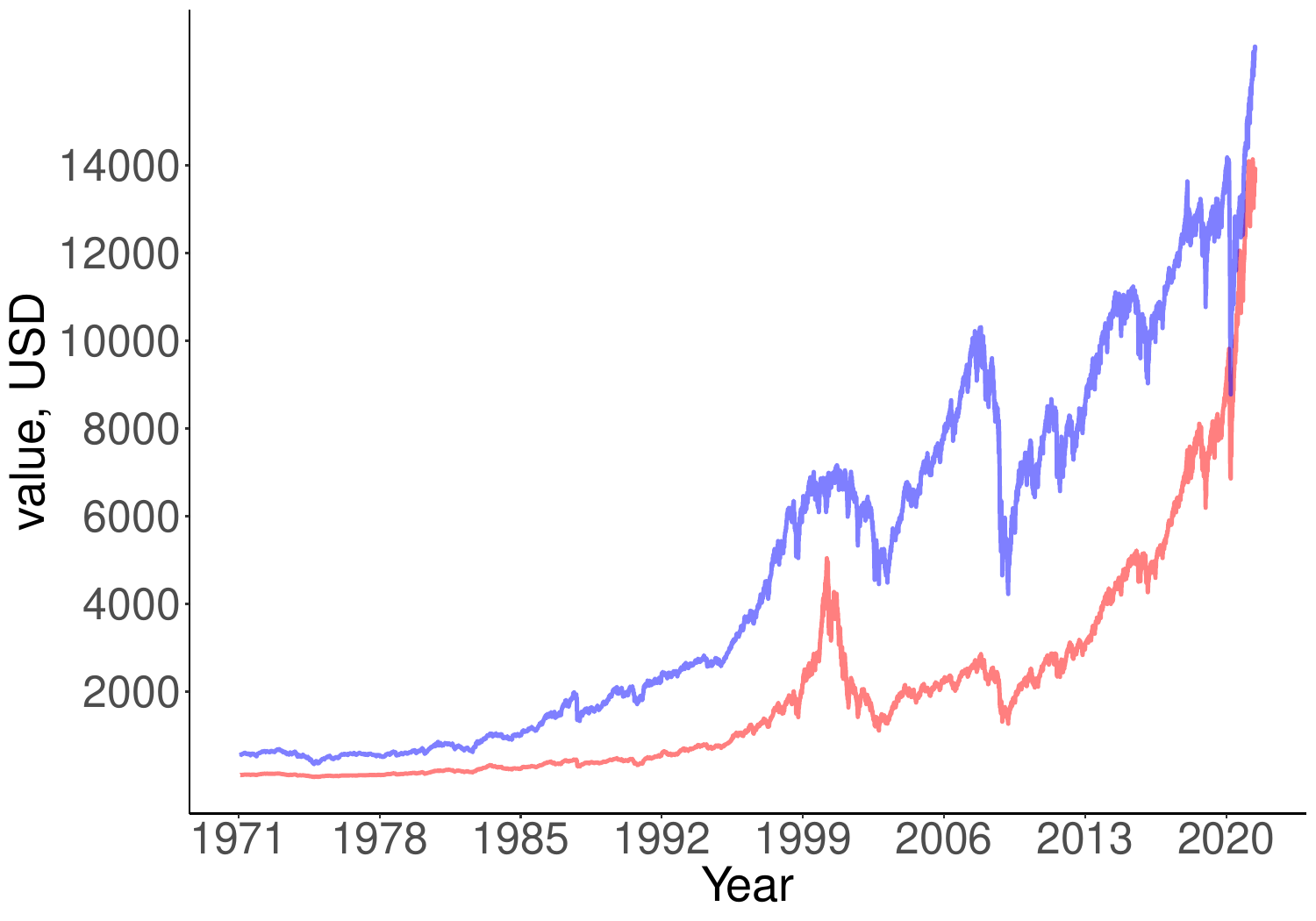}  \\ \ \\
  \begin{minipage}{0.48\linewidth}\centering
    (b)    
  \end{minipage}
  \begin{minipage}{0.48\linewidth}\centering
  (c)    
  \end{minipage}
  \includegraphics[scale = 0.3]{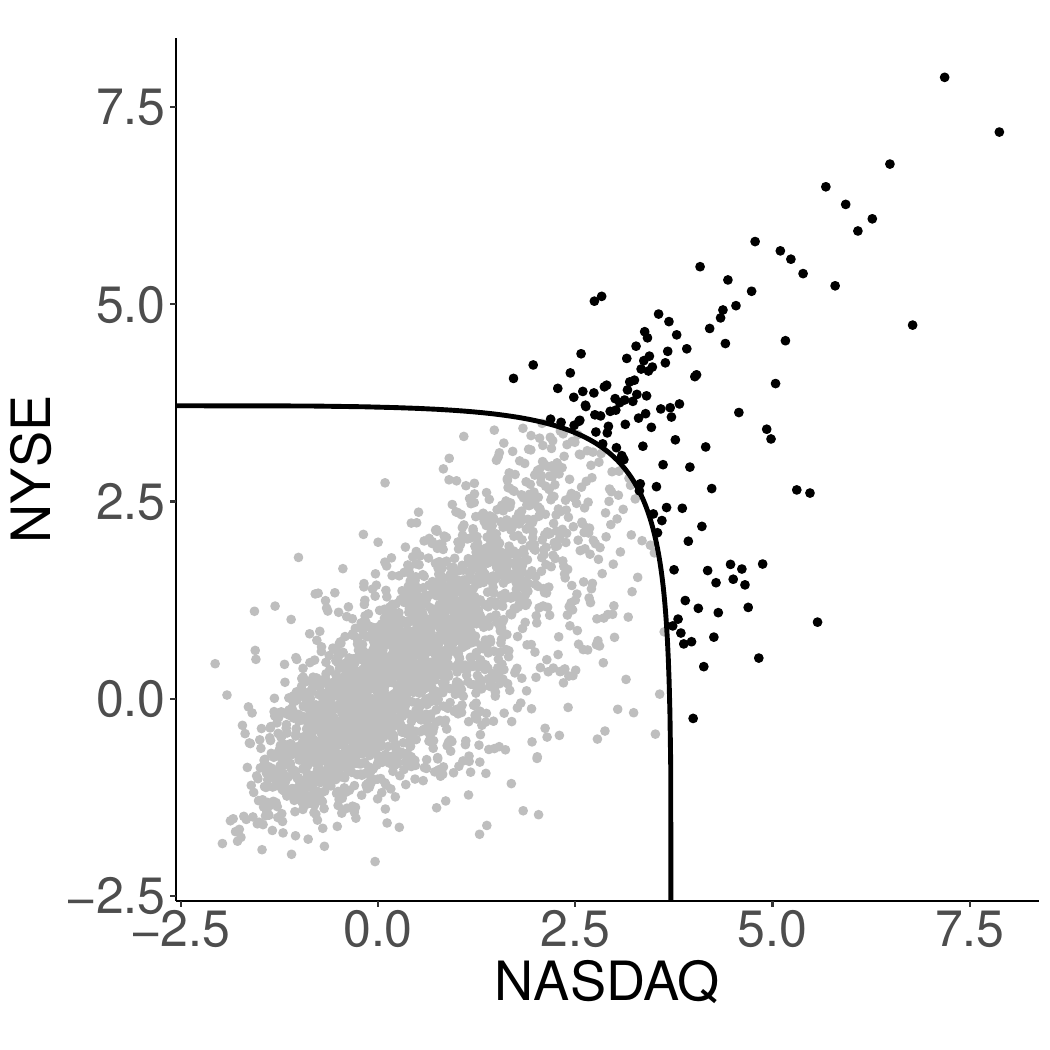}
  \includegraphics[scale = 0.3]{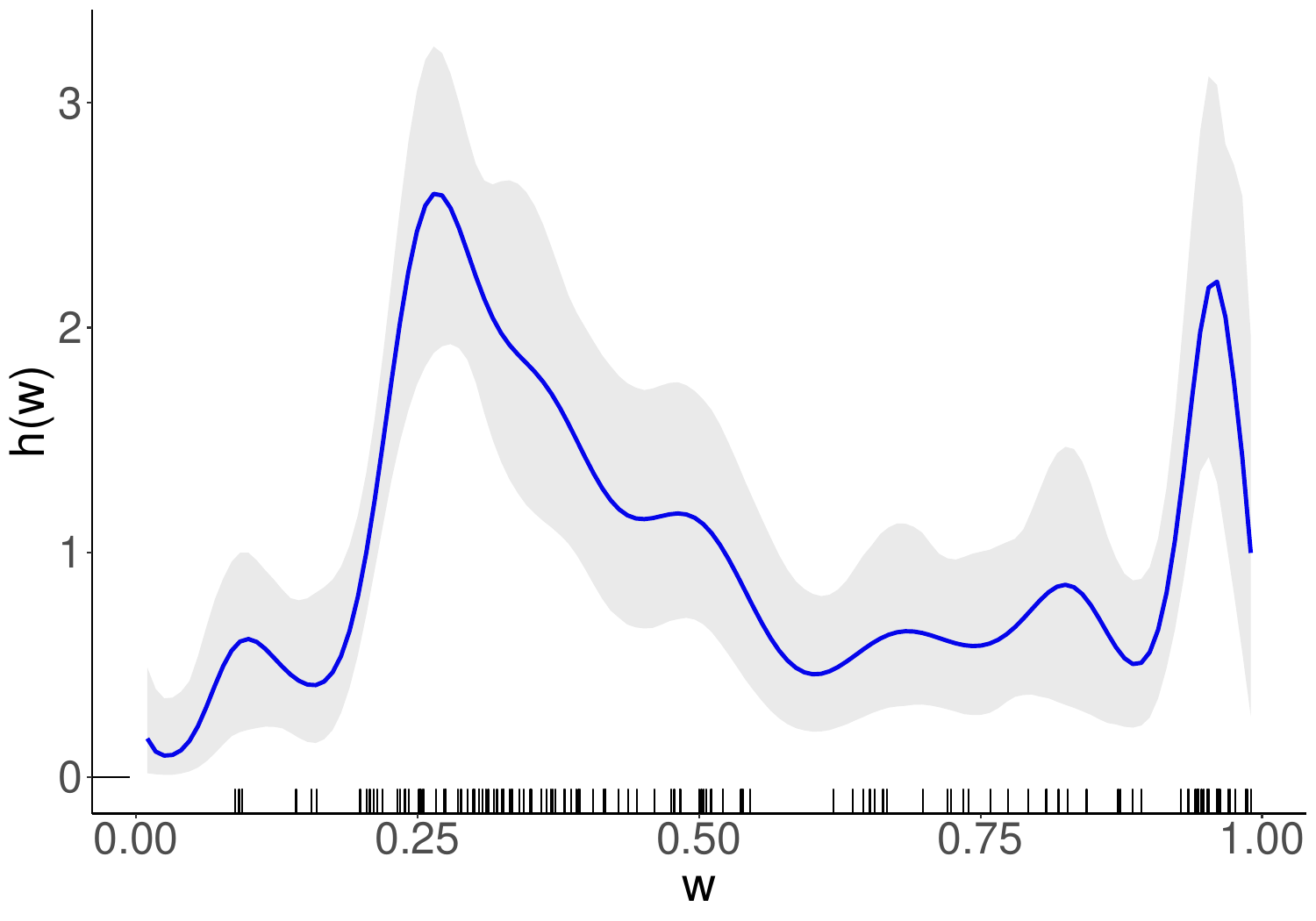}
  \caption{\label{data_ang} (a) NASDAQ (red) and NYSE (blue) composite indices. (b) Scatterplot of negative log returns of NASDAQ and NYSE composite indices converted to unit Fr\'echet margins; the solid line corresponds to the boundary threshold in the log-log scale, with both axes being logarithmic. (c) Angular density estimate with $95\%$ credible band along with a rug of pseudo-angles.}
\end{figure}

\subsection{Regression of extreme losses on extreme losses}
We now apply our model so to learn about how the extreme losses on both exchanges relate. To employ our model we start by fitting the angular density via the Bernstein polynomial-based approach from Section~\ref{bern} by using the pseudo-angles based on thresholding the pseudo-radius at their 95\% quantile. Some comments on prior specification and on posterior inference are in order. For our calculations we used a single component adaptive MCMC method with a wide Dirichlet prior defined on a generalized logit transformation of weights $\pi_{\alphab}$. We run a MCMC chain of length $25\,000$ with a burn-in period of $10\,000$ and set the number of basis functions to be equal to the number of exceedances. {The specified chain has the multivariate effective sample size of $1\,726\,657$.}

\renewcommand{\thefigure}{5.2}
\begin{figure}
  \centering \hspace{.2cm}
  \begin{minipage}{0.48\linewidth} \hspace{3.5cm} \footnotesize (a)
  \end{minipage}
  \begin{minipage}{0.48\linewidth} \hspace{3.3cm} \footnotesize (b)
    \end{minipage}\\ \vspace{-.5cm}
  \begin{minipage}{0.48\linewidth}\hspace{.3cm}
  \includegraphics[scale=0.43]{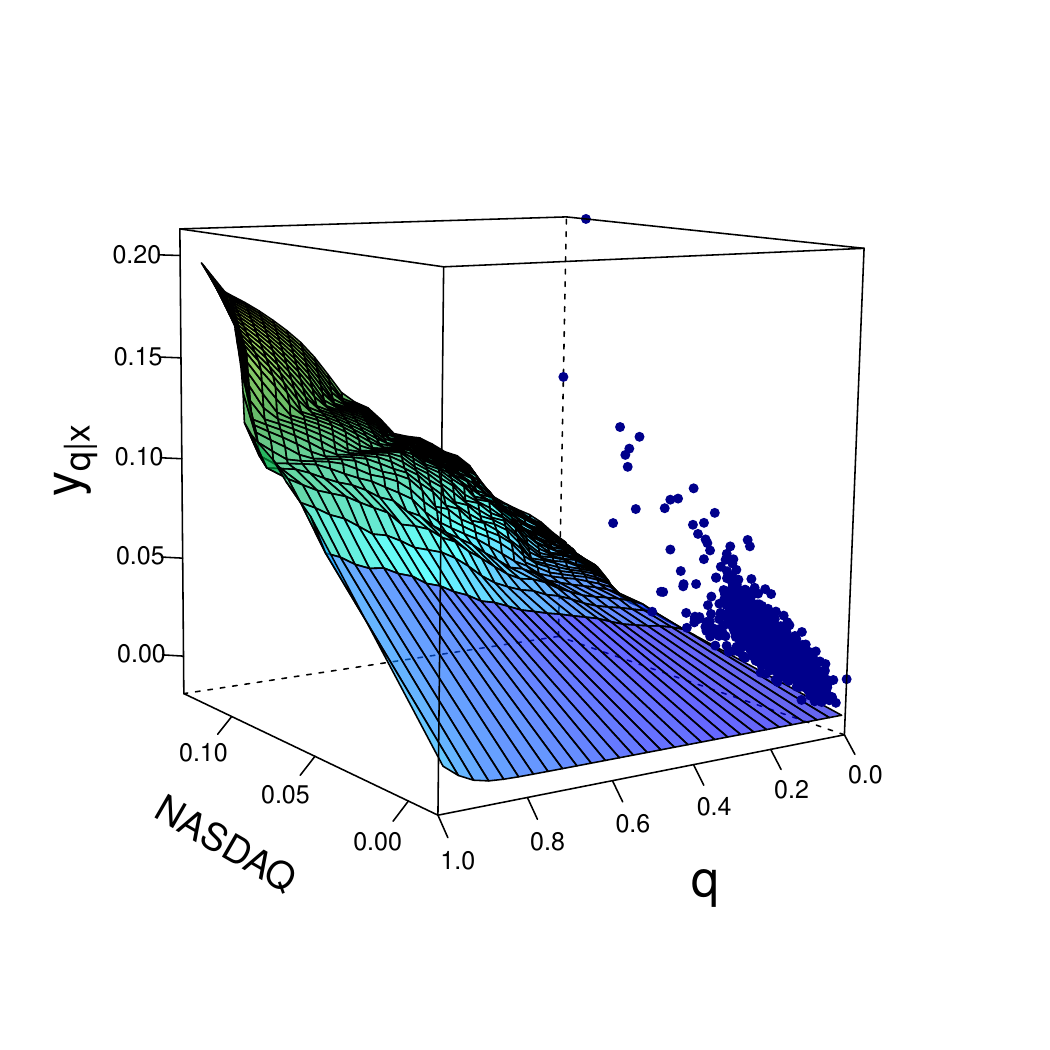}  
\end{minipage} 
\begin{minipage}{0.48\linewidth}\hspace{0.6cm}
  \includegraphics[scale = .34]{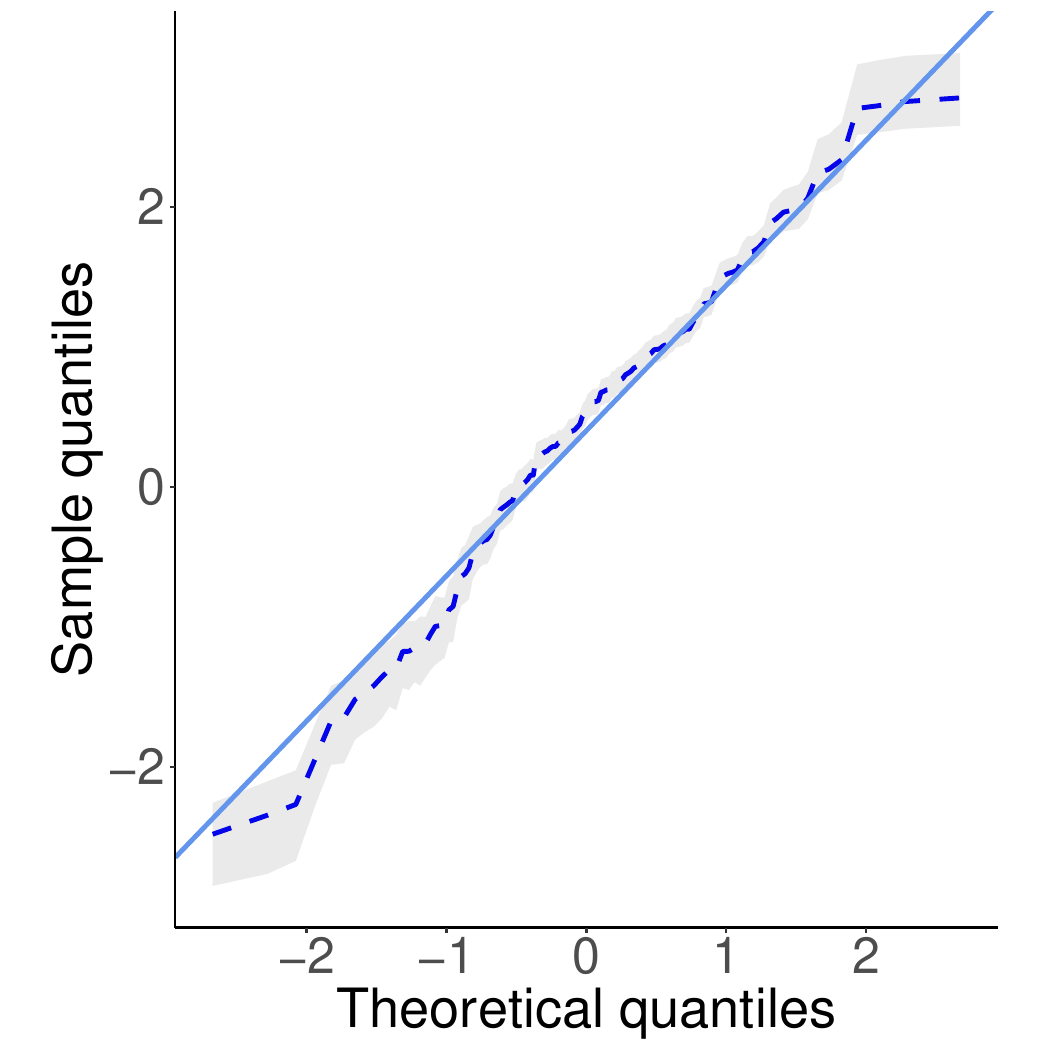}  
\end{minipage}
  \begin{minipage}{0.48\linewidth} \hspace{3.5cm} \footnotesize (c)
  \end{minipage}
    \begin{minipage}{0.48\linewidth} \hspace{3.3cm}  \footnotesize (d)
    \end{minipage}\\ 
  \begin{minipage}{0.48\linewidth} \hspace{-.4cm}    
    \includegraphics[scale=0.34]{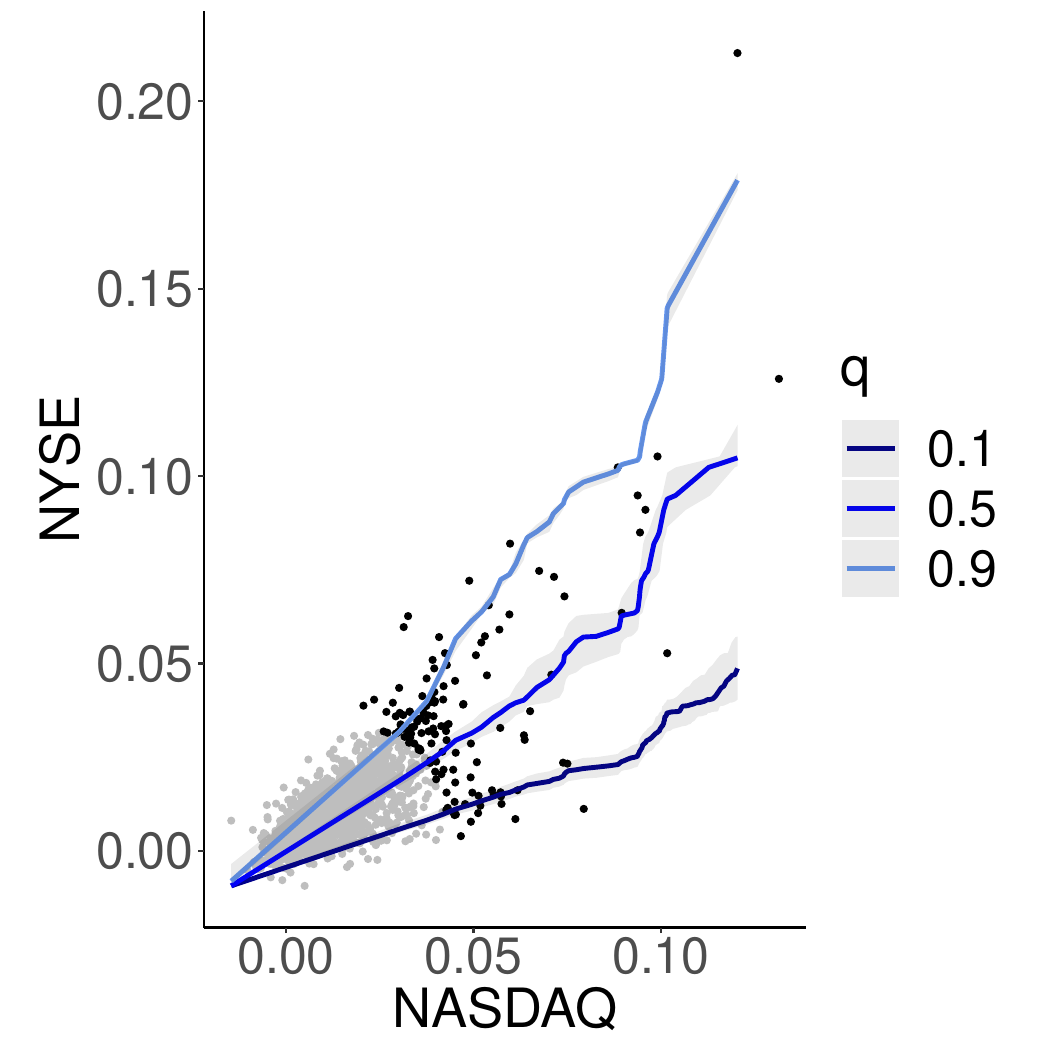} 
  \end{minipage}    
  \begin{minipage}{0.32\linewidth}\hspace{-.4cm}    
    \includegraphics[scale=0.34]{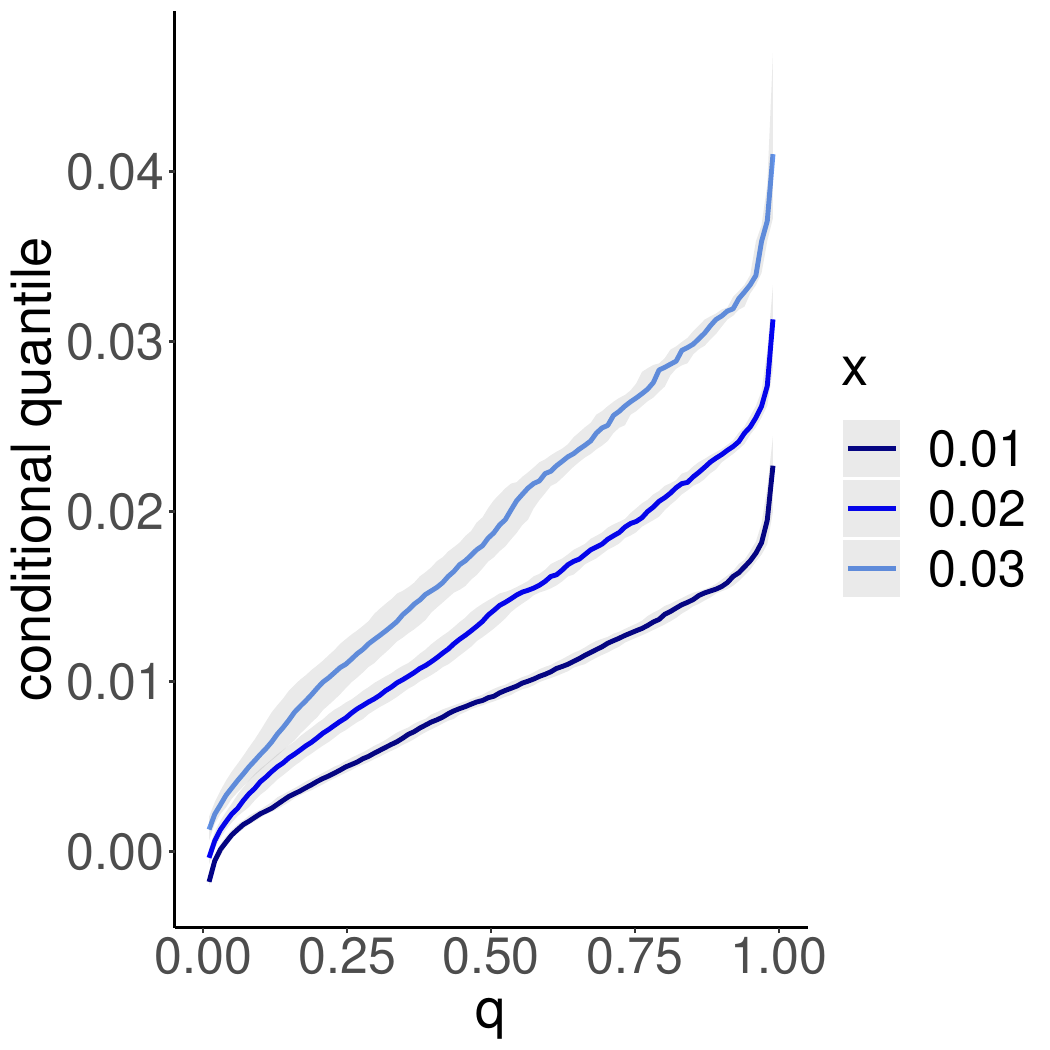}
  \end{minipage}
  \caption{\label{data_RL} (a) Posterior mean regression manifold $\mathscr{L}$ for NYSE given NASDAQ along with joint negative log returns overlaid on one of the faces of the box. (b) QQ-plot of randomized quantile residuals; the dashed line represents the posterior mean plotted along with credible bands. (c) Posterior mean  regression lines $L_q$ for $q=\{0.1,0.5,0.9\}$ for NYSE given NASDAQ along with $95\%$ credible bands and plotted against joint negative log returns. (d) Posterior mean conditional quantile curves $\{y_{q \mid x}:q\in(0,1)\}$ of negative log returns on NYSE for $x=\{0.01,0.02,0.03\}$, along with $95\%$ credible bands,  corresponding to negative log returns on NASDAQ in the original margins.}
\end{figure}


The obtained fit for the angular density is reported in Figure~\ref{data_ang} (right). As is illustrated by this plot most of the observed pseudo-angles lie closer to the middle of the interval $(0,1)$ and the estimate resembles a bell-shaped right-skewed density which suggests there is an asymmetric intermediate dependence between extremal losses on NASDAQ and NYSE composite indices.
Next we learn about the regression manifold. 
{Figure~\ref{data_RL} (a) represents the resulting estimates of the regression manifold together with cross-sections in $q$ and $x$, respectively in (b) and (d) for negative log-returns on NASDAQ and NYSE composite indices in the original margins. The regression manifold is highly non-linear and the regression lines on the middle graph substantially differ from those corresponding to independence and tend to be closer to the identity line. Moreover, the cross-sections for different values of $x$ reveal considerable variation in quantiles of $y$ supporting the conclusion about presence of the dependence between negative log-returns.} {To assess the quality of the fitted regression manifold we depict in Figure~\ref{data_ang} (b) a QQ-plot of a version of \cite{dunn1996} randomized quantile residuals adapted to our model, defined as $\varepsilon_{i} = \Phi^{-1}(G_{H}(Y_i \mid X_i))$, for $Y_i + X_i > u$, with $u$ denoting the 95\% quantile of the pseudo-radius. The latter chart depicts randomized quantile residuals against the theoretical standard Normal quantiles, and it suggests an acceptably good fit of the proposed model.}



{Having evaluated the regression manifold, we are now ready to examine by how much the NYSE can plummet, when the NASDAQ plummets. To examine this, we report in Table~\ref{quant_orig} predicted $75\%$, $90\%$ and $95\%$ quantiles of losses on NYSE evaluated for $1\%$, $2\%$ and $3\%$ weekly maxima losses on NASDAQ. This table follows from the regression manifold, and its interpretation is as follows. First, from a qualitative viewpoint, Table~\ref{quant_orig}  indicates that whenever the NASDAQ plummets, the NYSE tends to plummet reasonably by the same amount. Second---and more interesting from a financial outlook---are the quantitative claims that came be made from the analysis. For example, Table~\ref{quant_orig} indicates that whenever there is a $1\%$ weekly maximum loss on the NASDAQ, only in 5\% of the times we expect to suffer a loss in the NYSE above $1.71\%$. As another example, Table~\ref{quant_orig} indicates that whenever there is a $3\%$ weekly loss on the NASDAQ, only in 5\% of the times we expect to suffer a loss in the NYSE above $3.33\%$.}

\begin{table} \caption{Predicted $75\%$, $90\%$ and $95\%$ quantiles of losses on NYSE evaluated for $1\%$, $2\%$ and $3\%$ weekly maxima losses on NASDAQ, with $95\%$ credible intervals in brackets; negative log-returns used as proxy for losses}            
	\centering    
	\begin{threeparttable}
		\begin{tabular}{c c c c}  
			\hline 			\hline
			\multirow{2}{*}{NYSE} &
			\multicolumn{3}{c}{NASDAQ} \\ [0.5ex]
			\cline{2-4}
			 & $0.01$ & $0.02$ & $0.03$\\ [0.5ex]
			\hline 
			$75\%$ & $0.0129$ & $0.0194$ & $0.0266$ \\
			& \footnotesize $(0.0125,0.0134)$ &  \footnotesize $(0.0187,0.0201)$ & \footnotesize $(0.0259,0.0274)$ \\
			$90\%$ & $0.0156$ & $0.0233$ & $0.0315$  \\
			& \footnotesize $(0.0153,0.0159)$ &  \footnotesize $(0.0229,0.0237)$ & \footnotesize $(0.0309,0.0319)$ \\
			$95\%$ & $0.0171$ & $0.0249$ & $0.0333$
			\\
			& \footnotesize $(0.0166,0.0177)$ &  \footnotesize $(0.0245,0.0255)$ & \footnotesize $(0.0328,0.0339)$ \\
			[1ex]
			\hline 			\hline
		\end{tabular}
		\label{quant_orig}   
	\end{threeparttable}
\end{table}


\section{Closing remarks}\label{discussion}
We propose a regression-type model for the setup where both the response and the covariate are block maxima. The modeling starting point is the result that the limiting behavior of the vector of properly standardized componentwise maxima is given by a multivariate extreme value distribution. Conceptually, the model is then constructed in a similar fashion as in quantile regression, that is, by assessing how the conditional quantile of the response reacts to changes in the covariate while it takes into account the latter asymptotic result.
An important target in the proposed framework is the regression manifold, which consists of a family of regression lines obeying the proviso of multivariate extreme value theory. A Bernstein polynomial prior on the space of angular densities is used to learn about the model from data, with numerical studies showcasing its flexibility. 

One could wonder why not to resort to statistical models for nonstationary extremes \citep[e.g.][Section~6]{coles2001} as an alternative to methods proposed herein, as these can be used for assessing the effect of covariates on an extreme-valued response, by indexing the parameters of the GEV distribution with a covariate. Yet, since the latter models are built from the univariate theory of extremes they are not tailored for conditioning on another variable being extreme, as they fail to take on board information from the dependence structure between the extremes. Other related approaches include extremal quantile regression methods  \citep{chernozhukov2005}---which similarly to the statistical models for nonstationary extremes---have not been designed for conditioning on another variable being extreme, as they do not take into account the dependence structure between the extremes.


While not explored here, the comparison of the fitted models for both $Y \mid X = x$ and $X \mid Y = y$,  would look natural for some applied settings of interest so to get an idea of cause and effects, and indeed related ideas are analyzed by \cite{mhalla2020}. Finally, we close the paper with some comments on future research. For regressions with many predictors, it is likely that most covariates will have little effect on the response and thus one could wonder how to devise a version of the proposed method that shrinks towards zero the effect of such irrelevant covariates; the development of a Lasso \citep{tibshirani1996} version of the proposed model would thus seem natural for such situation, and is left as an open problem for future research. Another natural avenue for future research would be to devise regression-type methods for exceedances on exceedances by resorting to the so-called multivariate generalized Pareto distribution \citep{kiriliouk2019}, rather than with the multivariate extreme value distribution as herein. Finally, the development of a version of the model that could take into account asymptotic independence by resorting to the hidden angular measure \citep{ramos2009}, rather than the standard angular measure as herein, would seem natural as well. 

\section*{Appendix}\section*{\textit{Appendix~A: Conditional bivariate extreme value distribution}}\label{CondDistr}
Here we derive the expression for the conditional bivariate extreme value distribution function in \eqref{condD}. Sklar's theorem \cite[Theorem 2.3.3]{nelsen2006}, implies that a joint bivariate distribution function $G:\mathbb{R}^2 \to [0,1]$ with continuous marginal distributions $G_X:\mathbb{R} \to [0,1]$ and $G_Y:\mathbb{R} \to [0,1]$ can be uniquely represented through a copula function $C$ for $(x,y) \in \mathbb{R}^2$
$
G(x,y) = C(G_X(x),G_Y(y)),
$
or, equivalently,
$
C(u_1,u_2) = G(G^{-1}_X(u_1), G^{-1}_Y(u_2))$, for $(u_1,u_2) \in [0,1]^2$,
where $G^{-1}_X(q) = \inf \left\{ x: G_X(x) \geq q \right\}$. Using the following well-known property of copulas, 
\begin{align*}
C_{U_2\mid U_1}(u_2\mid u_1) \defeq \prob \left( U_2 \leq u_2 \mid U_1=u_1 \right) 
=
\partder{C(u_1,u_2)}{u_1}, \quad (u_1,u_2) \in [0,1]^2,
\end{align*}
we calculate the conditional distribution $(Y \mid X)$ as
\begin{align*}
G_{Y\mid X} (y\mid x) 
=
C_{U_2\mid U_1}(e^{-1/y}\mid e^{-1/x}).
\end{align*}
In our setting
\begin{align*}
	C(u_1,u_2) 
	&= 
	\exp\left[ - 2 \int_0^1 \max \{- w \log u_1,- (1 - w) \log u_2 \} \, \dif H(w) \right].
\end{align*}
Assuming $H$ is absolutely continuous with density $h$, we have
\begin{align*}
    &\partder{}{x} 2\int_0^1 \max \left( \dfrac{w}{x}, \dfrac{1-w}{y} \right) \, \dif H(w)
    =
    2\partder{}{x} \left( x^{-1}\int_{\omega(x, y)}^1 w \, \dif H(w) + y^{-1}\int_0^{\omega(x, y)} (1-w) \, \dif H(w) \right) 
     \\
     &\quad 
     =
     2 \left[ -x^{-2} \int_{\omega(x, y)}^1 w \, \dif H(w) -
     (x^{-1}xy - y^{-1}y^2)
     \dfrac{1}{(x+y)^3} h\{\omega(x, y)\}
     \right]
    =
     -2 x^{-2} \int_{\omega(x, y)}^1 w h(w) \, \dif w, 
\end{align*}
where $\omega(x, y) = x / (x + y)$. Then, the conditional copula has the following form
\begin{align*}
	C_{U_2\mid U_1}(u_2\mid u_1) 
	&=
	2 u^{-1}_1 C(u_1,u_2) \int\limits_{\dfrac{\log u_2}{\log u_1u_2}}^1 w h(w)  \, \dif w,
\end{align*}
which in turn yields
\begin{align*} 
    G_{Y|X}(y \mid x)
    =
    2
    \exp \left\{ - 2 \int_0^1 \max \left( \dfrac{w}{x}, \dfrac{1-w}{y}  \right) h(w)  \, \dif w + x^{-1} \right\} \int_{w(x, y)}^1 w h(w)  \, \dif w,
    \quad x,y>0.
\end{align*}


\section*{\textit{Appendix~B: Soft-maximum approximation for regression manifold of perfectly dependent extremes}}
\label{perfdep}
Here we give details on the soft maximum approximation for the regression lines for perfectly dependent extremes claimed in \eqref{softy}. We use a smooth approximation of a maximum function called \textit{soft-maximum}, 
\begin{align*}
    f(z_1, \dots, z_d;N) 
    =
    \dfrac{1}{N}\log(e^{N z_1} + \dots + e^{N z_d}),
\end{align*}
which is infinitely differentiable everywhere and converges to the maximum function as $N \to \infty$ \citep{cook2011}. Then, the approximation of a multivariate GEV distribution function for the case of perfect dependent extremes, $G(y, \mathbf{x}) = \max \{ y^{-1}, x_1^{-1}, \dots, x_p^{-1} \}$ is
\begin{align*}
    \Tilde{G}(y, \mathbf{x};N)
    = 
    \exp\left\{ -
    \dfrac{1}{N}\log( e^{N y^{-1}} + e^{N x_1^{-1}} + \dots + e^{N x_p^{-1}}) 
    \right\}
    =
    ( e^{N y^{-1}} + e^{N x_1^{-1}} + \dots + e^{N x_p^{-1}})^{-1/N},
\end{align*}
and its partial derivative of order $d$ is
\begin{align*}
    \Tilde{g}(y, \mathbf{x};N)
    &= 
    \dfrac{\partial^d}{\partial y \partial x_1 \cdots \partial x_p} \Tilde{G}(y, \mathbf{x};N)\\
    &=
    y^{-2} \prod\limits_{i=1}^p (1 + iN) x_i^{-2}
    \exp\left( Ny^{-1} + \sum\limits_{i=1}^p x_i^{-1} \right)
    ( e^{N y^{-1}} + e^{N x_1^{-1}} + \dots + e^{N x_p^{-1}} )^{-1/N - d}.
\end{align*}
This yields the following approximation of the conditional multivariate GEV density for perfectly dependent extremes, 
\begin{align*}
    \Tilde{g}_{Y\mid \mathbf{X}} (y \mid \mathbf{x}; N)
    &=
    \dfrac{y^{-2} \prod\limits_{i=1}^p (1 + iN) x_i^{-2}
    \exp\left( Ny^{-1} + \sum\limits_{i=1}^p x_i^{-1} \right)
    ( e^{N y^{-1}} + e^{N x_1^{-1}} + \dots + e^{N x_p^{-1}})^{-1/N - d}}{\prod\limits_{i=1}^{p-1} (1 + iN) \prod\limits_{i=1}^{p} x_i^{-2}
    \exp\left( \sum\limits_{i=1}^p x_i^{-1} \right)
    ( e^{N x_1^{-1}} + \dots + e^{N x_p^{-1}})^{-1/N - p}}\\
    &=
    (1+pN) y^{-2} e^{Ny^{-1}}
     \frac{( e^{N y^{-1}} + e^{N x_1^{-1}} + \dots + e^{N x_p^{-1}})^{-1/N - d} }{( e^{N x_1^{-1}} \dots + e^{N x_p^{-1}})^{-1/N - p}},
\end{align*}
and the following approximation for the corresponding conditional cumulative distribution function
\begin{align*}
    \Tilde{G}_{Y\mid \mathbf{X}} (y \mid \mathbf{x} ; N)
    &=
    \int_0^y \Tilde{g}_{Y\mid \mathbf{x}} (z \mid \mathbf{x})  \, \dif z\\
    &=
      (1+pN) ( e^{N x_1^{-1}} + \dots + e^{N x_p^{-1}} )^{1/N + p} \\
  &\hspace{.5cm}\int_0^y z^{-2} e^{Nz^{-1}}
    ( e^{N z^{-1}} + e^{N x_1^{-1}} + \dots + e^{N x_p^{-1}})^{-1/N - d}  \, \dif z\\
    &=
    \dfrac{(1+pN)(-1/N)}{-1/N - d + 1} ( e^{N x_1^{-1}} + \dots + e^{N x_p^{-1}} )^{1/N + p} \\
      &\hspace{.5cm}( e^{N y^{-1}} + e^{N x_1^{-1}} + \dots + e^{N x_p^{-1}} )^{-1/N - d + 1}\\
    &=
    \left( \dfrac{ e^{N y^{-1}} + e^{N x_1^{-1}} + \dots + e^{N x_p^{-1}} }{e^{N x_1^{-1}} + \dots + e^{N x_p^{-1}}} \right)^{-1/N - p}.
\end{align*}
Passing the last expression to the limit as $N \to \infty$ provides an ansatz for the true conditional distribution function 
\begin{align*}
    \Tilde{G}_{Y\mid \mathbf{X}}(y\mid \mathbf{x}) 
    = 
    \begin{cases}
		1, & y \geq \min(x_1,\dots,x_p)\\
		0, & y < \min(x_1,\dots,x_p)
	\end{cases}
\end{align*}
with $y, x_1,\dots,x_p > 0$, from where \eqref{softy} follows. 

\section*{\textit{Appendix~C: Proof of Proposition \ref{prop7}}}
\label{proof7}
Since $y\mapsto G_{Y \mid X} (y \mid x)$ is continuous (strictly increasing) for all $x \in (0,\infty)$, $y_{q\mid x}$  given by \eqref{genInv} is the solution to
$G_{Y\mid X}(y\mid x) = q$ for a fixed $q \in (0,1)$. Then $y$ satisfying $G_{Y\mid X}(y\mid x) = q$ is an implicit function of $x$ parametrized by $q$. Under our assumptions we apply the implicit function theorem and calculate the derivative of $y_{q\mid x}$ with respect to $x$ via
\begin{align}
\label{IFT}
    \partder{}{x}y_{q\mid x} 
    = 
    - \dfrac{\partder{}{x}G_{Y\mid X}(y\mid x)}{\partder{}{y}G_{Y\mid X}(y\mid x)}.
\end{align}
Equation~\eqref{IFT} combined with the monotone regression dependence property, i.e.~$x\mapsto G_{Y\mid X}(y\mid x)$ is non-increasing for all $y \in (0,\infty)$ \citep[][Theorem~1]{guillem2000}, and the strict monotonicity of $y\mapsto G_{Y\mid X}(y\mid x)$  (increasing) for all $x \in (0,\infty)$
gives
$$
\partder{}{x}y_{q\mid x} \geq 0.
$$
This completes the proof.
\section*{\textit{Appendix~D~Exact and limiting regression manifolds for logistic model}}
\label{linear_approx}
Here we give details on how the exact \eqref{RLLogistic} and approximated \eqref{RLLogisticA} regression manifolds for the logistic model can be derived. The derivations below require the use of Lampert $W$ function \citep{BorweinLindstrom2016LambertFunc} on which some properties and details can be found in the supplementary material. \\ \ \\
\noindent \textbf{\textit{D.1.~Exact regression manifold}}.~
Here we compute the conditional quantiles for bivariate extreme value distribution and their linear approximation for large $x$. Using (\ref{condD}), we calculate the conditional distribution function for the logistic model; for $(x,y)\in(0,\infty)^2$, it follows that 
\begin{align*}
    G_{Y|X}(y \mid x) 
    &=
    G(x,y) (x^{-1/\alpha} + y^{-1/\alpha})^{\alpha-1} x^{-1/\alpha - 1} x^2 \exp(x^{-1})\\
    &= 
    \exp \{ - ( x^{-1/\alpha} + y^{-1/\alpha})^{\alpha} + x^{-1} \} (x^{-1/\alpha} + y^{-1/\alpha})^{\alpha-1} x^{1-1/\alpha}.
\end{align*}
The conditional quantiles behave differently depending on the strength of dependence between extremes. The special case of the logistic model is for $\alpha=1$,  corresponding to independence between extremes for which the family of regression lines are known to be given by \eqref{regind}. We now derive conditional quantiles for bivariate dependent extremes, i.e. when $\alpha \in [0,1)$. Since $y \mapsto G_{Y|X}$ is continuous, a conditional quantile is a solution to 
\begin{align}\label{LogCond}
\exp \{ - ( x^{-1/\alpha} + y^{-1/\alpha})^{\alpha} \} ( x^{-1/\alpha} + y^{-1/\alpha})^{\alpha - 1} x^{1-1/\alpha}\exp(x^{-1})
=
q,
\end{align}
which can be written in terms of the Lambert $W$ function. Rewriting \eqref{LogCond} as
\begin{align*}
& \exp \left\{ - \dfrac{\alpha}{\alpha-1} ( x^{-1/\alpha} + y^{-1/\alpha})^{\alpha} \right\} ( x^{-1/\alpha} + y^{-1/\alpha})^{(\alpha - 1){\alpha}/{(\alpha-1)}} x^{{(\alpha - 1)}/{\alpha} {\alpha}/{(\alpha-1)}} \\ & \hspace{.5cm} \times \exp\left\{ \dfrac{\alpha}{\alpha-1} x^{-1} \right\}  
=
q^{{\alpha}/{(\alpha-1)}}\\
\Leftrightarrow & \exp \left\{ \dfrac{\alpha}{1-\alpha} ( x^{-1/\alpha} + y^{-1/\alpha}  )^{\alpha} \right\} \dfrac{\alpha}{1-\alpha} ( x^{-1/\alpha} + y^{-1/\alpha})^{\alpha} 
=
\dfrac{\alpha}{1-\alpha} x^{-1} \exp\left\{ \dfrac{\alpha}{1-\alpha} x^{-1} \right\} q^{{\alpha}/(\alpha-1)}\\
\Leftrightarrow & \dfrac{\alpha}{1-\alpha} ( x^{-1/\alpha} + y^{-1/\alpha})^{\alpha} 
=
W \left( \dfrac{\alpha}{1-\alpha} x^{-1} e^{\alpha/(1-\alpha) x^{-1}} q^{{\alpha}/(\alpha-1)} \right),
\end{align*}
gives
\begin{align}
\label{sol}
	y_{q\mid x} 
	= 
	\left[ \left\{ \dfrac{1-\alpha}{\alpha} x W \left( \dfrac{\alpha}{1-\alpha} x^{-1} e^{ \alpha/(1-\alpha) x^{-1}} q^{{\alpha}/(\alpha-1)} \right) \right\}^{1/\alpha} - 1 \right]^{-\alpha} x , \quad x > 0.
\end{align}

\noindent From properties of the Lambert $W$ function (see supplementary material for details) it follows that
$$
	\underset{x \rightarrow \infty}{\lim} x \, W \left(\dfrac{\alpha}{1-\alpha} \, x^{-1} \, e^{ {\alpha}/(1-\alpha) x^{-1}} q^{{\alpha}/{(\alpha-1)}} \right) = \dfrac{\alpha}{1-\alpha}q^{{\alpha}/{(\alpha-1)}},
$$
and that the conditional quantiles tend to infinity as $x \rightarrow \infty$. \\ \ \\
\noindent \textbf{\textit{D.2.~Limiting regression manifold}}.~
Below we show that \eqref{RLLogisticA} holds. To find $\gamma_q$ and $\beta_q$ we use the expansion of the principal branch of the Lambert $W$ function, a solution to $z = w e^w$ when $z>0$, around $0$, that is  
\begin{align*}
    W(z) 
    =
    \sum\limits_{n=1}^{\infty} \dfrac{(-n)^{n-1}}{n!} z^n = z + O(z^2),
\end{align*}
where $O(z)$ is big-$O$ of $z$. Substituting in \eqref{sol} the expansion of $W$ leads to 
\begin{align*}
	y_{q\mid x} 
	&=
	\bigg[ \bigg\{ \dfrac{1-\alpha}{\alpha}x \bigg( \dfrac{\alpha}{1-\alpha} x^{-1} e^{ \frac{\alpha}{1-\alpha} x^{-1}} q^{\dfrac{\alpha}{\alpha-1}} + O(x^{-2} e^{{2 \alpha}/{(1 - \alpha)} x^{-1}}) \bigg) \bigg\}^{1/\alpha} - 1 \bigg]^{-\alpha} x\\
	&=
	\bigg[ \bigg\{ e^{\alpha/(1-\alpha) x^{-1}} q^{\dfrac{\alpha}{\alpha-1}} + O(x^{-1} e^{{2 \alpha}/{(1 - \alpha)} x^{-1}})  \bigg\}^{1/\alpha} - 1 \bigg]^{-\alpha} x,
\end{align*}
and taking $x \to \infty$ we find the linear asymptote of $x \mapsto y_{q\mid x}$ which is based on 
\begin{align}\label{betaq}
	\beta_q
	&=
	\lim\limits_{x \to \infty} \dfrac{y_{q \mid x}}{x} 
	=
	\{  q^{-1/(1-\alpha)}  - 1 \}^{-\alpha},
\end{align}
and, the more involved calculation of $\gamma_q$. The derivative of the asymptotic expansion of a function does not necessarily correspond to the asymptotic expansion of the derivative of the function; hence, we use L'Hospital with \eqref{sol} and only then inject the asymptotic expansions. We have then after some tedious derivations that  
\begin{align}\label{gammaq}
	\gamma_q
	&=
	\lim\limits_{x \to \infty} \left( y_{q \mid x} - \beta_q x \right)
	=
    \dfrac{\alpha}{1-\alpha} \{ q^{{1}/{(\alpha-1)}} - 1 \}^{-\alpha-1} 
    \{ q^{{\alpha}/{(1-\alpha)}}
     -  1
     \}  q^{1 / (\alpha-1)}.
\end{align}
The resulting linear approximation is as follows, for $q \in (0,1)$ and $x \gg 1$:
\begin{align*}
    \tilde{y}_{q \mid x}
    &=
    \dfrac{\alpha}{1-\alpha} \{ q^{{1}/{(\alpha-1)}} - 1 \}^{-\alpha-1} 
    \{ q^{{\alpha}/{(1-\alpha)}}
     -  1
     \}  q^{{1}/{(\alpha-1)}}
     +
     \{  q^{-{1}/{(1-\alpha)}}  - 1 \}^{-\alpha}x \\
     &=\alpha_q + \beta_q x.
\end{align*}

\begin{acknowledgements}
We thank, without implicating, Johan Segers (Universit\'e catholique de Louvain) and Rapha\"el Huser (King Abdullah University of Science and Technology) for insightful discussions, suggestions, and comments. M.~de Carvalho acknowledges support from the \emph{Funda{\c c}$\tilde{\text{a}}$o para a Ci$\hat{e}$ncia e a Tecnologia} (Portuguese NSF) through the projects PTDC/MAT-STA/28649/2017 and UID/MAT/00006/2020.
G.~dos Reis acknowledges support from the \emph{Funda{\c c}$\tilde{\text{a}}$o para a Ci$\hat{e}$ncia e a Tecnologia} (Portuguese NSF) through the project UIDB/00297/2020 (Centro de Matem\'atica e Aplica\c c$\tilde{\text{o}}$es CMA/FCT/UNL). A.~Kumukova was supported by The Maxwell Institute Graduate School in Analysis and its Applications, a Centre for Doctoral Training funded by the UK Engineering and Physical Sciences Research Council (grant EP/L016508/01), the Scottish Funding Council, Heriot-Watt University, and the University of Edinburgh. 
\end{acknowledgements}

{\footnotesize 
\noindent \textbf{Data Availability Statement}~The datasets analysed during the current study are available from Yahoo Finance (\url{https://finance.yahoo.com}).}

 \newpage

\section*{Supplementary material}
\section*{\textit{SM~A: Additional numerical evidence}}
\subsection*{\textit{SM~A.1.Induced prior for $p$-covariate setting}}
We report on two one-shot numerical experiments aimed at illustrating the approach in Section~\ref{higher_dim} in the paper, that induces a prior on the space of all regression manifolds by resorting to Bernstein polynomials and an approximation of a multivariate GEV density due to \cite{cooley2012}. For the numerical experiments in this supplementary material, we test our model by taking a trivariate logistic extreme value distribution with dependence parameter $\alpha=0.1$ (`strongly' dependent extremes) for the case $p=2$, i.e. with the trivariate GEV distribution
\begin{equation*}
G(y, x_1, x_2) 
= 
\exp \{ - ( y^{-1/\alpha} + x_1^{-1/\alpha} + x_2^{-1/\alpha} )^{\alpha} \},
\quad
y, x_1, x_2>0.
\end{equation*}
We generate two samples of sizes $n=10000$ and $n=20000$ which, after thresholding at $95\%$ empirical quantiles of the pseudo-radius, yield $k=500$ and $k=1000$ data points to fit the model. Here, we use a similar prior specification and MCMC setup as in Section~\ref{prelim} of the paper.\par
Figure~\ref{fig:contoursplots} indicates that the proposed estimator of the angular density captures reasonably well the dependence between extremes by concentrating around the barycenter of the simplex, though in a less pronounced form than the true density. As can be seen from Figure~\ref{fig:regressionsurface}, the resulting fitted regression lines resemble the true ones, $L_q$, and increasing sample size improves the fit as the lateral surfaces of the estimates become more slanting, for $q=\{0.3,0.5,0.7\}$.

\renewcommand{\thefigure}{SM.1}
\begin{figure}[htbp]
	\centering
	\includegraphics[scale = 0.4]{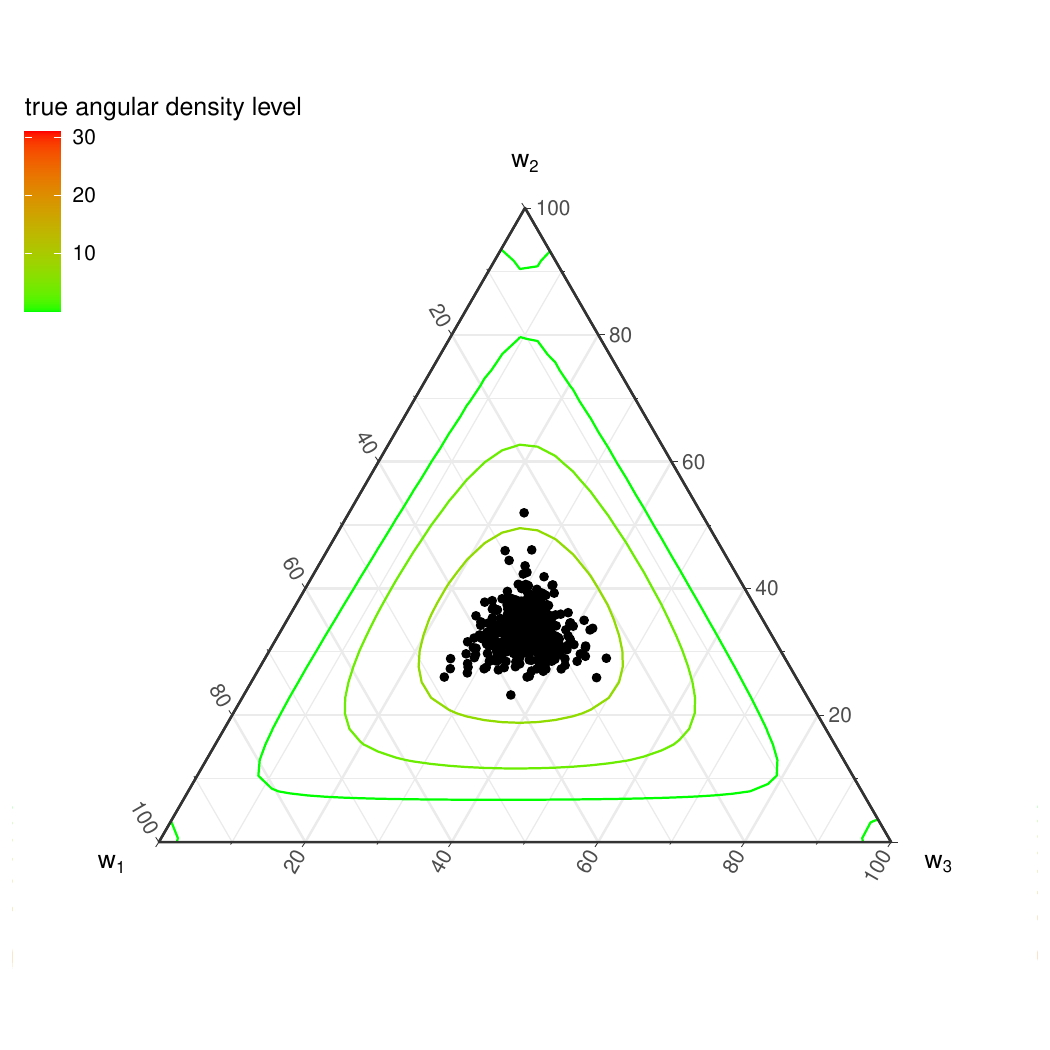}
	\hspace{-0.75cm}
	\includegraphics[scale = 0.4]{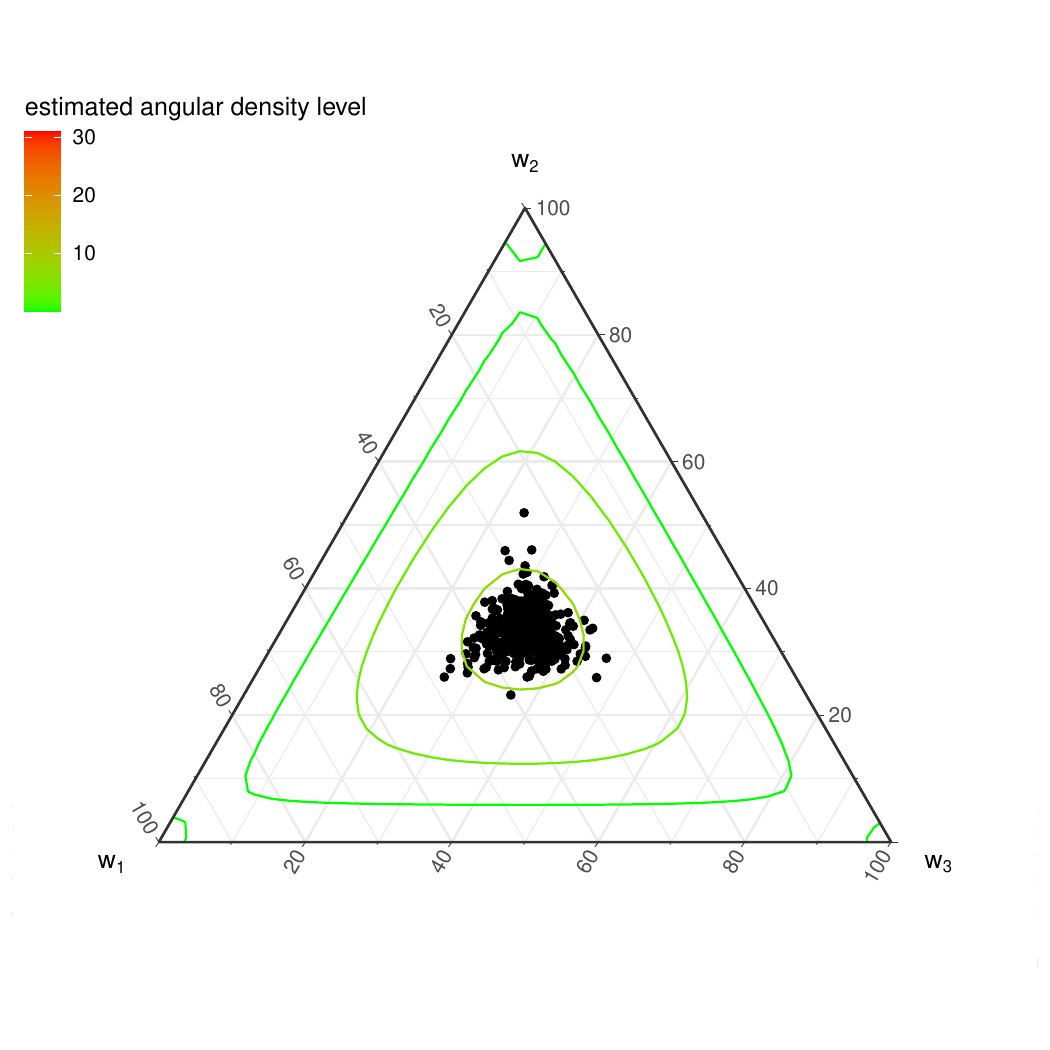}
	\caption{\label{ang_dens} Level plots of the true angular density (left) along with the posterior mean estimate resulting from the methods from Section~3.2 (right) on $n=10000$ observations for the trivariate logistic extreme value distribution, on a single-run experiment, with dependence parameter $\alpha=0.1$.}
	\label{fig:contoursplots}
\end{figure}

\renewcommand{\thefigure}{SM.2}
\begin{figure}[htbp]
	\begin{center}
		\begin{center}
			\begin{footnotesize}
				\textbf{$q=0.3$}
			\end{footnotesize}
		\end{center}
		\vspace{-0.5cm}
		\begin{minipage}[c]{.3\textwidth}
			\includegraphics[scale = 0.28]{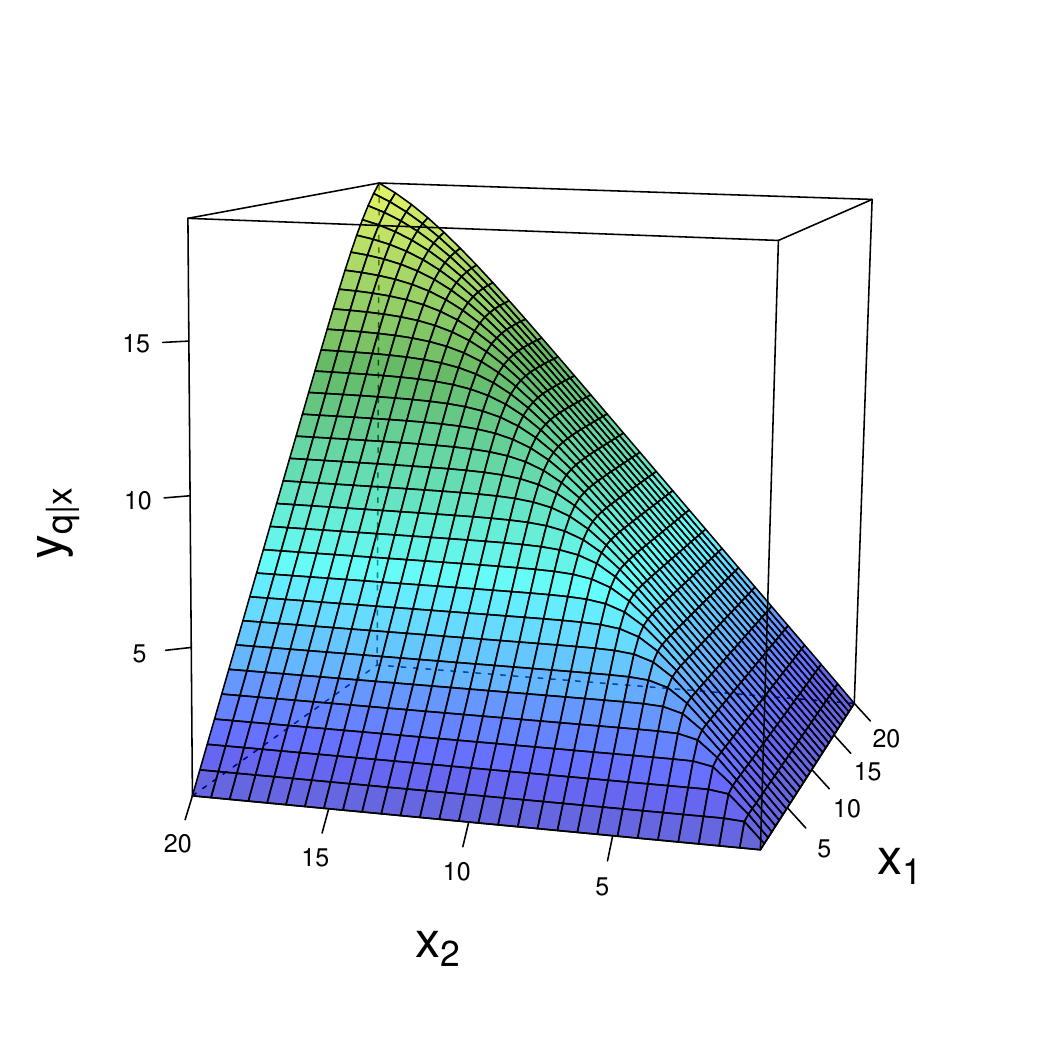}
			\text{\hspace{0cm}}
		\end{minipage} 
		\begin{minipage}[c]{.3\textwidth}
			\includegraphics[scale = 0.28]{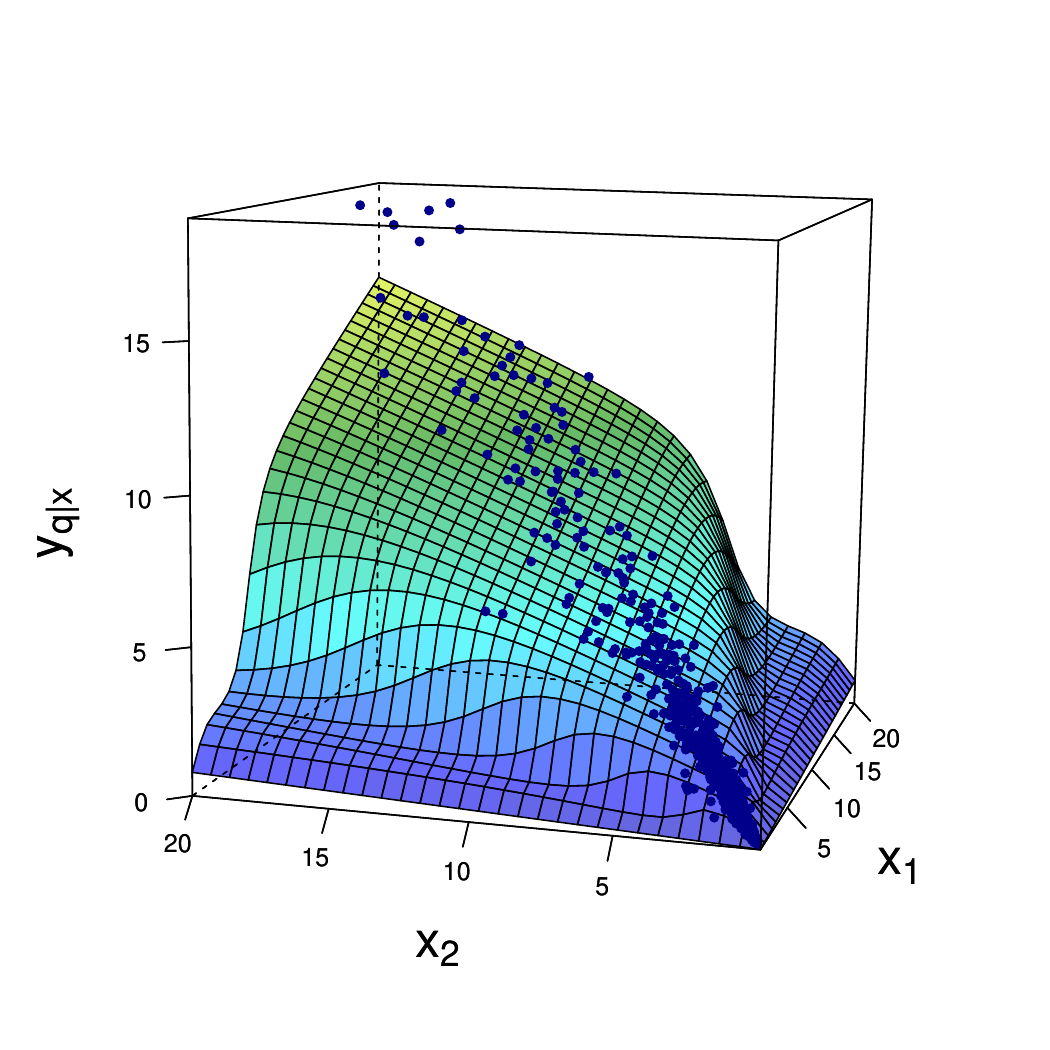}
		\end{minipage}
		\begin{minipage}[c]{.3\textwidth}
			\includegraphics[scale = 0.28]{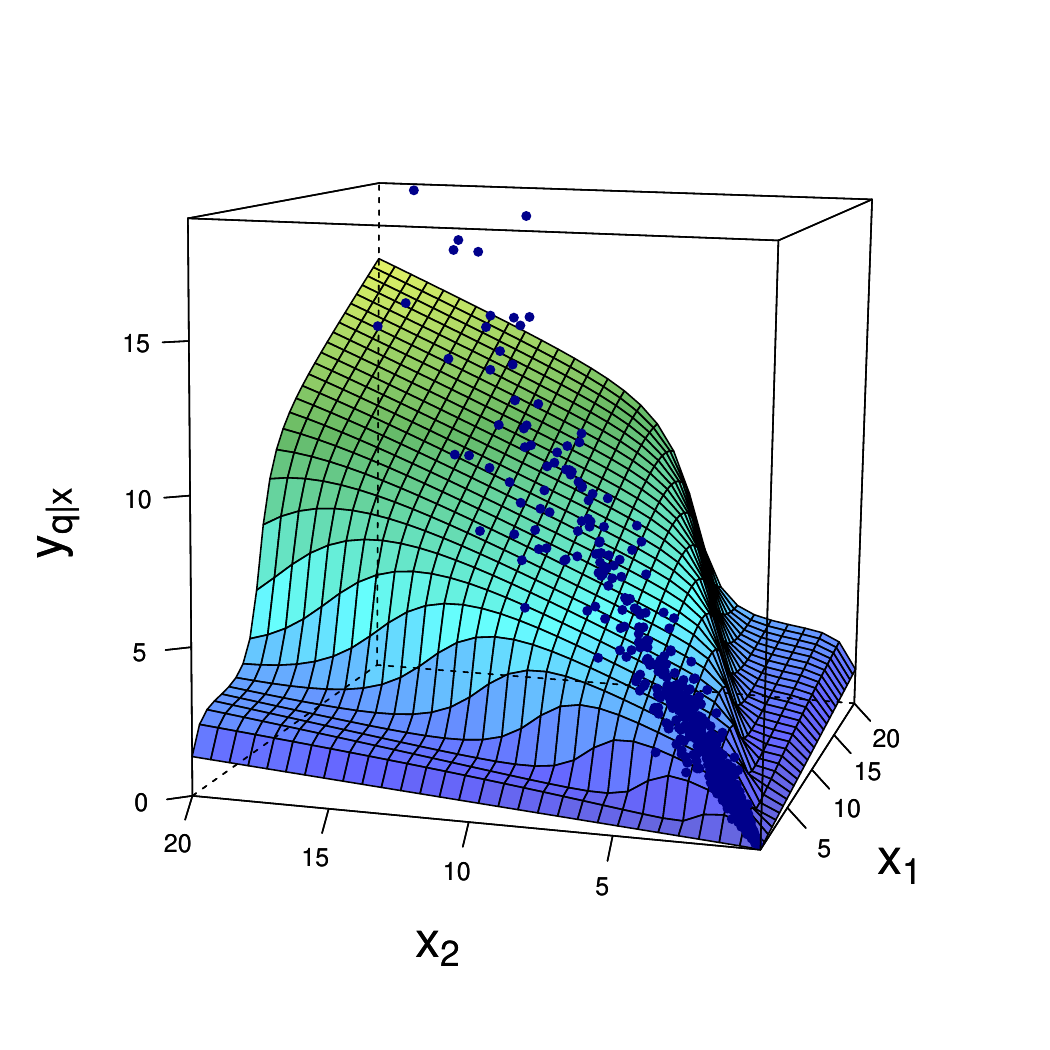}
		\end{minipage}
		\vspace{-0.5cm}
		\begin{center}
			\begin{footnotesize}
				$q=0.5$
			\end{footnotesize}
		\end{center}
		\vspace{-0.5cm}
		\begin{minipage}[c]{.3\textwidth}
			\includegraphics[scale = 0.28]{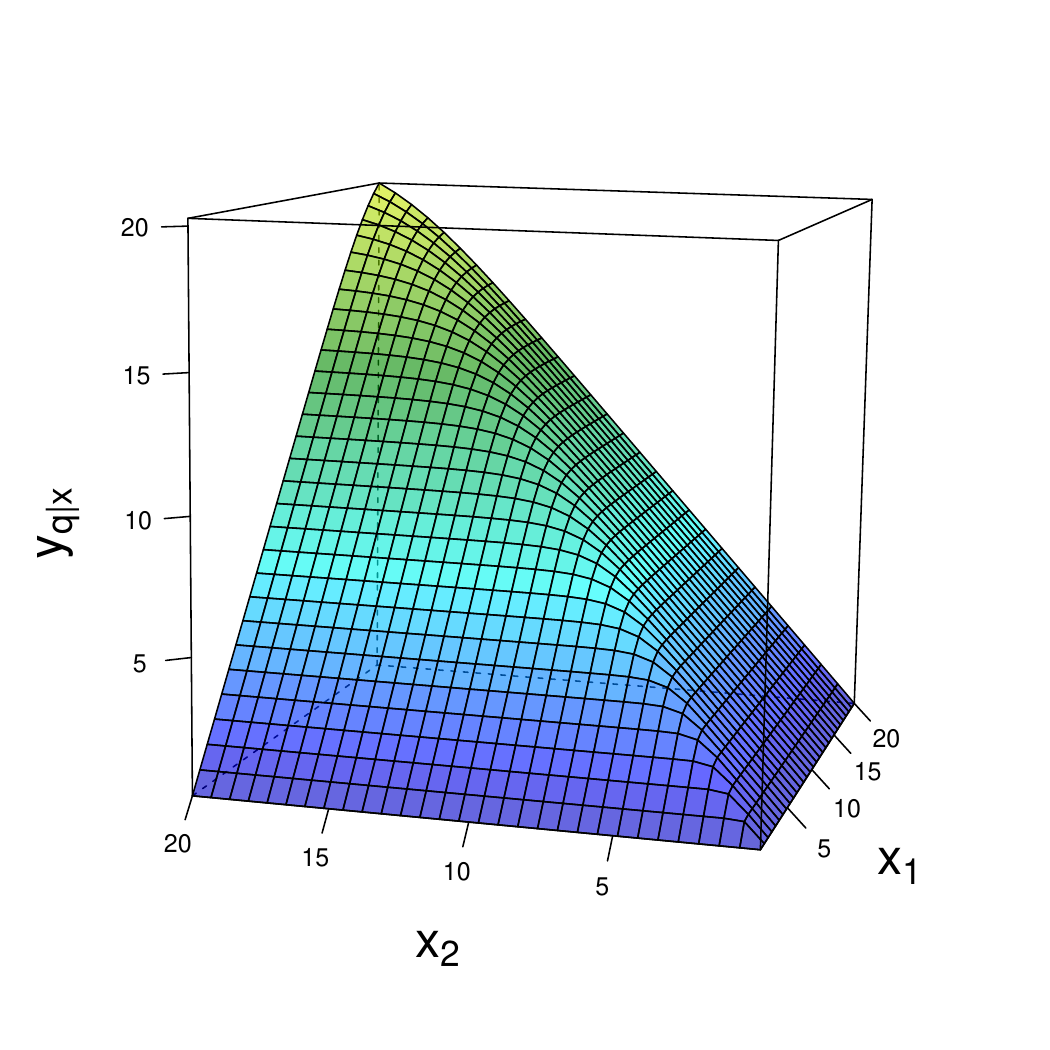}
			\text{\hspace{0cm}}
		\end{minipage} 
		\begin{minipage}[c]{.3\textwidth}
			\includegraphics[scale = 0.28]{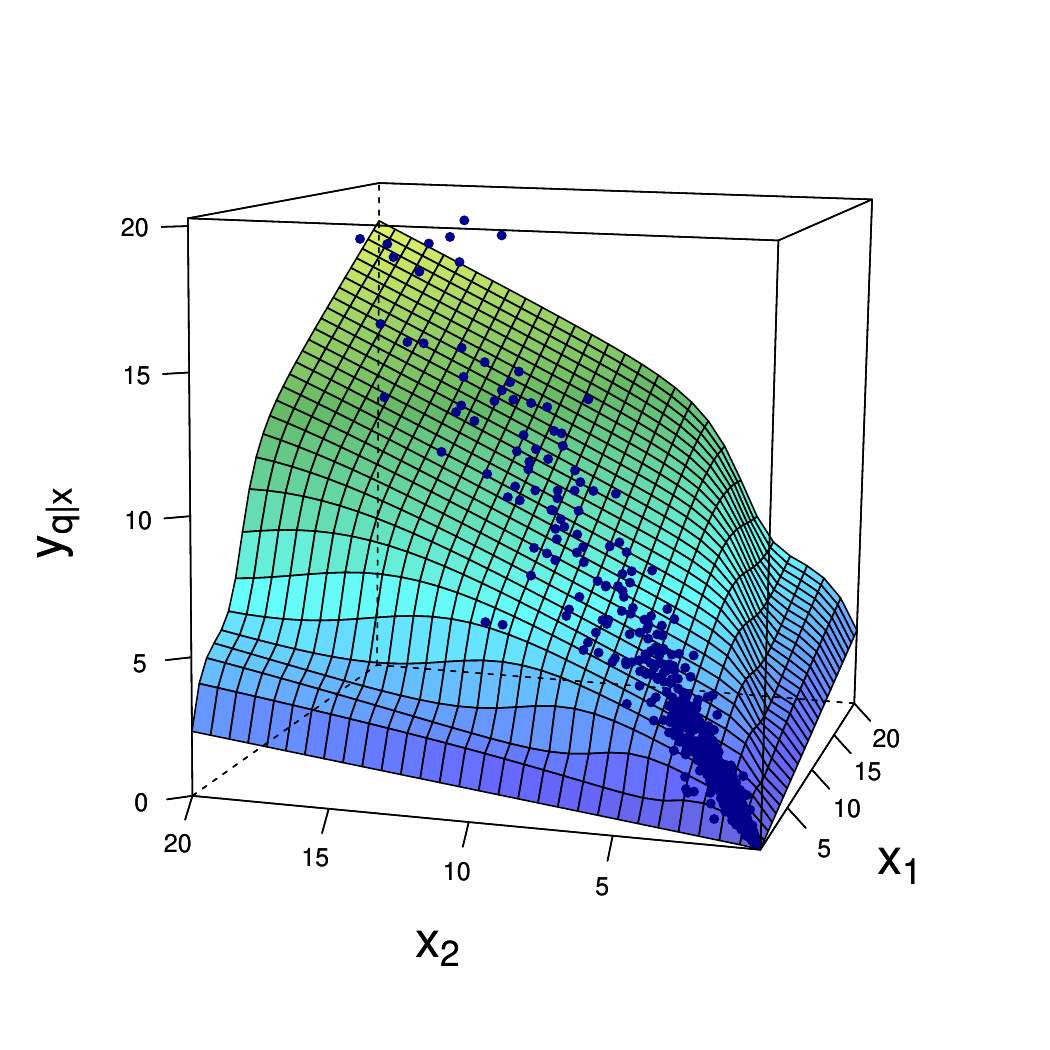}
		\end{minipage}
		\begin{minipage}[c]{.3\textwidth}
			\includegraphics[scale = 0.28]{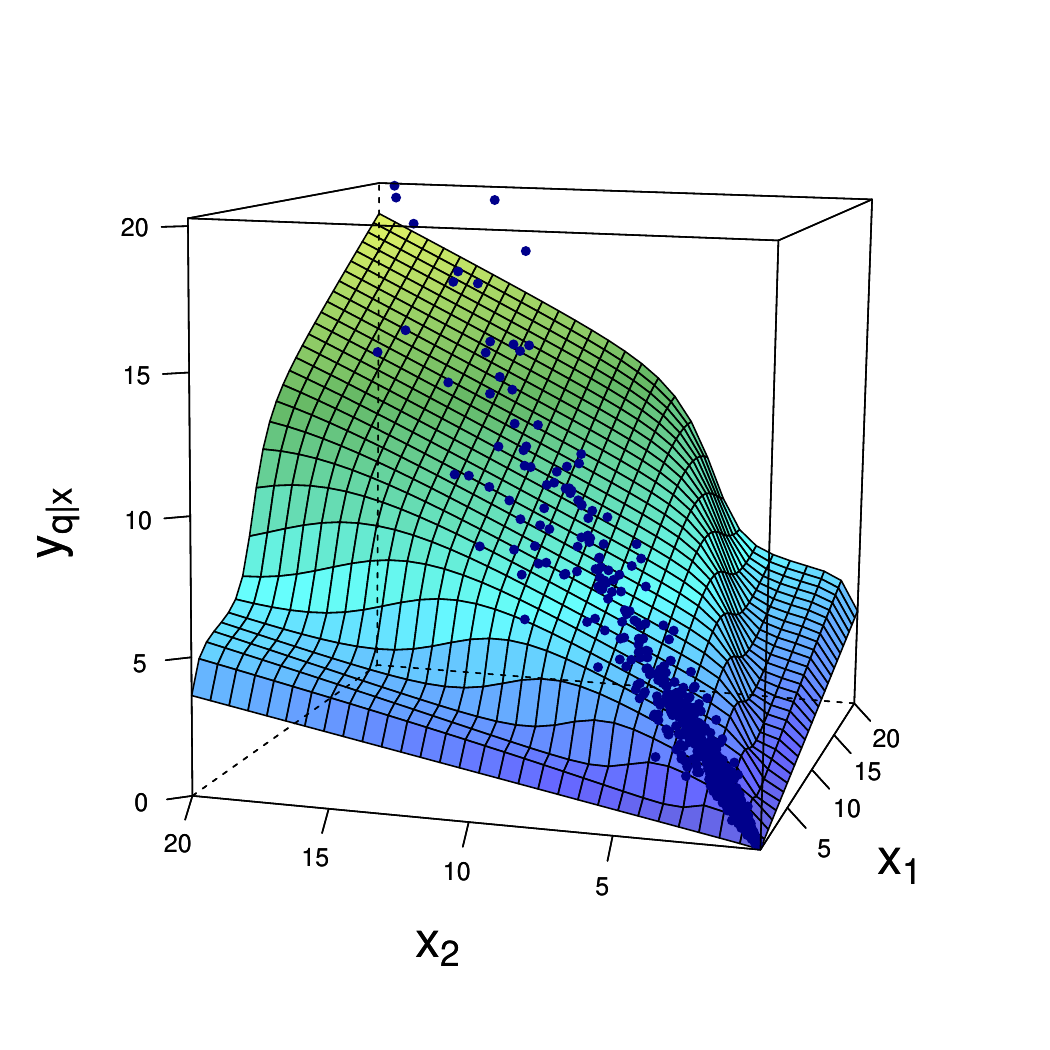}
		\end{minipage}
		\vspace{-0.5cm}
		\begin{center}
			\begin{footnotesize}
				$q=0.7$
			\end{footnotesize}
		\end{center}
		\vspace{-0.5cm}
		\begin{minipage}[c]{.3\textwidth}
			\includegraphics[scale = 0.28]{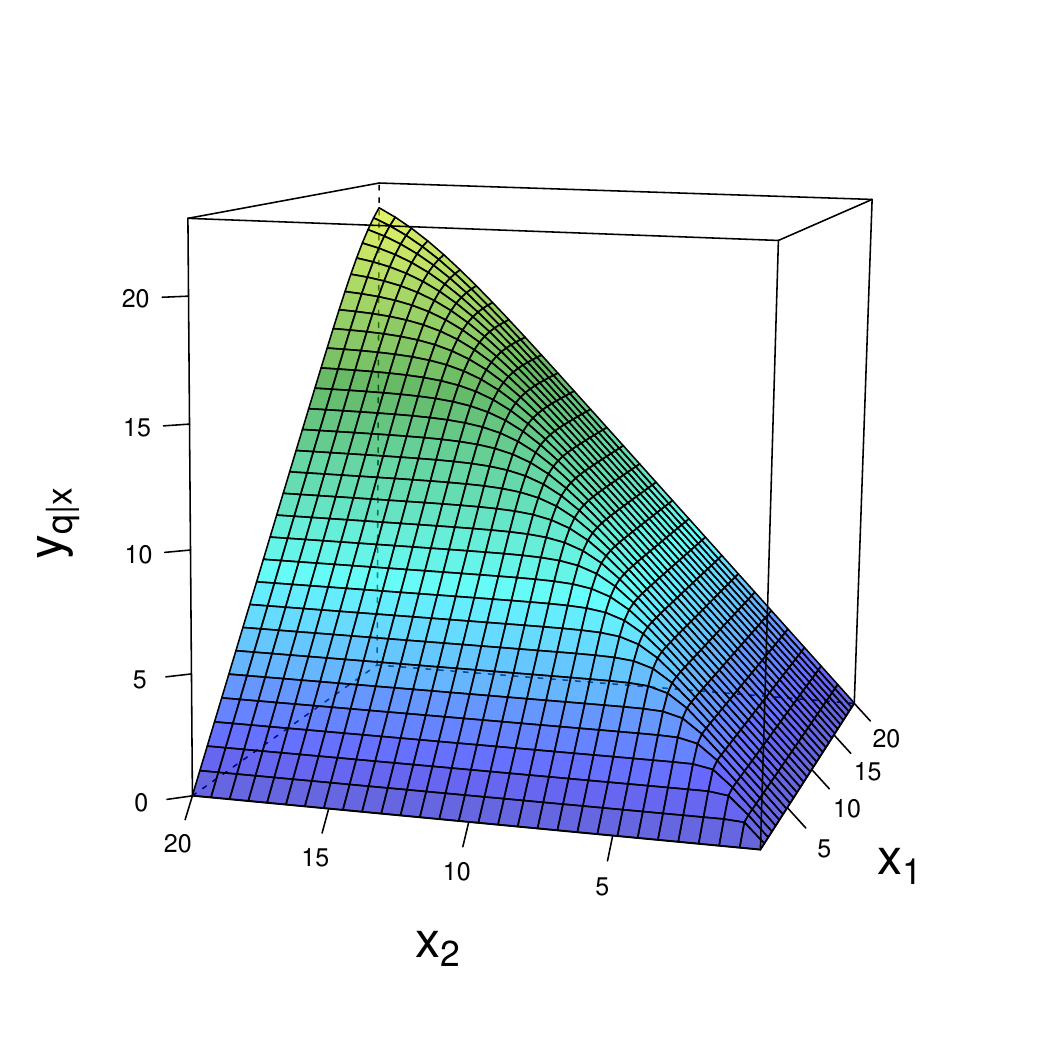}
			\text{\hspace{0cm}}
		\end{minipage} 
		\begin{minipage}[c]{.3\textwidth}
			\includegraphics[scale = 0.28]{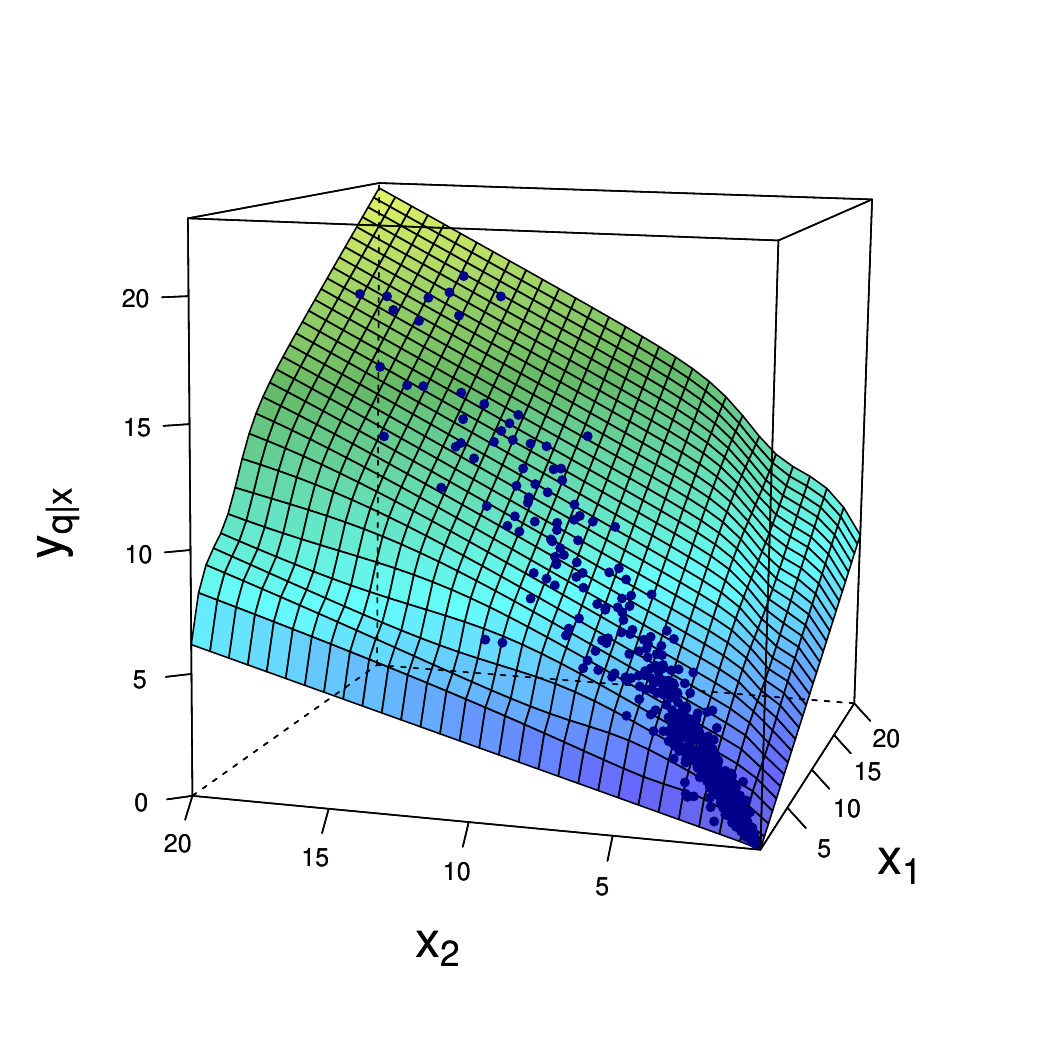}
		\end{minipage}
		\begin{minipage}[c]{.3\textwidth}
			\includegraphics[scale = 0.28]{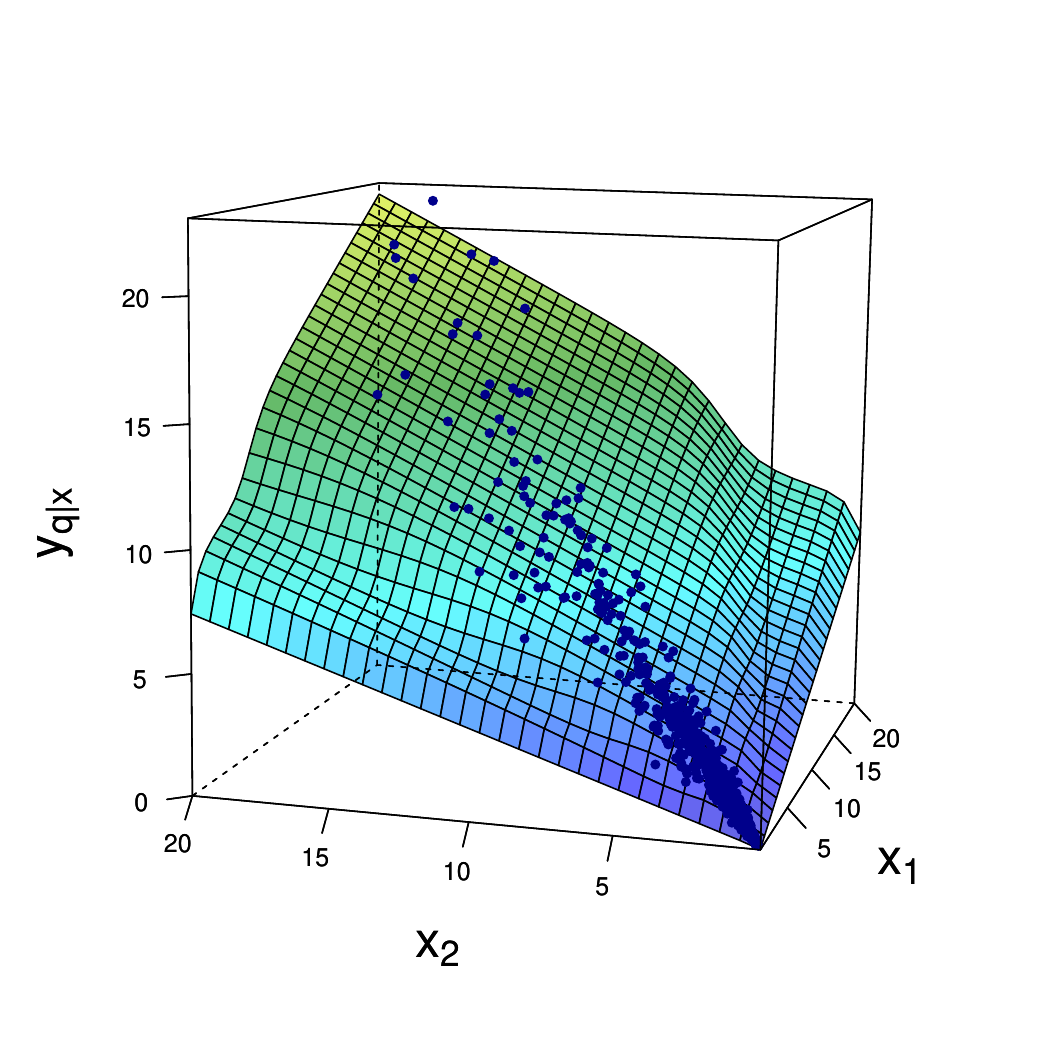}
		\end{minipage}
	\end{center}
	\caption{The true $L_q$ (left) for $q=\{0.3, 0.5,0.7\}$ (top to bottom) along with the posterior mean estimate resulting from the methods from Section~3.3 on $n=10000$ (middle) and $n=20000$ (right) observations for the trivariate logistic extreme value distribution, on a single-run experiment, with the dependence parameter $\alpha=0.1$ over the domain $\mathbf{x}=(x_1,x_2)\in(0,20]^2$.}
	\label{fig:regressionsurface}
\end{figure}

\subsection*{\textit{SM~A.2.Induced prior for $p$-covariate setting}}

\noindent Figure~\ref{mc250} below complements Figure~\ref{mc500} in the paper; the number of exceedances of the figure reported here is $k=250$ while that in the main paper is $k=500$. 

\renewcommand{\thefigure}{SM.3}
\begin{figure}[htbp]
	\centering
	\vspace{0cm}
	\begin{center}
		\begin{footnotesize}
			\textbf{Scenario 1}--strongly dependent extremes: Husler--Reiss model
		\end{footnotesize}
	\end{center}
	\vspace{0cm}
	\hspace{0cm}
	\includegraphics[width=1\textwidth]{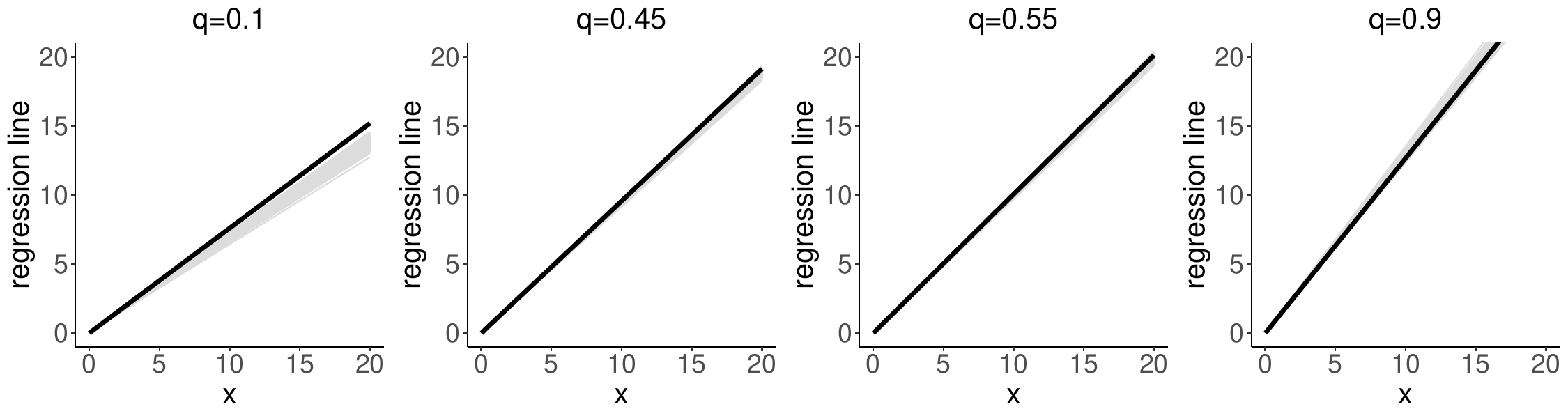}
	\vspace{-0.5cm}
	\begin{center}
		\begin{footnotesize}
			\textbf{Scenario 2}--weakly dependent extremes: Logistic model
		\end{footnotesize}
	\end{center}
	\vspace{0cm}
	\hspace{0cm}
	\includegraphics[width=1\textwidth]{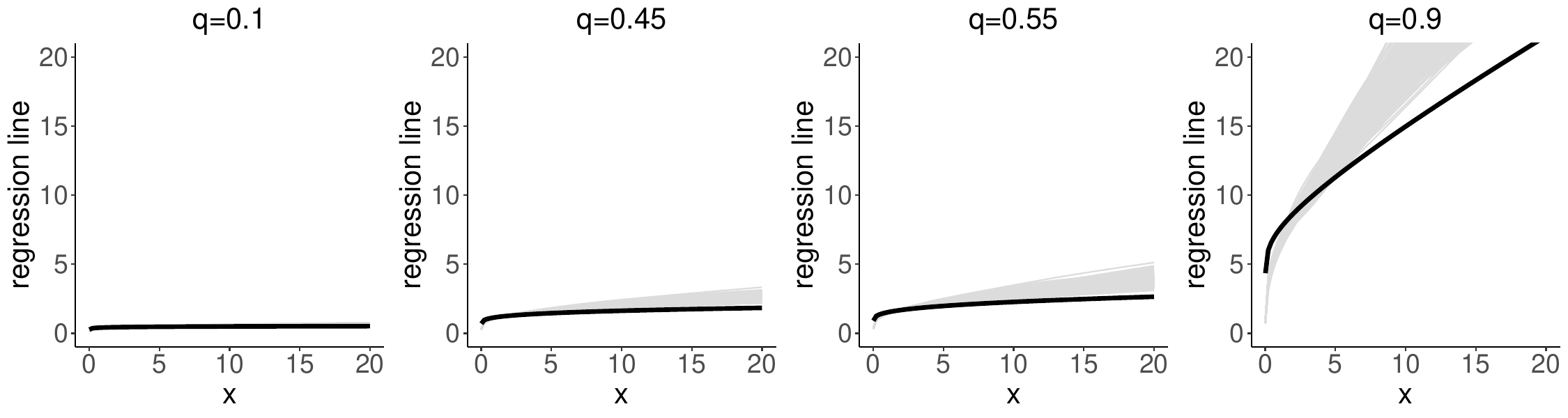}
	\vspace{-0.5cm}
	\begin{center}
		\begin{footnotesize}
			\textbf{Scenario 3}--asymmetric intermediate dependence: Coles--Tawn model
		\end{footnotesize}
	\end{center}
	\vspace{0cm}
	\hspace{0cm}
	\includegraphics[width=1\textwidth]{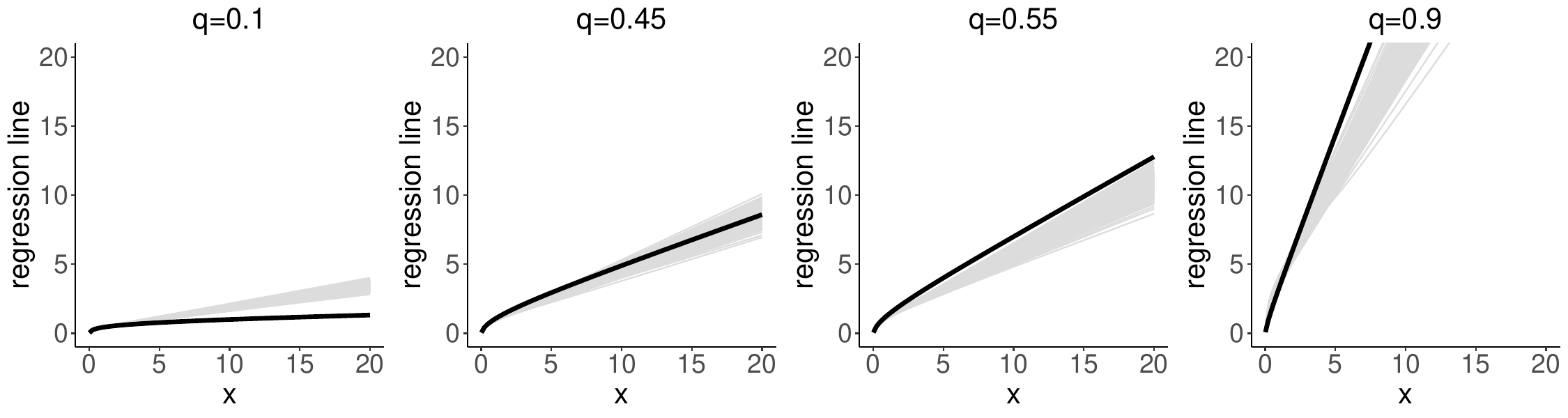}
	\caption{Posterior mean regression lines $L_q$ for $q=\{0.1,0.45,0.55,0.9\}$ and $x \in (0,20]$ for each of the $500$ Monte Carlo samples ($k = 500$, gray lines) plotted against the true conditional quantiles (black line) for Husler--Reiss, Logistic, and Coles--Tawn bivariate extreme value models (top to bottom).}
	\label{mc250}
\end{figure}

\section*{\textit{SM~B: Details on the Lambert $W$ function}}
The Lambert $W$ function is used in the paper for deriving the regression manifold for the logistic model (cf~Example~\ref{logistic} and Appendix~D), and thus we offer here some details on it. Formally, the Lambert $W$ function is a set of functions representing the inverse relation of the function $f(z) = z e^z$ for any complex $z$. Since we deal only with positive real valued $z$, the equation $f(z) = z e^z$ has only one solution $w = W(z)$, with $W$ being the principal branch of the Lambert $W$ function. A useful property of this function is that for any constant $a\in \mathbb{R}$ one has 
$$
\underset{z \rightarrow \infty}{\lim}  z W(a/z)  
=
\underset{z \rightarrow \infty}{\lim}   a e^{-W(a/z)}  
= 
a,
$$
which is derived from
$$
\underset{z \rightarrow \infty}{\lim} \frac az 
= 
\underset{z \rightarrow \infty}{\lim} e^{W(a/z)} W(a/z) 
\quad \Rightarrow \quad 
\underset{z \rightarrow \infty}{\lim} W(a/z) 
= 0.
$$
See \cite{BorweinLindstrom2016LambertFunc} for further details.

\section*{\textit{SM~C: Further empirical analysis}}\label{empirical}
In this section we present the reverse analysis to that presented in Section~\ref{application} of the paper; that
is, here NASDAQ is the response, whereas NYSE is taken as covariate. Figure~\ref{data_RL2} is
thus the equivalent of Fig.~\ref{data_RL} in the paper but for the reverse analysis; and the same applies to Table~\ref{quant_orig2}, which is the reverse analysis equivalent of Table~\ref{quant_orig} in the paper. Interpretations follow along the same lines as in
Section~\ref{application} of the paper.

\renewcommand{\thefigure}{SM.4}
\begin{figure}
	\centering \hspace{.2cm}
	\begin{minipage}{0.48\linewidth} \hspace{3.5cm} \footnotesize (a)
	\end{minipage}
	\begin{minipage}{0.48\linewidth} \hspace{3.3cm} \footnotesize (b)
	\end{minipage}\\ \vspace{-.5cm}
	\begin{minipage}{0.48\linewidth}\hspace{.3cm}
		\includegraphics[scale=0.43]{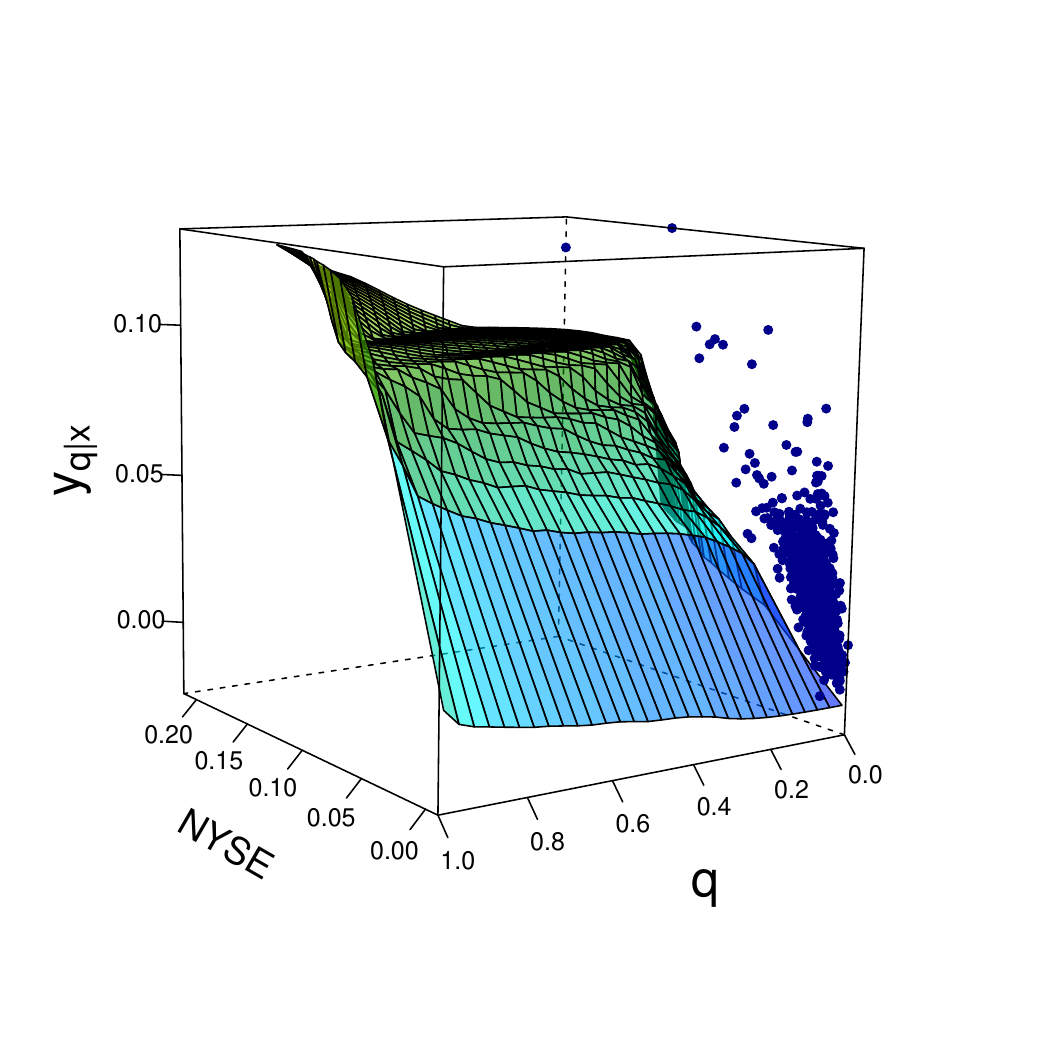}  
	\end{minipage} 
	\begin{minipage}{0.48\linewidth}\hspace{0.6cm}
		\includegraphics[scale = .34]{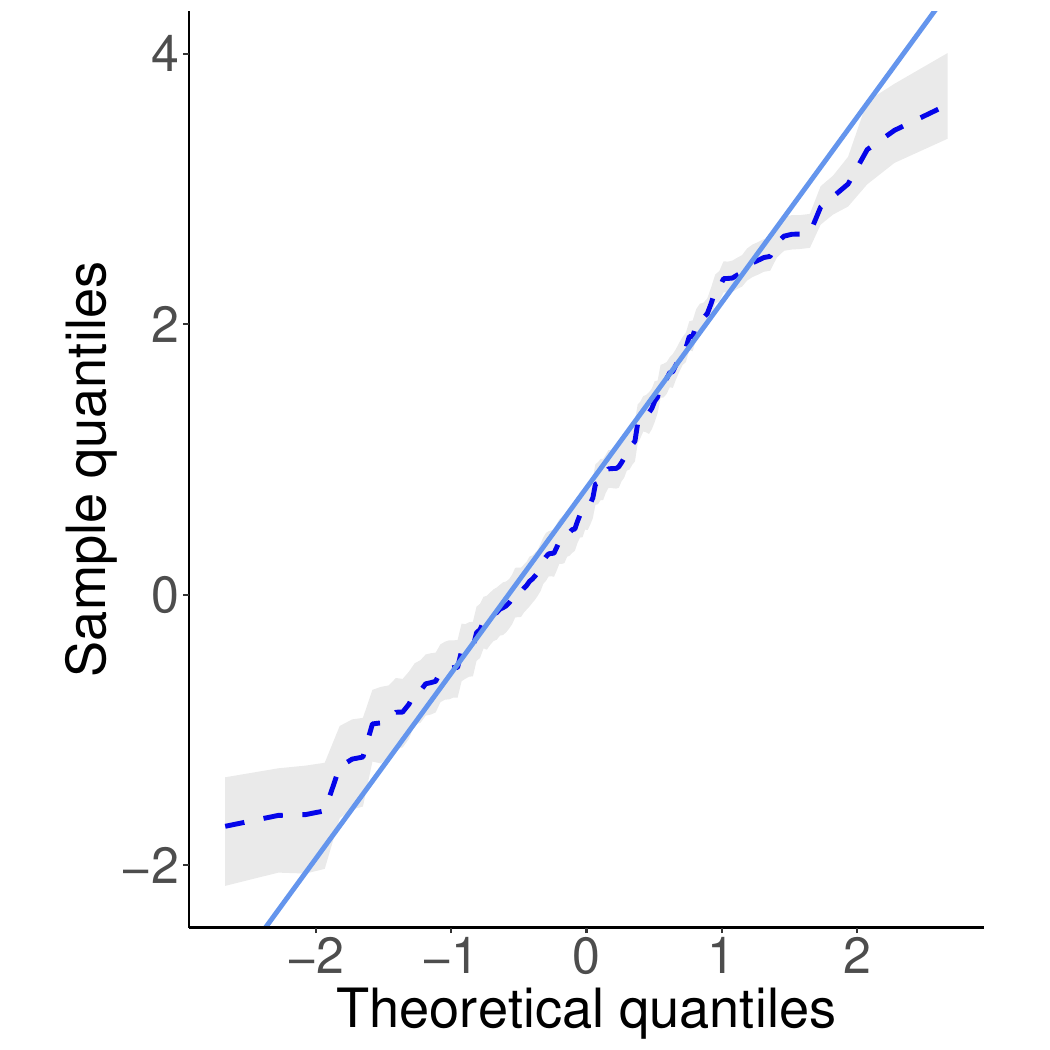}  
	\end{minipage}
	\begin{minipage}{0.48\linewidth} \hspace{3.5cm} \footnotesize (c)
	\end{minipage}
	\begin{minipage}{0.48\linewidth} \hspace{3.3cm}  \footnotesize (d)
	\end{minipage}\\ 
	\begin{minipage}{0.4\linewidth} \hspace{-.4cm}    
		\includegraphics[scale=0.27]{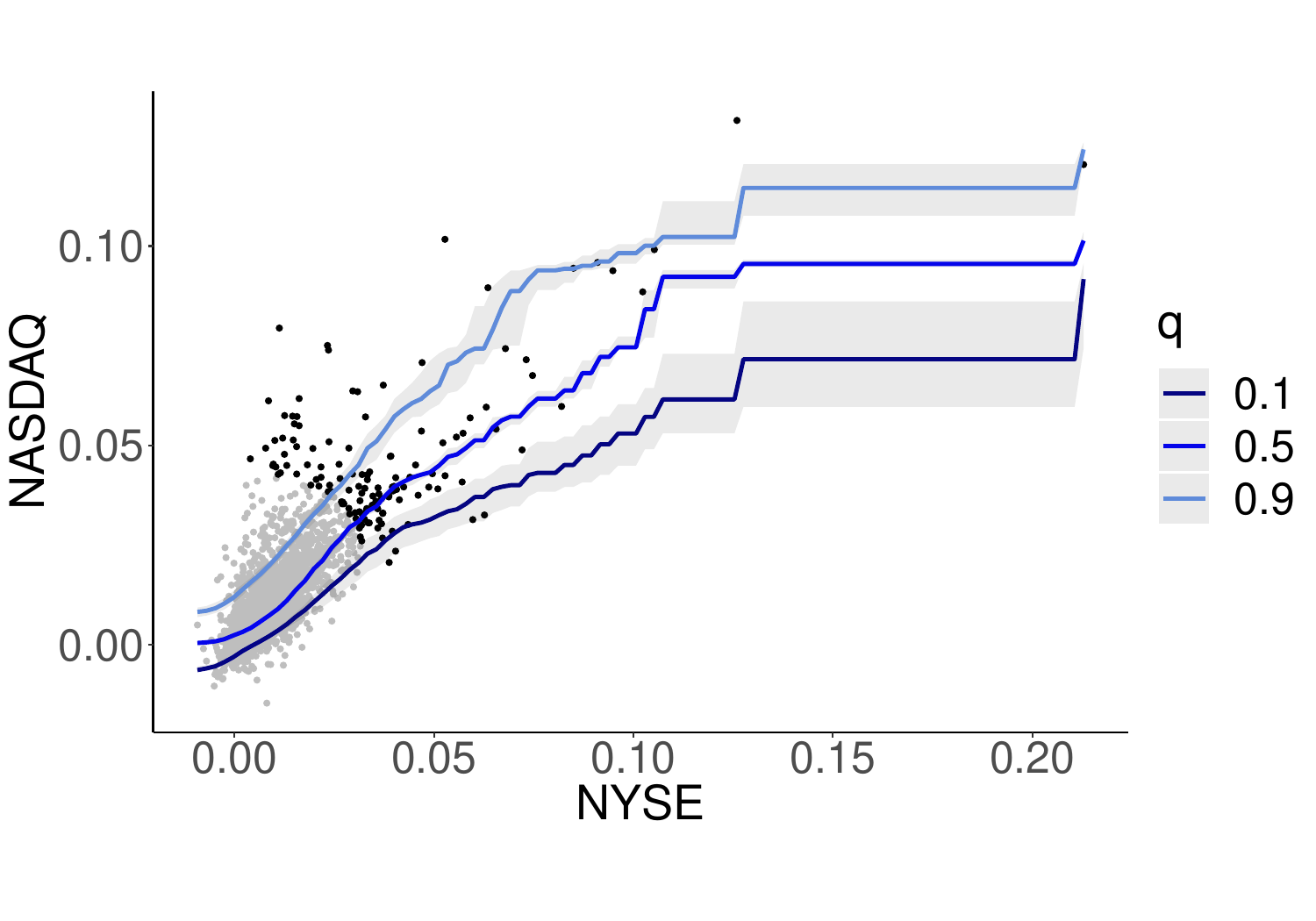} 
	\end{minipage}    
	\begin{minipage}{0.4\linewidth}\hspace{.8cm}    
		\includegraphics[scale=0.24]{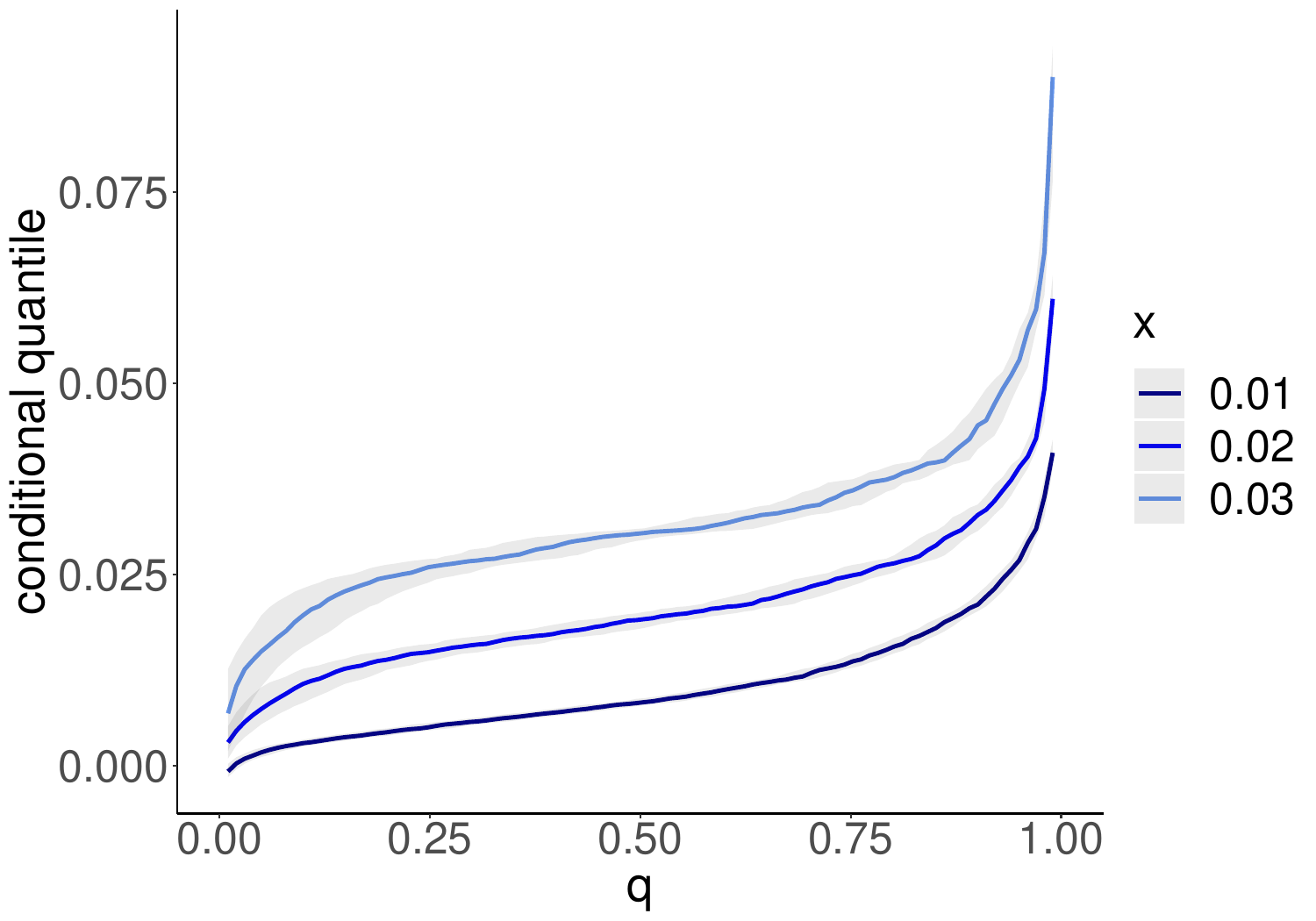}
	\end{minipage}
	\caption{\label{data_RL2} (a) Posterior mean regression manifold $\mathscr{L}$ for NASDAQ given NYSE along with joint negative log returns overlaid on one of the faces of the box. (b) QQ-plot of randomized quantile residuals; the dashed line represents the posterior mean plotted along with credible bands. (c) Posterior mean  regression lines $L_q$ for $q=\{0.1,0.5,0.9\}$ for NYSE given NASDAQ along with $95\%$ credible bands and plotted against joint negative log returns. (d) Posterior mean conditional quantile curves $\{y_{q \mid x}:q\in(0,1)\}$ of negative log returns on NASDAQ for $x=\{0.01,0.02,0.03\}$, along with $95\%$ credible bands,  corresponding to negative log returns on NYSE in the original margins.}
\end{figure}

\begin{table} \caption{Predicted $75\%$, $90\%$ and $95\%$ quantiles of losses on NASDAQ evaluated for $1\%$, $2\%$ and $3\%$ weekly maxima losses on NYSE, with $95\%$ credible intervals in brackets; negative log-returns used as proxy for losses}            
	\centering    
	\begin{threeparttable}
		\begin{tabular}{c c c c}  
			\hline 			\hline
			\multirow{2}{*}{NASDAQ} &
			\multicolumn{3}{c}{NYSE} \\ [0.5ex]
			\cline{2-4}
			& $0.01$ & $0.02$ & $0.03$\\ [0.5ex]
			\hline 
			$75\%$ & $0.0136$ & $0.0249$ & $0.0359$ \\
			& \footnotesize $(0.0128,0.0146)$ &  \footnotesize $(0.0236, 0.0262)$ & \footnotesize $(0.0339,0.0375)$ \\
			$90\%$ & $0.0210$ & $0.0328$ & $0.0439$  \\
			& \footnotesize $(0.0202,0.0224)$ &  \footnotesize $(0.0307,0.0341)$ & \footnotesize $(0.0413,0.0476)$ \\
			$95\%$ & $0.0268$ & $0.0389$ & $0.0529$
			\\
			& \footnotesize $(0.0254,0.0284)$ &  \footnotesize $(0.0371,0.0400)$ & \footnotesize $(0.0498,0.0570)$ \\
			[1ex]
			\hline 			\hline
		\end{tabular}
		\label{quant_orig2}   
	\end{threeparttable}
\end{table}

\section*{\textit{SM~D: Comparing exact and limiting regression manifold for logistic model}}
Here we illustrate how the exact and limiting regression manifold for logistic model compare; see Appendix~D for details on the derivation of these. As it can be seen from Figs.~\ref{comparison_1}--\ref{comparison_2}, the linearly approximated regression manifold derived in Appendix~D.1 in the paper offers a sensible approximation of the true regression manifold, for large values of $x$.

\renewcommand{\thefigure}{SM.5}
\begin{figure}
	\begin{center}
		\begin{center}
			\begin{footnotesize}
				$\alpha=0.1$
			\end{footnotesize}
		\end{center}
		\vspace{-1cm}
		\begin{minipage}[c]{7cm}\hspace{1cm}
			\includegraphics[width=1\textwidth]{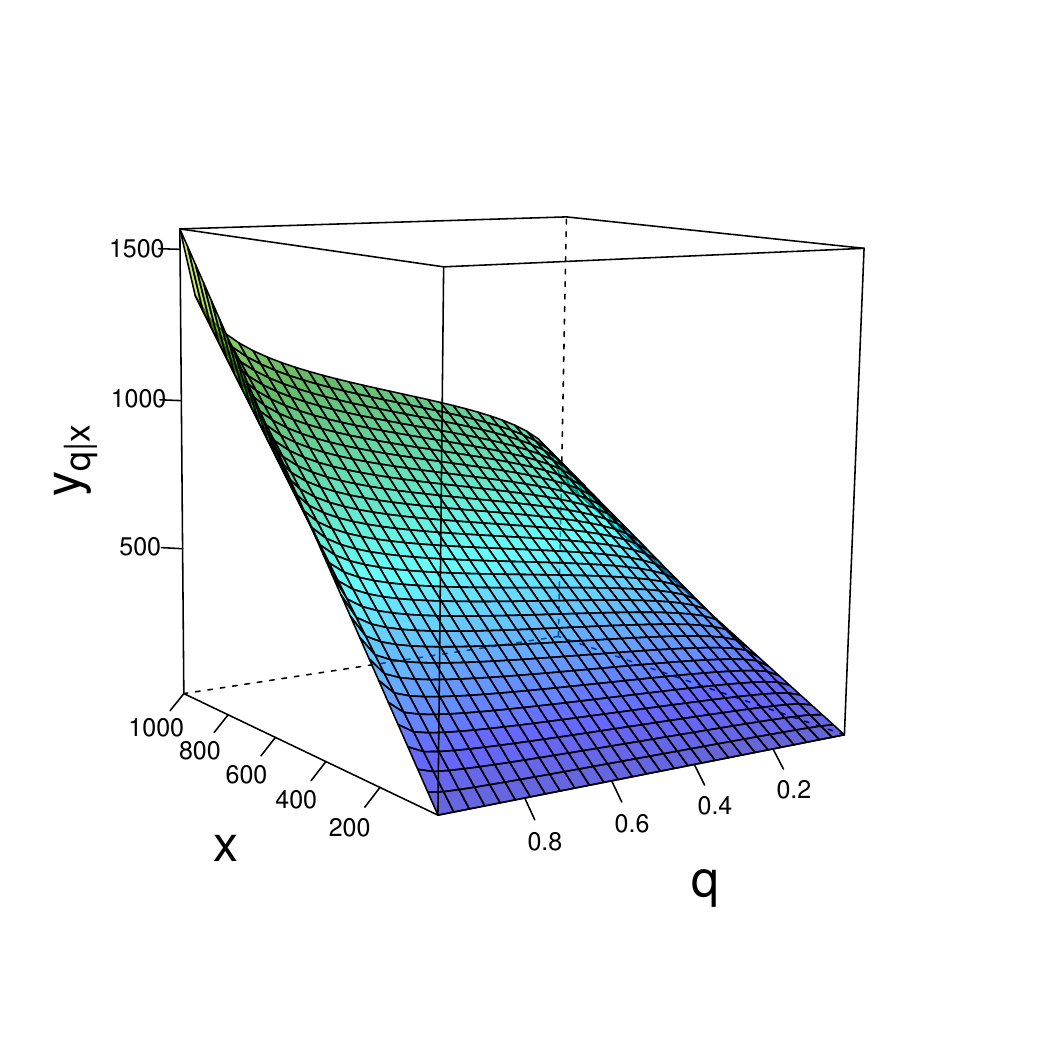}
			\text{\hspace{0cm}}
		\end{minipage} 
		\begin{minipage}[c]{7cm}
			\includegraphics[width=1\textwidth]{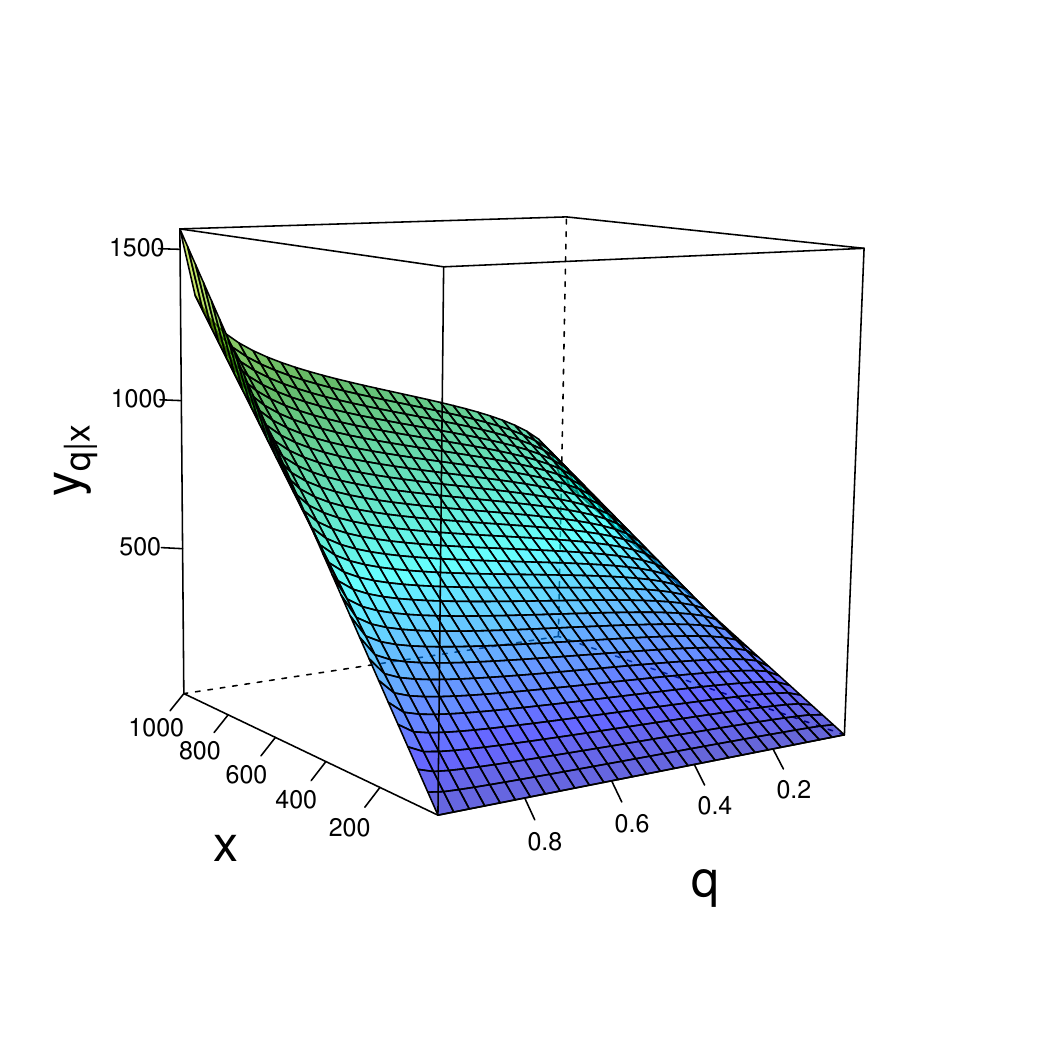}
		\end{minipage}
		\vspace{-1cm}
		\begin{center}
			\begin{footnotesize}
				$\alpha = 0.4$
			\end{footnotesize}
		\end{center}
		\vspace{-1cm}
		\begin{minipage}[c]{7cm}\hspace{1cm}
			\includegraphics[width=1\textwidth]{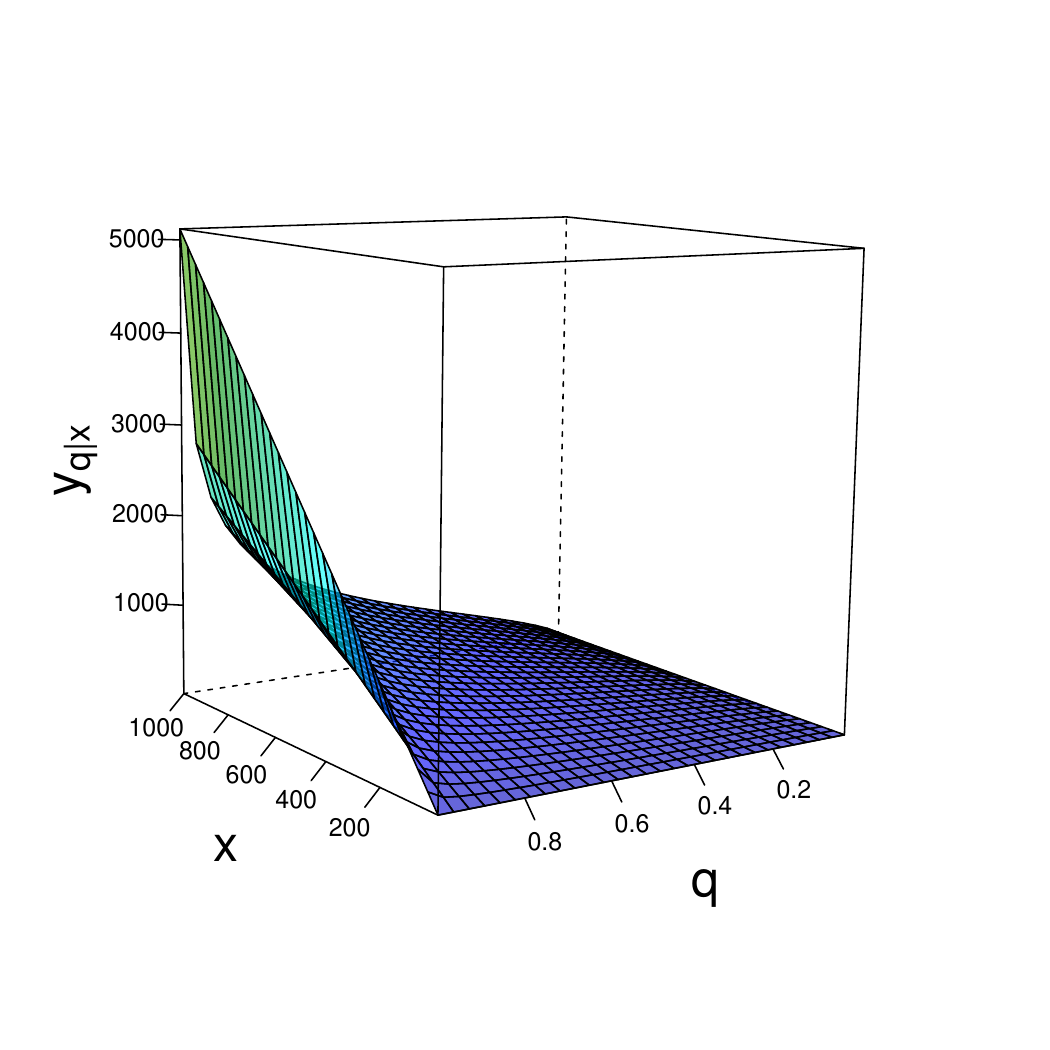}
			\text{\hspace{0cm}}
		\end{minipage} 
		\begin{minipage}[c]{7cm}
			\includegraphics[width=1\textwidth]{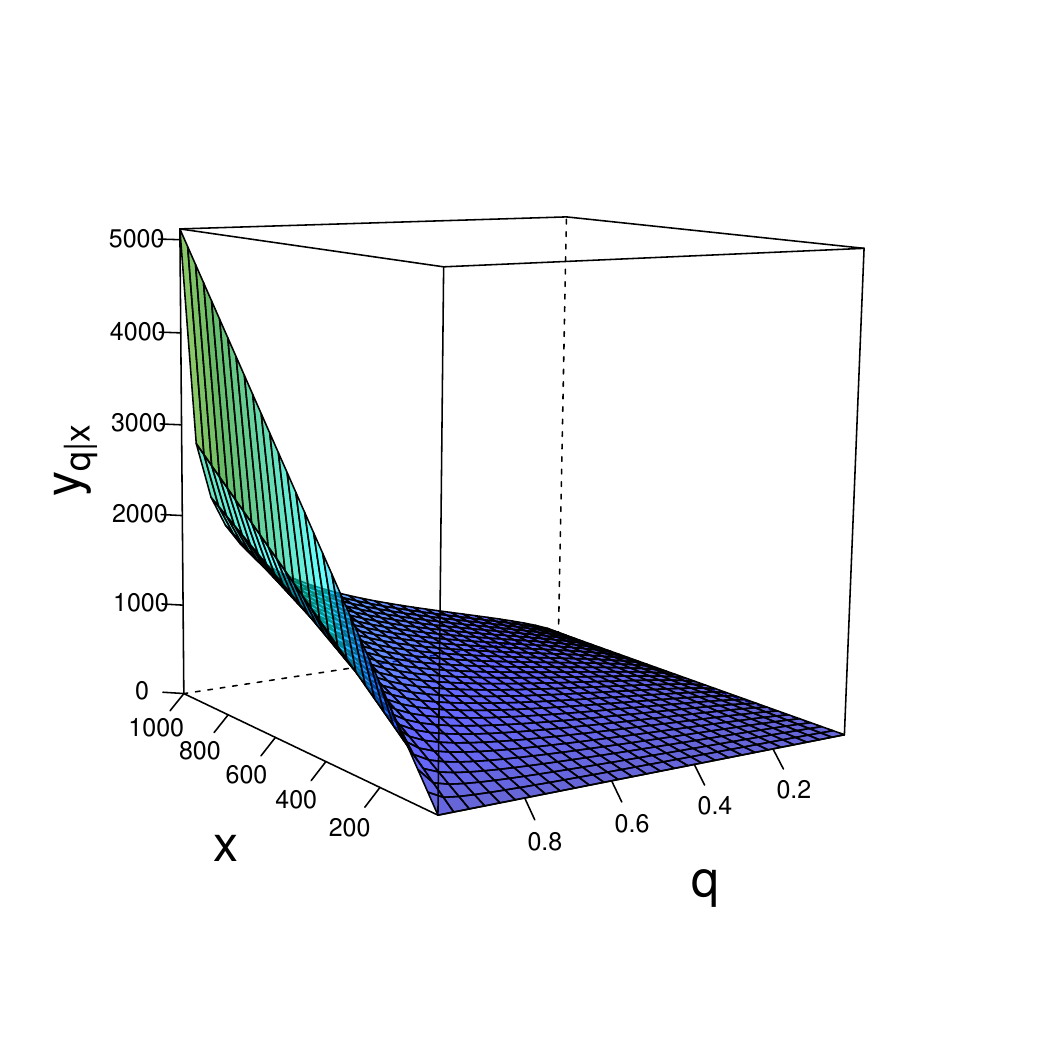}
		\end{minipage}
		\vspace{-1cm}
		\begin{center}
			\begin{footnotesize}
				$\alpha=0.6$
			\end{footnotesize}
		\end{center}
		\vspace{-1cm}
		\begin{minipage}[c]{7cm}\hspace{1cm}
			\includegraphics[width=1\textwidth]{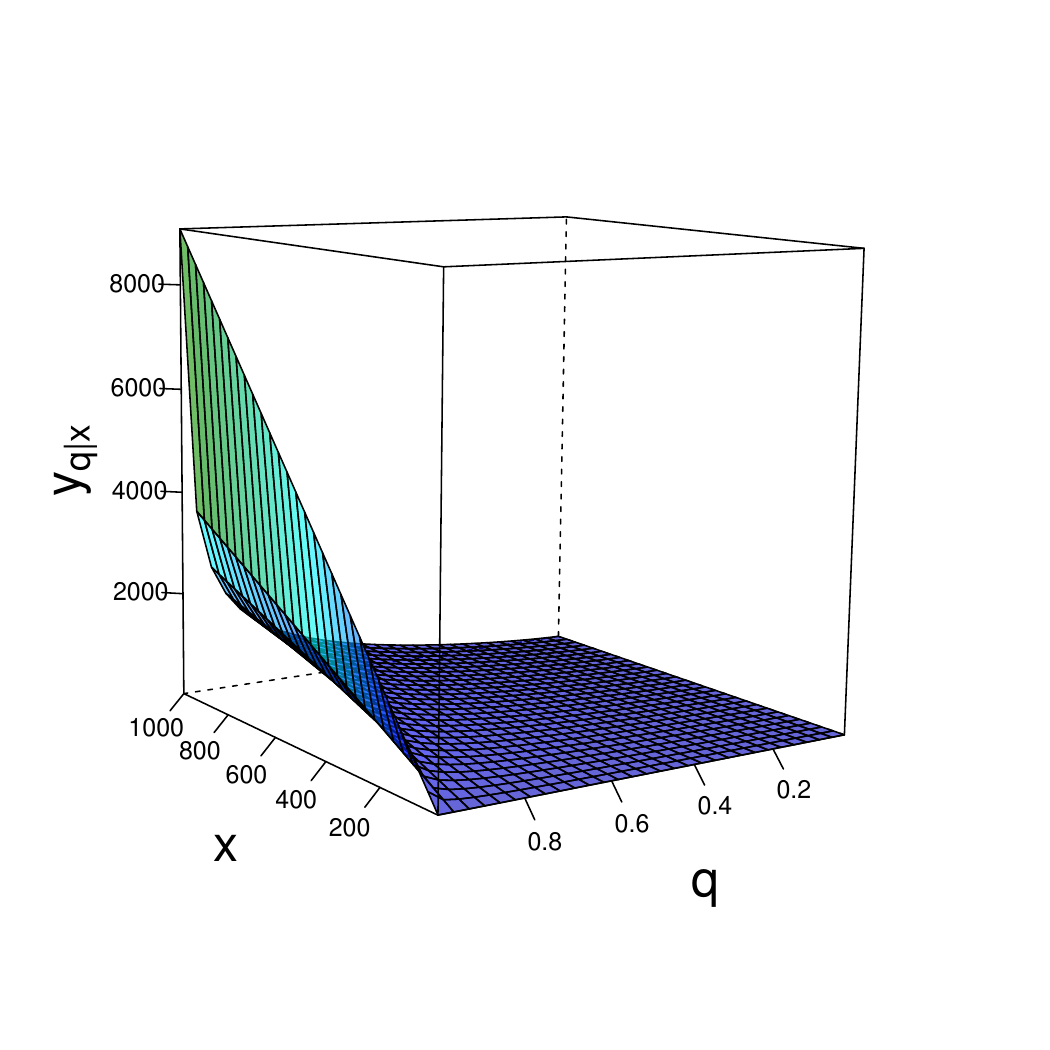}
			\text{\hspace{0cm}}
		\end{minipage} 
		\begin{minipage}[c]{7cm}
			\includegraphics[width=1\textwidth]{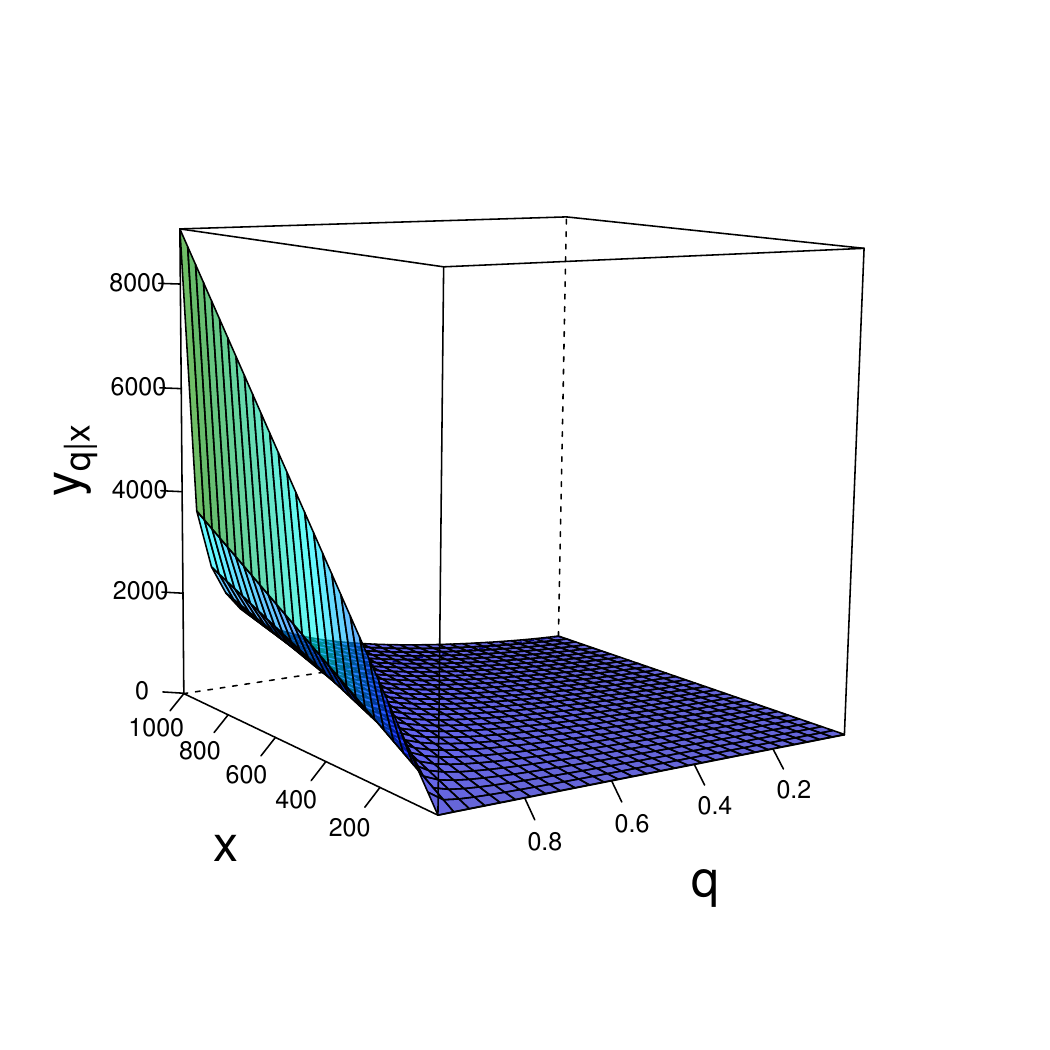}
		\end{minipage}
		\vspace{-1cm}
		\begin{center}
			\begin{footnotesize}
				$\alpha=0.9$
			\end{footnotesize}
		\end{center}
		\vspace{-1cm}
		\begin{minipage}[c]{7cm}\hspace{1cm}
			\includegraphics[width=1\textwidth]{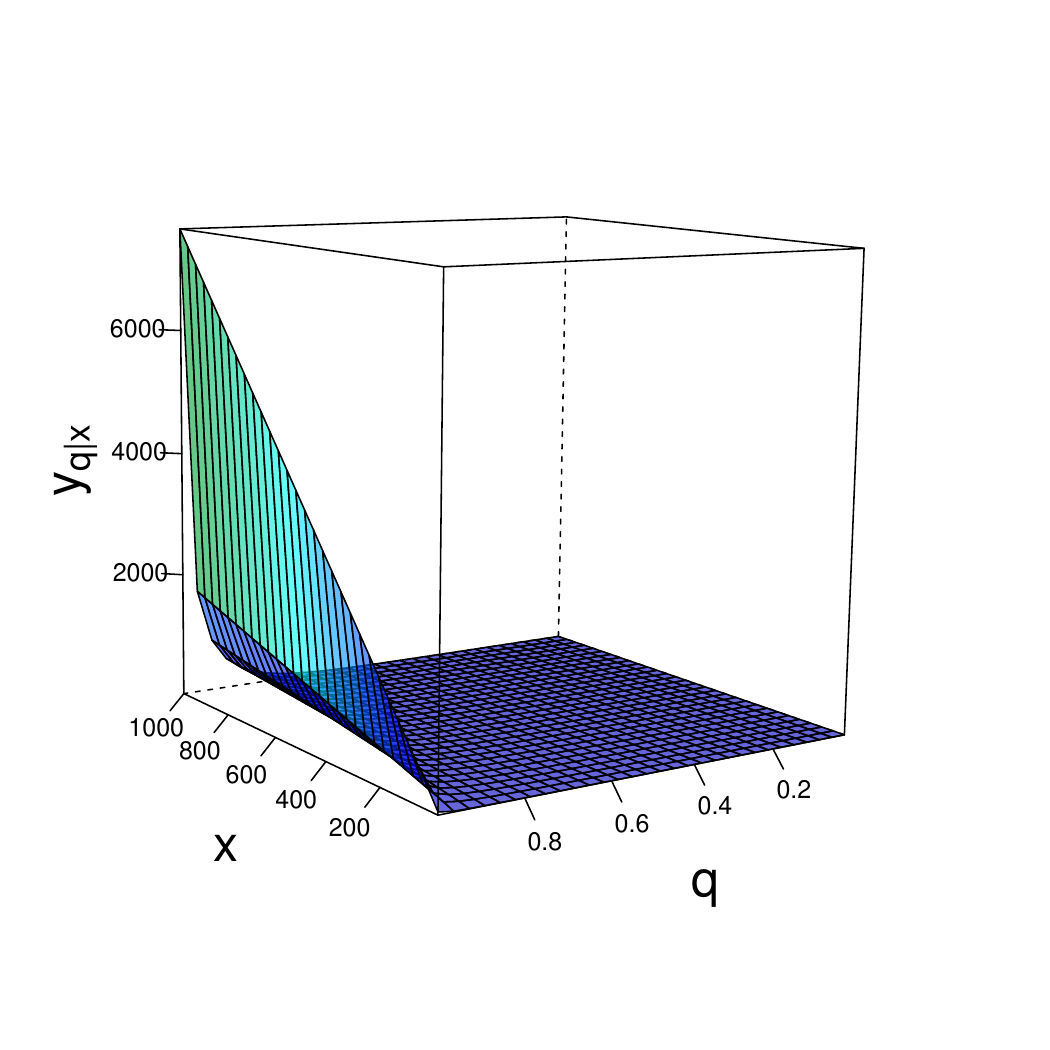}
			\text{\hspace{0cm}}
		\end{minipage} 
		\begin{minipage}[c]{7cm}
			\includegraphics[width=1\textwidth]{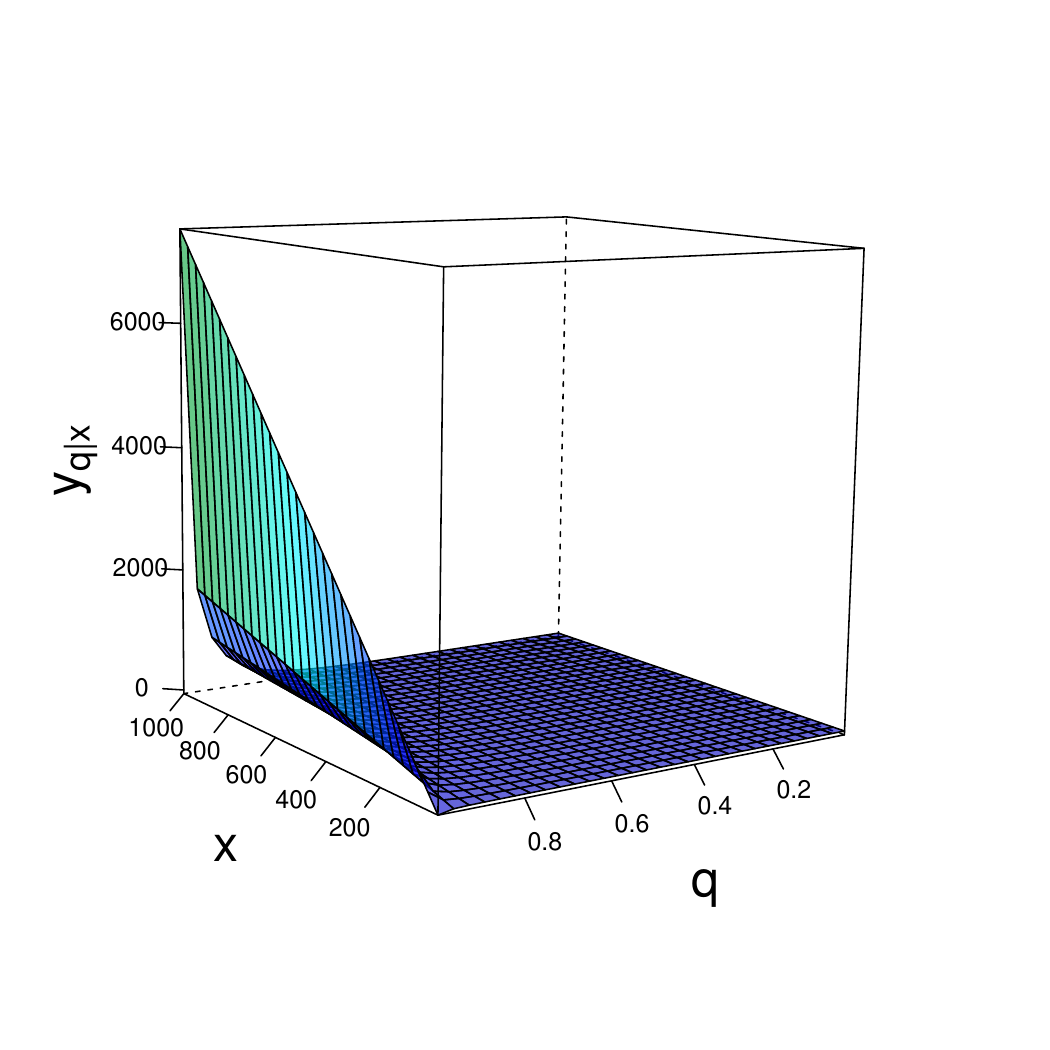}
		\end{minipage}
		\vspace{-1cm}
	\end{center}
	\caption{\label{comparison_1} The true (left) and limiting (right) regression manifold for bivariate logistic model in order of decreasing dependence (from top to bottom) with the dependence parameter $\alpha = \{0.1,0.4,0.6,0.9\}$.}
\end{figure}

\renewcommand{\thefigure}{SM.6}
\begin{figure}
	\centering \hspace{.2cm}
	\begin{minipage}{0.48\linewidth} \hspace{2cm} \footnotesize $\alpha = 0.1$
	\end{minipage}
	\begin{minipage}{0.48\linewidth} \hspace{3.3cm} \footnotesize $\alpha = 0.4$
	\end{minipage}\\ 
	\begin{minipage}{0.48\linewidth}\hspace{-.4cm}
		\includegraphics[scale=0.27]{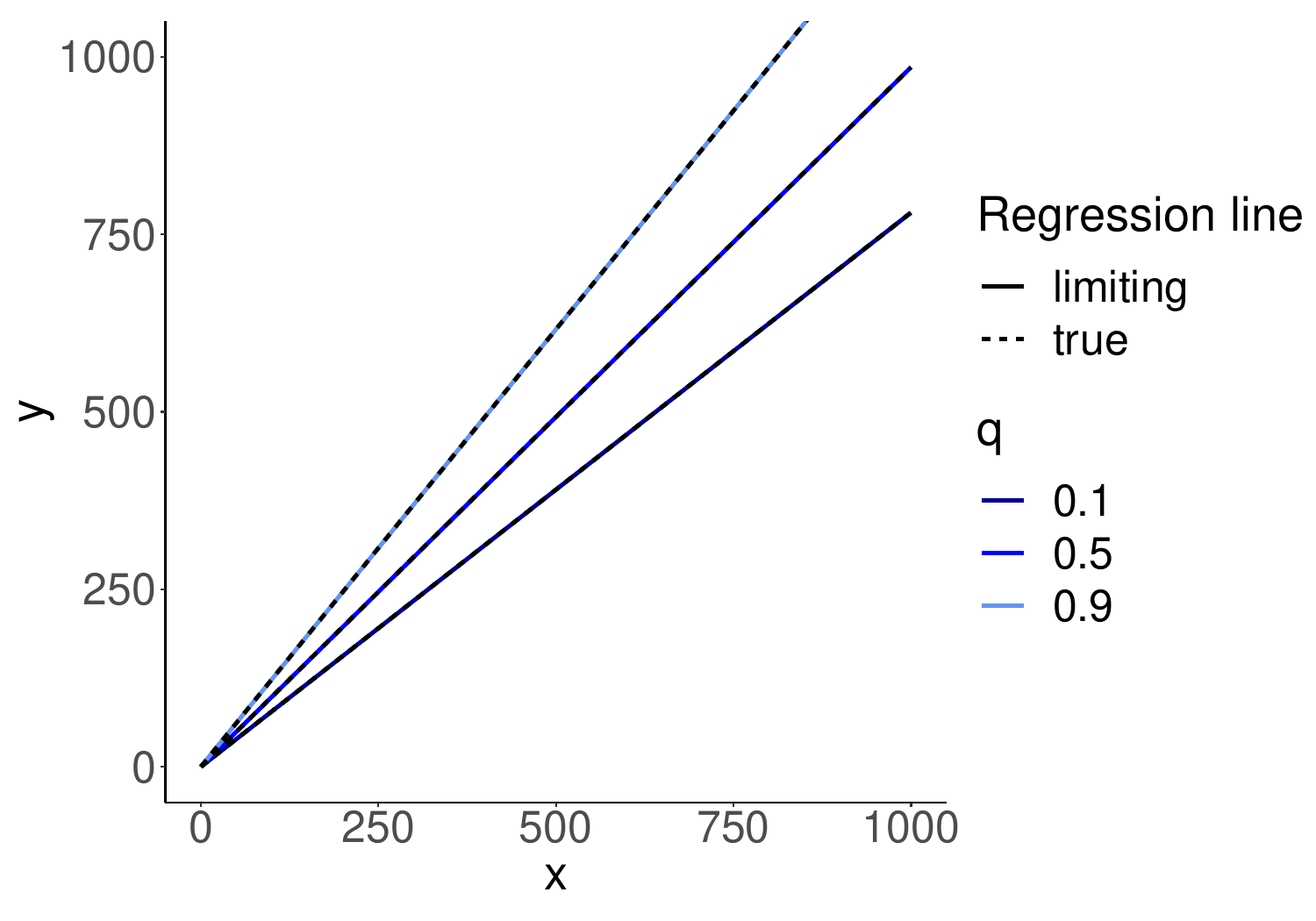}  
	\end{minipage} 
	\begin{minipage}{0.48\linewidth}\hspace{0.6cm}
		\includegraphics[scale = .27]{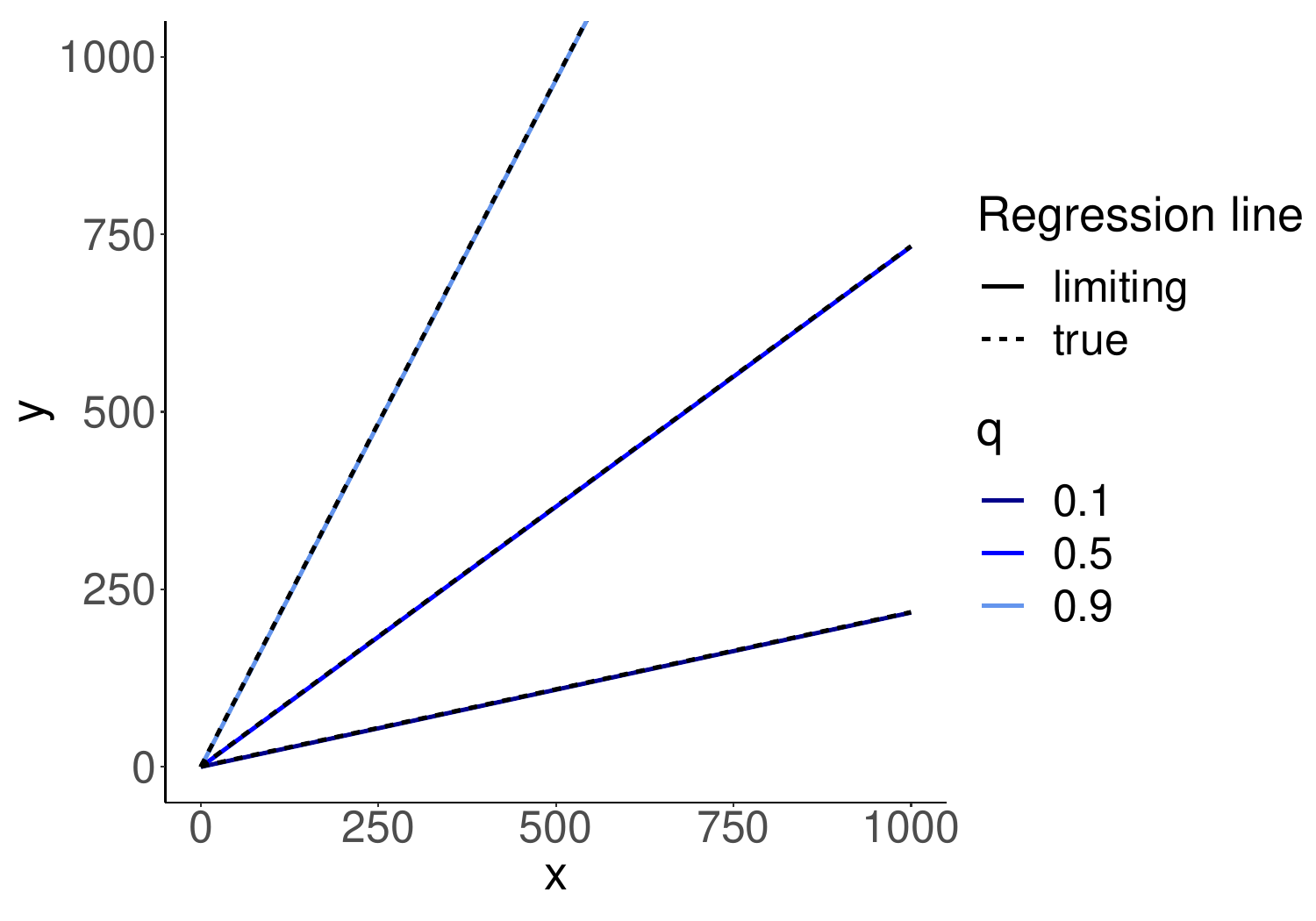}  
	\end{minipage}
	\begin{minipage}{0.48\linewidth} \hspace{2cm} \footnotesize $\alpha = 0.6$
	\end{minipage}
	\begin{minipage}{0.48\linewidth} \hspace{3.3cm}  \footnotesize $\alpha = 0.9$
	\end{minipage}\\ 
	\begin{minipage}{0.48\linewidth} \hspace{-.4cm}    
		\includegraphics[scale=0.27]{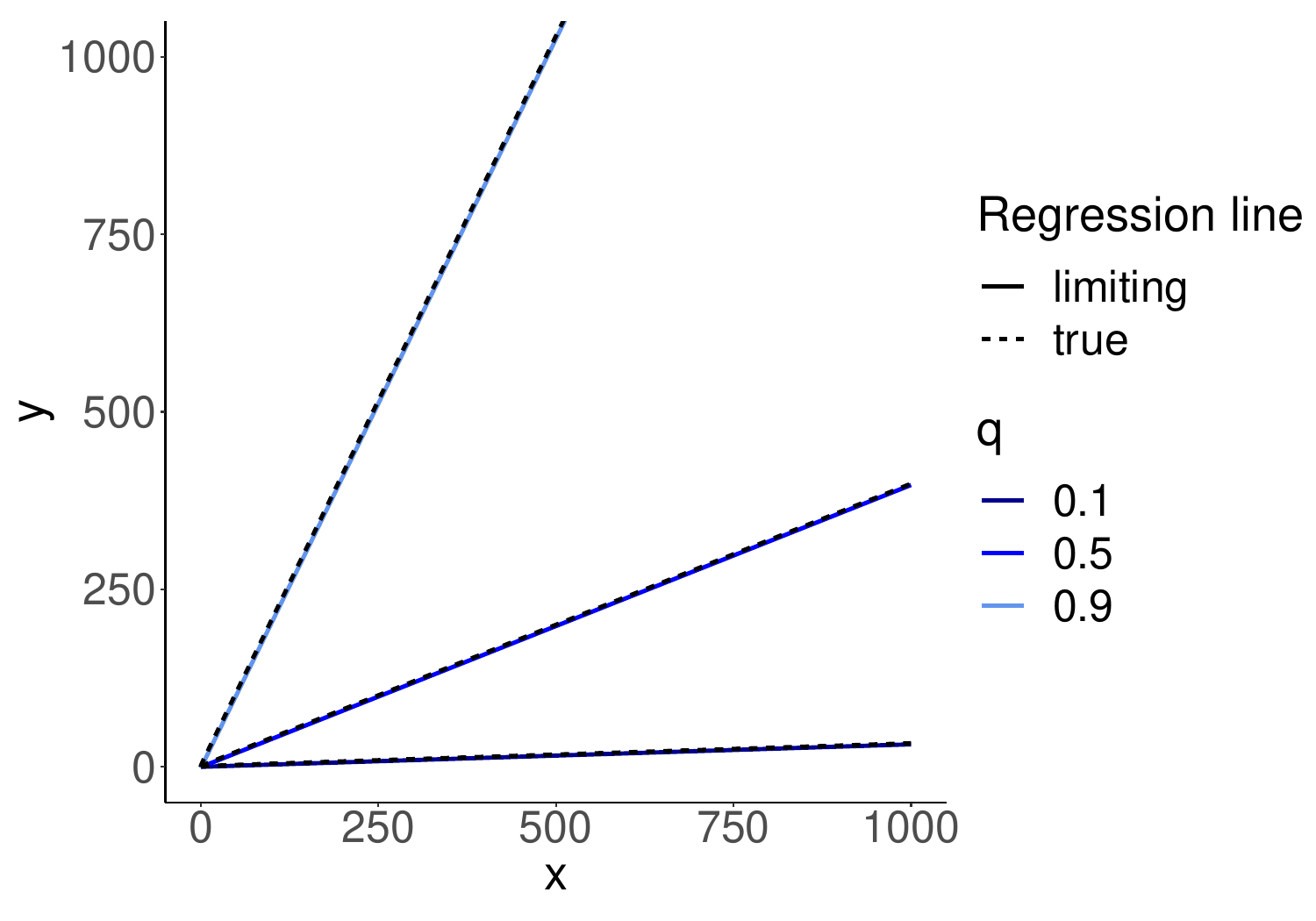} 
	\end{minipage}    
	\begin{minipage}{0.48\linewidth}\hspace{.8cm}    
		\includegraphics[scale=0.27]{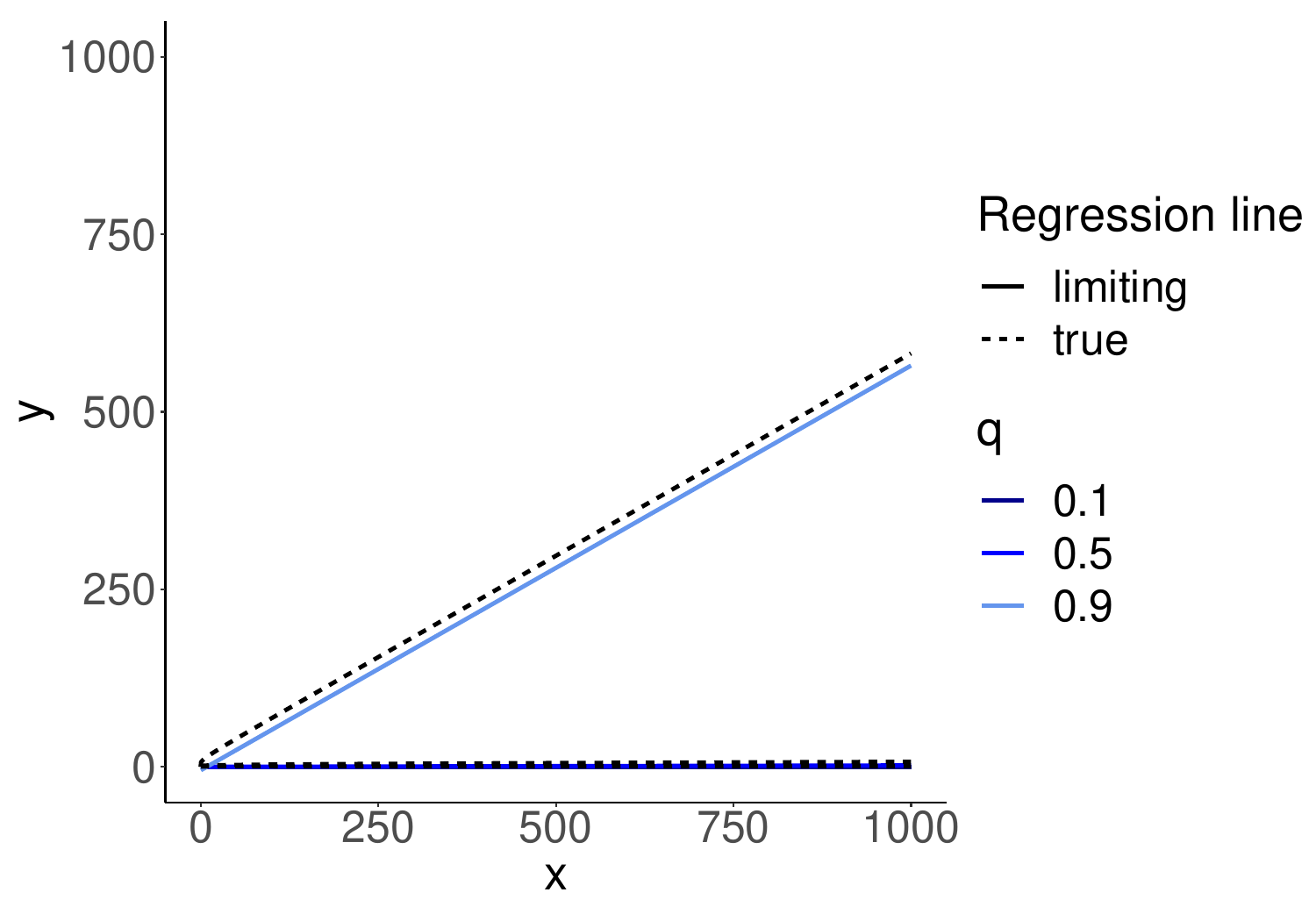}
	\end{minipage}
	\caption{\label{comparison_2} Cross-sections of the true (black dashed line) and limiting (solid line) regression manifold for bivariate logistic model for $q=\{0.1,0.5,0.9\}$.}
\end{figure}


\bibliographystyle{spbasic_updated}
\bibliography{library.bib}

\end{document}